\documentclass[reqno,english]{amsart}

\usepackage{amsmath,amsfonts,amssymb,graphicx,amsthm,enumerate,url}
\usepackage[noadjust]{cite}
\usepackage{stmaryrd}
\usepackage{mathrsfs,booktabs,tabularx}
\usepackage{xifthen,xcolor,tikz,setspace}
\usetikzlibrary{decorations.pathmorphing,patterns,shapes,calc,decorations}
\usetikzlibrary{decorations.pathreplacing}
\usepackage{mathtools,upgreek,comment}
\usepackage[final]{showkeys} 

\definecolor{refkey}{gray}{.75}
\definecolor{labelkey}{gray}{.5}
\usepackage[letterpaper]{geometry}
\usepackage[colorinlistoftodos]{todonotes}
\presetkeys{todonotes}{inline, color=green}{}

\usepackage{accents}
\usepackage[algo2e,boxed,vlined,algoruled]{algorithm2e}

\usepackage[small]{caption}
\usepackage[colorlinks=true]{hyperref}
\colorlet{DarkGreen}{green!50!black}
\colorlet{DarkGray}{gray!60!black}

\numberwithin{equation}{section}

 \definecolor{refkey}{gray}{.5}
 \definecolor{labelkey}{gray}{.5}
\definecolor{light}{gray}{.9}

\newtheorem{theorem}{Theorem}[section]
\newtheorem*{theorem*}{Theorem}
\newtheorem{lemma}[theorem]{Lemma}

\newtheorem{claim}[theorem]{Claim}
\newtheorem{proposition}[theorem]{Proposition}

\newtheorem{observation}[theorem]{Observation}
\newtheorem{fact}[theorem]{Fact}
\newtheorem{corollary}[theorem]{Corollary}

\theoremstyle{definition}{

\newtheorem{definition}[theorem]{Definition}

\newtheorem*{definition*}{Definition}

\newtheorem{remark}[theorem]{Remark}
\newtheorem*{remark*}{Remark}

}

\newcommand{\bin}{\operatorname{Bin}}

\renewcommand{\epsilon}{\varepsilon}
\newcommand{\given}{\;\big|\;}
\newcommand{\one}{\mathbf{1}}

\usepackage{color}
 \definecolor{refkey}{gray}{.5}
 \definecolor{labelkey}{gray}{.5}
\definecolor{light}{gray}{.9}

\newcommand{\E}{\mathbb E}

\renewcommand{\P}{\mathbb P}

\newcommand{\R}{\mathbb R}
\newcommand{\Z}{\mathbb Z}

\newcommand{\cB}{\ensuremath{\mathcal B}}
\newcommand{\cC}{\ensuremath{\mathcal C}}
\newcommand{\cD}{\ensuremath{\mathcal D}}
\newcommand{\cE}{\ensuremath{\mathcal E}}
\newcommand{\cF}{\ensuremath{\mathcal F}}
\newcommand{\cG}{\ensuremath{\mathcal G}}
\newcommand{\cH}{\ensuremath{\mathcal H}}
\newcommand{\cI}{\ensuremath{\mathcal I}}
\newcommand{\cJ}{\ensuremath{\mathcal J}}

\newcommand{\cL}{\ensuremath{\mathcal L}}

\newcommand{\cP}{\ensuremath{\mathcal P}}

\newcommand{\cS}{\ensuremath{\mathcal S}}

\newcommand{\llb }{\llbracket}
\newcommand{\rrb }{\rrbracket}

\newcommand{\sfh}{\ensuremath{\mathsf{h}}}

\newcommand{\fC}{\mathfrak{C}}

\newcommand{\fm}{\mathfrak{m}}

\newcommand{\fG}{\mathfrak{G}}
\newcommand{\fI}{\mathfrak{I}}

\newcommand{\fW}{\mathfrak{W}}

\newcommand{\sC}{{\ensuremath{\mathscr C}}}
\newcommand{\sE}{{\ensuremath{\mathscr E}}}
\newcommand{\sF}{{\ensuremath{\mathscr F}}}

\newcommand{\fD}{\mathfrak{D}}

\newcommand{\Ifloor}[1][\sfh]{\bI^{\mathsf{fl}}_{#1}}

\newcommand{\st}{\ensuremath{\textsc{st}}}

\newcommand{\g}{{\ensuremath{\mathbf g}}}
\newcommand{\br}{{\ensuremath{\mathbf r}}}

\newcommand{\bC}{{\ensuremath{\mathbf C}}}

\newcommand{\bF}{{\ensuremath{\mathbf F}}}

\newcommand{\bH}{{\ensuremath{\mathbf H}}}
\newcommand{\bI}{{\ensuremath{\mathbf I}}}
\newcommand{\bJ}{{\ensuremath{\mathbf J}}}
\newcommand{\bV}{{\ensuremath{\mathbf V}}}
\newcommand{\bW}{{\ensuremath{\mathbf W}}}

 \renewcommand{\epsilon}{\varepsilon}

\DeclareMathOperator{\diam}{diam}

\newcommand{\Clust}{{\mathsf{Clust}}}

\DeclareMathOperator{\hgt}{ht}

\newcommand{\trivincr}{\varnothing}
\newcommand{\hull}[1]{{\mathbullet{#1}}}
\newcommand{\smhull}[1]{{\smash{\hull{#1}}\vphantom{#1}}}

\DeclareMathOperator{\isodim}{\dim_{\sf ip}}

\newcommand{\lowdag}{{\hbox{\raisebox{-4pt}{\scriptsize$\dagger$}}}}

\newcommand{\mufloor}{{\widehat\upmu}}

\makeatletter
\newcommand{\superimpose}[2]{%
  {\ooalign{$#1\@firstoftwo#2$\cr\hfil$#1\@secondoftwo#2$\hfil\cr}}}
  
\newcommand{\sbullet}{%
  \hbox{\fontfamily{lmr}\fontsize{.4\dimexpr(\f@size pt)}{0}\selectfont\textbullet}}
\DeclareRobustCommand{\mathbullet}{\accentset{\sbullet}}

\makeatother

\newcommand{\udarrow}{{\ensuremath{\scriptscriptstyle \updownarrow}}}

\begin{document}

\title[Entropic repulsion of 3D Ising interfaces]{Entropic repulsion of 3D Ising interfaces \\ conditioned to stay above a floor}

\author{Reza Gheissari}
\address{R.\ Gheissari\hfill\break
Department of Mathematics \\ Northwestern University}
\email{gheissari@northwestern.edu}

\author{Eyal Lubetzky}
\address{E.\ Lubetzky\hfill\break
Courant Institute\\ New York University\\
251 Mercer Street\\ New York, NY 10012, USA.}
\email{eyal@courant.nyu.edu}

\vspace{-1cm}

\begin{abstract}
We study the interface of the Ising model in a box of side-length $n$ in $\mathbb Z^3$ at low temperature~$1/\beta$ under Dobrushin's boundary conditions, conditioned to stay in a half-space above height $-\sfh$ (a hard floor). 
Without this conditioning, Dobrushin showed in 1972 that typically most of the interface is flat at height~$0$.
With the floor, for small $\sfh$, the model is expected to exhibit \emph{entropic repulsion}, where the typical height of the interface lifts off of $0$. Detailed understanding of the SOS model---a more tractable height function approximation of 3D Ising---due to Caputo et al., suggests that there is a single integer value $-h_n^* \sim -c\log n$ of the floor height, delineating the transition between rigidity at height $0$ and entropic repulsion. 

We identify an explicit $h_n^*=( c_\star+o(1))\log n$ such that, for the typical Ising interface above a hard floor at~$-\sfh$, all but  an $\epsilon(\beta)$-fraction of  the sites are propelled to be above height~$0$ if  $\sfh < h_n^*-1$, whereas all but an $\epsilon(\beta)$-fraction of the sites remain at height $0$ if  $\sfh\geq h_n^*$.
Further, $c_\star$ is such that the typical height of the unconditional maximum is $(2c_\star + o(1))\log n$; this confirms scaling predictions from the SOS approximation. 
\end{abstract}

 {\mbox{}
 \vspace{-.8cm}
\maketitle
 }
\vspace{-0.9cm}

\section{Introduction}
Let $\mu_n^\mp$ be the distribution of the Ising model with Dobrushin's boundary conditions on the 3D cylinder, 
\[\Lambda_n=  \llb -\tfrac n2,\tfrac n2\rrb ^2 \times \mathbb Z = \{- \lfloor \tfrac n2 \rfloor,\ldots,\lceil\tfrac n2\rceil\}^2 \times \mathbb Z\,.\]
More precisely, let $\mu_n^\mp$ be the distribution over assignments of $\pm 1$ spins to the cells of $\Lambda_n$, denoted $\sC(\Z^3)$. The probability of a configuration $\sigma$ is proportional to $\exp(-\beta \cH(\sigma))$, where $\cH$ counts the number of disagreeing neighboring cells under $\mp$-boundary conditions that are minus in the upper-half space $\{(x_1,x_2,x_3):x_3>0\}$ and plus in the lower-half space. The parameter $\beta$ is an inverse temperature parameter which in our context will be a large enough constant. 
The weak limit of $\mu_n^\mp$ as $n\to\infty$ is denoted $\mu_{\Z^3}^\mp$, which was the famous example of Dobrushin~\cite{Dobrushin72a} for a Gibbs measure on $\Z^3$ which is not translational invariant in every direction. 

For an Ising configuration $\sigma\sim\mu_n^\mp$ with Dobrushin boundary conditions, its interface is defined as follows: consider  every $f$ in $\sF(\Z^3)$ (the faces of $\Z^3$) that separates two disagreeing spins of~$\sigma$, and let the \emph{interface} $\cI$ be the $*$-connected component of such faces incident to the $\{x_3=0\}$ plane outside $\Lambda_n$ (two faces are $*$-connected if they have a common bounding~vertex). Dobrushin~\cite{Dobrushin72a} showed that for $\beta>0$ large enough,  with high probability (w.h.p.) the interface $\cI$ is \emph{rigid} (hence the conclusion on $\mu_{\Z^3}^\mp$) at height $0$; that is, for some fixed $\epsilon_\beta>0$ going to $0$ as $\beta$ goes to infinity,
\[
\big|\cI\cap (\llb- \tfrac{n}{2},\tfrac{n}{2}\rrb^2\times\{0\})\big|
\geq (1-\epsilon_\beta) n^2\,,\qquad\mbox{w.h.p.\ over $\cI\sim\mu_n^\mp$}\,.
\]

In this work we consider the \emph{entropic repulsion} effect of a \emph{hard floor} constraint for $\mu_n^\mp$: namely, let 
\begin{align}\label{eq:fI-h-def}
    \Ifloor = \{\cI \subset \llb- \tfrac{n}{2},\tfrac{n}{2}\rrb^2 \times [-\sfh,\infty)\}
\end{align}
be the event that the interface is everywhere above height $-\sfh$, and define the measure ``with a floor at $-\sfh$,"
\begin{align}\label{eq:mufloor-def}
    \mufloor_n^\sfh = \mu_n^\mp \big( \cdot \mid \Ifloor\big)\quad\mbox{for}\quad \sfh\geq 0\,.
\end{align}
The hard constraint of confining $\cI$ to a half-space above height $-\sfh$ creates an energy vs.\ entropy competition: wherever the interface is to remain flat at height $0$,  its downward oscillations would be capped at depth~$\sfh$. Indeed, in several other more tractable models, a conditioning of this sort was shown to propel the interface above height $0$ (see~\S\ref{subsec:related-work} for related work). Of course, one expects this to hold only for small values of $\sfh$, whereas for large enough $\sfh$, the effect of a hard floor at height $-\sfh$ should be unnoticeable, and the interface should remain flat at height $0$ (this would certainly be the case if one should take $\sfh$ to be larger than the height of the global maximum in a typical interface $\cI\sim\mu_n^\mp$).

Our main result, Theorem~\ref{thm:1}, identifies the exact $\sfh$ delineating this phase transition: for any $\beta$ large enough, we identify a critical integer $h_n^*$ such that, for all $\sfh<h_n^*$,
 a typical interface $\cI \sim \mufloor_n^\sfh$ would have (say) at most a $0.01$ fraction of the faces of $\cI$ be at height $0$, whereas for $\sfh> h_n^*$ at least (say) $0.99$ of the points would be such.
This height is defined by 
\begin{equation}\label{eq:h-star-def}
h_n^* = h_n^* (\beta) : = \inf\{ h\geq 1 \,:\; \alpha_h > \log n - 2\beta\}\,,
\end{equation}
in which the quantities $\alpha_h=\alpha_h(\beta)$ are given by
\begin{equation}\label{eq:alpha-h-def}
\alpha_h = -\log \mu_{\Z^3}^\mp\Big((\tfrac12,\tfrac12,\tfrac12) \xleftrightarrow[\R^2\times [0,\infty)]{+} (\Z+\tfrac 12)^2\times\{h-\tfrac12\}\Big)\,,
\end{equation}
where  $v\xleftrightarrow[A~]{+}w$ denotes there is a path of adjacent or diagonally adjacent plus spins between $v$ and $w$ in~$A$.

\begin{theorem}\label{thm:1}
There exist $\beta_0,c_0>0$ so that, if $\epsilon_\beta = 1/(c_0 e^{\beta})$ and $h_n^*$ is as in~\eqref{eq:h-star-def}, then for every $\beta>\beta_0$, with probability at least $1-\exp(-\epsilon_\beta n)$ the interface $\cI\sim\mufloor_n^\sfh$ satisfies 
\begin{align}
\label{eq:A-subcritical-h}
\big|\cI \cap (\llb- \tfrac{n}{2},\tfrac{n}{2}\rrb^2\times\{0\})\big| &< \epsilon_\beta \, n^2   \qquad\qquad \mbox{if $\sfh < h_n^*-1$}\,, \\
\label{eq:A-supercritical-h} 
 \big|\cI\cap (\llb- \tfrac{n}{2},\tfrac{n}{2}\rrb^2\times\{0\})\big| &> (1-\epsilon_\beta)n^2 \qquad \mbox{if $\sfh \geq h_n^*$}\,.
\end{align}
Furthermore,~\eqref{eq:A-subcritical-h} also holds if we replace $\llb- \tfrac{n}{2},\tfrac{n}{2}\rrb^2\times\{0\}$  by $\llb- \tfrac{n}{2},\tfrac{n}{2}\rrb^2 \times (-\infty,h_n^*-\sfh-1)$.
\end{theorem}

While one may easily be convinced that the hard floor constraint in $\mufloor_n^\sfh$ will become unnoticeable once $\sfh$ exceeds the typical value of $M_n$, the maximum height of the unconstrained interface $\cI\sim\mu_n^\mp$ (which, by reflection symmetry, is the same as the typical value of the minimum height, whence $\cI$ will stay in $\llb-\tfrac n2,\tfrac n2\rrb^2\times[-\sfh,\infty)$ w.h.p.), in fact $h_n^*$ is asymptotically $1/2$ of that value, as the next remark states. 
\begin{remark}\label{rem:relation-to-maximum}
In~\cite{GL19a}, the authors derived a law of large numbers for the maximum height $M_n$ of~$\cI\sim\mu_n^\mp$: as $n\to\infty$ one has $M_n / \log n \to 2/\alpha$ in $\mu_n^\mp$-probability, for the quantity $\alpha=\alpha(\beta)$ defined there as
\begin{equation}\label{eq:alpha-def}
\alpha := \lim_{h\to\infty} \frac{\alpha_h}h \in[4\beta-C,4\beta+e^{-4\beta}]\,,
\end{equation}
with the $\alpha_h$'s as given in~\eqref{eq:alpha-h-def} (the tightness, and moreover Gumbel tails, of $M_n-\E[M_n]$ were proved in~\cite{GL19b}).
Theorem~\ref{thm:1} shows that the critical $\sfh$ for repulsion in $\mufloor_n^\sfh$ is $h_n^*=(\frac{1}{\alpha}+o(1))\log n$, which is $\frac{1+o(1)}2 \E[M_n]$ for $\cI\sim\mu_n^\mp$, in line with known scaling results for the SOS approximation of the 3D Ising model (see~\S\ref{subsec:related-work}).
\end{remark}

Akin to the behavior of the SOS model, the case $\sfh=h_n^*-1$ can depend on $n$, as the next remark details.
\begin{remark}\label{rem:h=h_n-1}
Theorem~\ref{thm:1} extends to treat $\sfh=h_n^*-1$ whenever the determinstic quantity $\lambda_n := \log n - \alpha_{h_n^*}$, which is known to belong to $[-2\beta-\epsilon_\beta,2\beta]$ for all $n$ (the upper bound is by definition~\eqref{eq:h-star-def} whereas the lower bound is by results in~\cite{GL19a}) does not fall in a certain $\epsilon_\beta$-fraction of this interval. For instance,~\eqref{eq:A-subcritical-h} extends to $\sfh=h_n^*-1$ for $\lambda_n \geq 2\log\beta$, while~\eqref{eq:A-supercritical-h} extends to $\sfh=h_n^*-1$ for $\lambda_n \leq \log\beta$ (see Remarks~\ref{rem:h=hn-1-lower} and~\ref{rem:h=hn-1-upper}). 
\end{remark}

\begin{figure}
\includegraphics[width=0.73\textwidth]{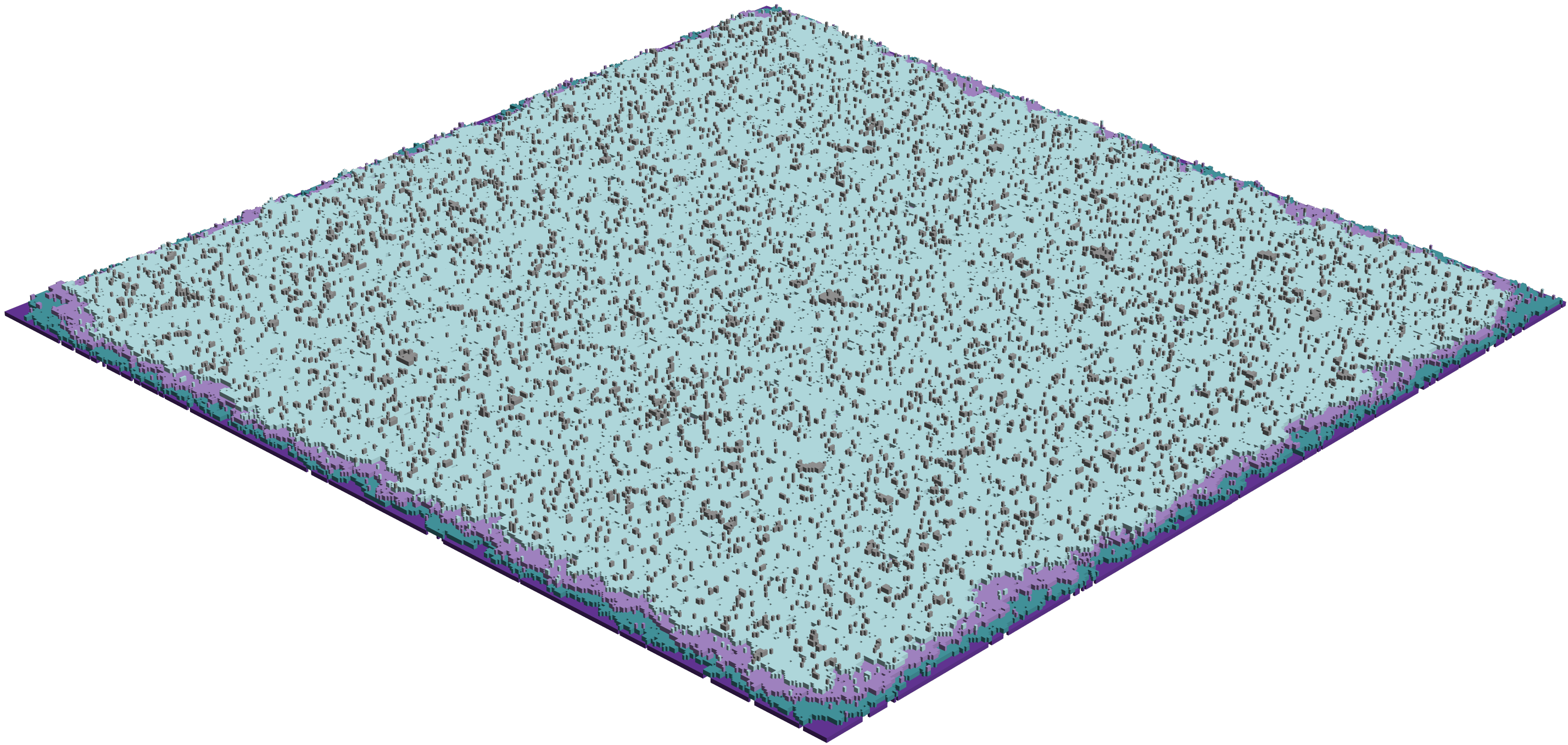}
\hspace{15pt}
\hspace{-0.25in}\raisebox{0.4in}{\includegraphics[width=0.25\textwidth]{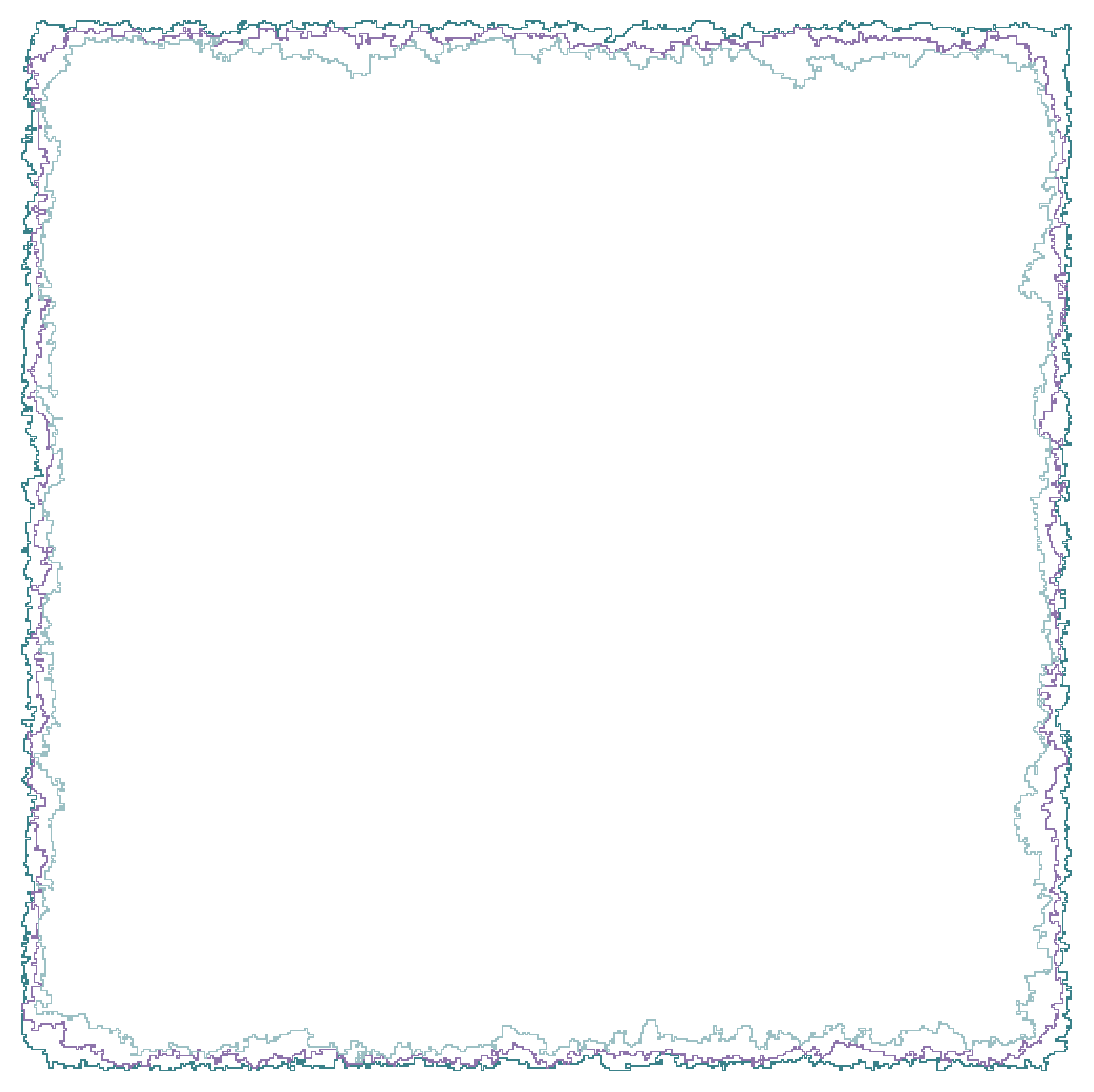}}
\vspace{-0.1in}
\caption{An Ising interface in $\Lambda_{n}$ for $n=500$ conditioned to lie above a floor at height $\sfh=0$. 
}
\vspace{-0.1in}
\end{figure}

In the special case $\sfh=0$, Theorem~\ref{thm:1}, together with a short and self-contained result (Claim~\ref{clm:upper-bound-1+eps}), translates to the following: with high probability, $\cI\sim \mufloor_n^0$ has
    \begin{align}\label{eq:h=0}
        |\cI\cap (\llb - \tfrac{n}{2},\tfrac{n}{2}\rrb^2 \times \llb h_n^* - 1, (1+\epsilon_\beta) h_n^*\rrb)| >(1-\epsilon_\beta)n^2\,.
    \end{align}

In view of Theorem~\ref{thm:1}, we conjecture the lower bound here is exact (once the interface has reached a height $h_n^*$ above the floor, it has no further incentive to rise), so that the upper bound $(1+\epsilon_\beta)h_n^*$ in~\eqref{eq:h=0} can be replaced by $h_n^*$. More generally, in Theorem~\ref{thm:1}, we conjecture that for every $\sfh$, the interface $\cI\sim \mufloor_n^\sfh$ has  $(1-\epsilon_\beta)n^2$ of its faces be at height either $h_n^*-\sfh-1$ or $ h_n^* - \sfh$.

\begin{remark}
While our results are stated for the cylinder $\Lambda_n= \llb -\tfrac n2,\tfrac n2\rrb ^2 \times \mathbb Z $, all proofs go through mutatis mutandis on the box $\Lambda_n= \llb -\tfrac n2,\tfrac n2\rrb ^3$ as the interface never feels the effect of the boundary conditions on the top and bottom, except with $e^{ - cn}$ probability.
\end{remark}

\begin{remark}\label{rem:extension-general-d}
Our results extend to any dimension $d\geq 3$ and are presented for $d=3$ to simplify the exposition. Theorem~\ref{thm:1} would be given analogously, with $n^2$ replaced by $n^{d-1}$, and with $h_n^*$ defined as in~\eqref{eq:h-star-def} with $2\beta$ replaced by $(d-1)\beta$ and with the natural generalization of $\alpha_h(d)$. Pertaining Remark~\ref{rem:relation-to-maximum}, when $d>3$, one has $M_n/\log n \xrightarrow[]{\;\textrm p\;} (d-1)/\alpha(d)$, whence the ratio $\E[M_n]/h_n^*$ goes to $d-1$ as $n\to\infty$. 
\end{remark}

\subsection{Related work}\label{subsec:related-work}
The study of entropic repulsion for interfaces separating stable phases of the Ising model has a long history. We first describe the picture in $\mathbb Z^2$. Without any floor (i.e., under $\mu_n^\mp$ as opposed to $\mufloor_n^\sfh$), the Ising interface in a strip of side-length $n$ with Dobrushin's boundary conditions is known to have $\sqrt{n}$ height fluctuations~\cite{DKS}, and converge to a Brownian bridge under a diffusive rescaling at all $\beta>\beta_c(d)$~\cite{Hryniv98,GreenbergIoffe}. In the presence of a floor at height zero, either by means of plus boundary conditions in the entire lower half-space or by means of conditioning on the interface being restricted to the upper half-space, one could easily deduce that the fluctuations remain $O(\sqrt{n})$; convergence to a Brownian excursion was recently shown in~\cite{IOVW-convergence-to-Brownian-excursion}. In particular, no matter the conditioning on a floor (at whatever negative height) the entropic repulsion effect does not change the order of the typical height of the interface. We also mention that in this two-dimensional setting,~\cite{IST-interaction-entropic-repulsion} established an equivalence between the floor effects induced by conditioning on the interface to be non-negative, and by plus boundary conditions in the lower half-space. 

In dimensions $d\ge 3$, the phenomenology is completely different. Recall that without any floor conditioning, the interface is rigid about its ground state, i.e., the interface is mostly flat at height zero, with $O(1)$ height oscillations about that with exponential tails~\cite{Dobrushin72a,vanBeijeren75}. In~\cite{GL19a,GL19b} the authors studied more refined features regarding the typical shape of this interface near its high points, and used that to deduce that the maximum height is tight and has uniform Gumbel tails about its median $m_n^*= (\frac{2}{\alpha}+o(1))\log n$. 

When introducing a hard floor, either via conditioning or via a set of plus spins in the lower half-space, far less is known, though this problem has been discussed in the physics literature at least since~\cite{BEF86}. That work proposed using
the solid-on-solid (SOS) model~\cite{Abraham}, a distribution over $\varphi: \llb -\tfrac{n}{2},\tfrac{n}{2}\rrb^2\to\mathbb Z$ with Hamiltonian, $\sum_{x\sim y} |\varphi_x - \varphi_y|$ 
as an approximation to the Ising interface on which such questions can be studied. 
The work~\cite{BEF86} identified the order of the typical height the SOS  surface rises to due to the entropic repulsion effect at sufficiently low temperatures, and in~\cite{CLMST12,CLMST14} the exact asymptotics of this height were determined. In particular, it was found in~\cite{CLMST12,CLMST14} that in the SOS model, the typical height of the surface above a floor at height zero, is exactly half of the maximum height in the absence of a floor (note this matches the scaling relation between $h_n^*$ and $m_n^*$ we found for the Ising model: see Remark~\ref{rem:relation-to-maximum}). Beyond that,~\cite{CLMST16} established that the macroscopic level lines converge to the Wulff shape, and have $n^{1/3+o(1)}$ fluctuations along the sides of the box. The work~\cite{LMS16} generalized the results of~\cite{CLMST12,CLMST14} of this entropic repulsion phenomenon to integer valued $|\nabla \varphi|^p$ interfaces (the SOS model being the $p=1$ case). Let us also mention that the question of entropic repulsion has been extensively studied in the context of the discrete Gaussian free field~\cite{BDG01,BDZ95,Deuschel96,DeuschelGiacomin99}. 
In all such height function approximations, identification of the height to which the interface rises given a floor at height zero is effectively equivalent to identification of the floor height $-\sfh$ at which the phase transition in the form of Theorem~\ref{thm:1} occurs. 

Despite this progress for understanding the entropic repulsion phenomenon on the SOS approximation, for the actual Ising interface, much less was known. A non-quantitative delocalization result that the typical height of the interface goes to infinity when the floor is at height zero was established by means of correlation inequalities in~\cite{FP87a,FP87b}. The works~\cite{melichervcik1991entropic,HZ93} sketched an argument that this interface rises to a height that is between $(c/\beta)\log n$ and $(C/\beta)\log n$ for two different constants $C>c>0$. However, we are not aware of any previous work on the 3D Ising model pinpointing the location of an entropic repulsion phenomenon to within $o(\log n)$, let alone identifying its exact height.

Understanding the entropic repulsion phenomenon also opens the door to studying \emph{wetting phenomena} for interfaces between phases: see~\cite{BEF86} as well as the surveys~\cite{Velenik06,IV18} for details on the below discussion. Here, one can for instance change the interaction strength for the edges along the floor from $1$ to some $J\ge 0$; the interface then undergoes a transition at some $J_w(\beta)\le 1$ between a \emph{partial wetting} regime ($J< J_w$) where the typical interface height is $O(1)$, and a \emph{complete wetting} regime ($J> J_w$) in which the typical interface height diverges. In $d=2$, this transition can be deduced for the Ising model by means of the machinery of~\cite{PV97,PV99,CIV03}. In $d\ge 3$, wetting for the SOS model was first studied in~\cite{Chalker}, and $J_w(\beta)$ was precisely identified in~\cite{Lacoin-wetting}; for the Ising model, some partial results can be obtained, as in~\cite{FP87b,Basuev,ADM-layering-Ising}. A related physical phenomenon occurs when instead keeping $J=1$, but introducing an external field $-\lambda<0$ which pushes the interface down, competing with the repulsion away from the floor. This competition induces a physical phenomenon known as \emph{pre-wetting}, whereby as $\lambda = \lambda_n \downarrow 0$, the interface interpolates between the delocalized one at $\lambda=0$ and one having uniformly bounded height oscillations when $\lambda$ is uniformly bounded away from zero. In $d= 2$, this interpolation is smooth, and the interface behavior is now well-understood, exhibiting KPZ fluctuation exponents and local convergence to the Ferrari--Spohn diffusion~\cite{Velenik04,GaGh20,IOSV21}. In the $d\ge 3$ case, less is known rigorously, aside from the first few jumps of the interface for $\lambda$ small but independent of the system size~\cite{Basuev}.

\subsection{Proof ideas}\label{subsec:proof-ideas}
The proof of Theorem~\ref{thm:1} proceeds in two parts corresponding to each of~\eqref{eq:A-subcritical-h} and~\eqref{eq:A-supercritical-h}.
A more detailed formulation of~\eqref{eq:A-subcritical-h} can be found in Theorem~\ref{thm:lower-gen-h}, and a more detailed formulation of~\eqref{eq:A-supercritical-h} can be found in Theorem~\ref{thm:upper-gen-h}. Unlike~\cite{Dobrushin72a,GL19a,GL19b,GL20} the proofs in this paper are primarily probabilistic in nature, with all combinatorial objects we use having already appeared in the previous work~\cite{GL20}.

\subsubsection{A priori regularity estimates}
The first step in our analysis consists of establishing that, typically,
\begin{enumerate}[(i)]
	\item all but an~$\epsilon_\beta$-fraction of the faces of~$\cI$ are horizontal faces for which the intersection of $\cI$ with the column through them is a singleton---call these \emph{ceiling faces} following Dobrushin~\cite{Dobrushin72a}; and
	\item  all but an $\epsilon_\beta$-fraction of the ceiling faces belong to connected components---referred to as \emph{ceilings}---that are reasonably large and regular, being of size at least $n^{1.9}$ and with a boundary of size at most $n^{1+o(1)}$. 
\end{enumerate} 
In light of these facts,  we may restrict our attention to the heights of faces of $\cI$ belonging to  ceilings as in~(ii).
This is crucial because the estimates we import from~\cite{GL20} on large deviations of the maximal height oscillations interior to a ceiling are only sharp inside ceilings with a uniformly bounded isoperimetric dimension.

Such a priori estimates are established using the basic fact that, if $\Gamma \subset \Ifloor$ is such that $\mu_n^\mp(\Gamma) \ll \mu_n^\mp(\Ifloor)$, then $\Gamma$ is also atypical under the conditional measure $\mufloor_n^\sfh=\mu_n^\mp(\cdot\mid\Ifloor)$. Concretely, the lower bound on $\mu_n^\mp(\Ifloor[0])$ given in the next proposition allows us to rule out any event $\Gamma$ which, e.g., has $\mu_n^\mp(\Gamma)\leq \exp(-1.1 n\log n)$.

\begin{proposition}
\label{prop:prob-positive-upper-lower}
There exists $\beta_0>0$ such that, for every fixed $\beta>\beta_0$ and every sufficiently large~$n$, the probability of the event $\Ifloor[0]$, i.e.\ that the 3D Ising interface $\cI$ is a subset of $\llb-\tfrac n2,\tfrac n2\rrb^2 \times \R_+$, satisfies
\[e^{-(1+\epsilon_\beta)n\log n} \leq \mu_n^\mp\left(\Ifloor[0]\right) \leq e^{-(1-\epsilon_\beta)n\log n}\,,\]
where $\epsilon_\beta$ is some sequence going to $0$ as $\beta$ increases.
\end{proposition}

The lower bound in this proposition is relatively straightforward, and we establish it in~\S\ref{sec:rough-bounds-no-floor} en route to deriving the mentioned preliminary regularity estimates. The upper bound is more delicate, and while not needed for our arguments, in~\S\ref{sec:pf-prop-positive} we obtain it as a corollary of our proof of~\eqref{eq:A-subcritical-h} (namely, from the more detailed result in Theorem~\ref{thm:lower-bound}, which rules out the existence of large ceilings below height $h_n^*-\sfh-1$ under $\mufloor_n^\mp$).

\subsubsection{Idea of proof of~\eqref{eq:A-subcritical-h}}
We begin with a natural approach for ruling out the existence of a ceiling $\cC$ (a connected component of (horizontal) ceiling faces of $\cI$), whose height $\hgt(\cC)$ is $h_n^*-\sfh-1-k$ for $k\geq1$ and whose projection on height $0$, together with every finite component it bounds (in these components the interface can exhibit local oscillations relative to $\cC$) is a set $S$ with area $|S|\geq \theta^{k} n^2$ for some~$0<\theta(\beta)<1$.
(The aim would be to show that such a $\cC$ is atypical under $\mufloor_n^\sfh$,  and then take a union bound over $k$.)
Denote the maximal downward oscillation inside $\cC$ by $\bar M_S^\downarrow = \hgt(\cC) - \min\{x_3: (x_1,x_2,x_3)\in\cI\,,\,(x_1,x_2)\in S\}$, and suppose for the moment that the following bound holds: 
$\mu_n^\mp(\bar M^\downarrow_S < h) \lesssim \exp(-|S|e^{-\alpha_h})$ for any $h$,
with~$\alpha_h$ from~\eqref{eq:alpha-h-def}. (The actual bound available to us---see~\eqref{eq:max-thick-upper-large-h}---has a few small differences (error terms, a restriction on~$h$) in addition to one major proviso we will soon describe.)
This estimate is quite intuitive in light of the fact (\cite{GL19a,GL19b}) that the $\mu_n^\mp$-probability of a local oscillation of height at least $h$ below a site in the bulk is approximately $e^{-\alpha_h}$; indeed, if the oscillations in every site of $S$ were mutually independent, we would get the  bound $(1-e^{-\alpha_h})^{|S|} \approx \exp(-|S|e^{-\alpha_h})$. 
In the event that a ceiling $\cC$ as described above exists, we are guaranteed to have $\bar M^\downarrow_S < h_n^*-k$ by the implicit conditioning on $\Ifloor$ in $\mufloor_n^\sfh$. To rule this out, we shift the interface up by $k$, effectively sending $\cC$ to height $h_n^*-\sfh-1$, which reduces the weight of the interface by about $\exp(-4(\beta-C) k n)$ due to the necessary addition of $4kn$ faces to the interface. The benefit of doing so is that, even though we can still only say that $\bar M^\downarrow_S < h_n^* - k$ (the local oscillations within $\cC$ are unaffected by the shift), the shifted interface is only guaranteed to have $\bar M^\downarrow_S < h_n^*$ by the implicit conditioning in $\mufloor_n^\sfh$, at which point one could utilize the fact that $\mu_n^\mp(\bar M^\downarrow_S < h_n^*-k \mid \bar M^\downarrow_S< h_n^*) \approx \mu_n^\mp(\bar M^\downarrow_S < h_n^*-k) \lesssim \exp(-|S|e^{-\alpha_{h_n^*-k}})$
(where one also uses a sharp \emph{lower} bound on the denominator in the ratio $\mu_n^\mp(\bar M^\downarrow_S < h_n^*-k) / \mu_n^\mp(\bar M^\downarrow_S < h_n^*)$).
In summary, the probability of encountering such a ceiling $\cC$ can be bounded from above by approximately
\[
  \exp\left( 4(\beta-C) k n  -|S| e^{ - \alpha_{h_n^* - k}}\right)\,.
\]
The sequence $(\alpha_h)$ is known~\cite{GL19b} to satisfy $\alpha_{h_n^*-k} \leq \alpha_{h_n^*} - (4\beta-C)k$, whereas $e^{-\alpha_{h_n^*}} \leq e^{2\beta}/n$ by the definition of $h_n^*$ in~\eqref{eq:h-star-def}. Plugging these in the bound above, together with the fact that $|S|\geq \theta^k n^2$, yields
\[\exp\left(4 (\beta-C) k n -  e^{-2\beta} (\theta e^{4\beta-C})^k n\right)\,,\] 
in which the second term in the exponent dominates the first term for $\theta:=e^{-\beta}$ when $\beta$ is large enough.

This approach highlights the competition between energy and entropy (the terms $4\beta kn$ and $|S|e^{-\alpha_{h_n^*-k}}$) propelling the interface from being rigid at height $0$. It is worthwhile comparing this argument to the simpler analysis of entropic repulsion in the SOS model: there, to rule out a large subset $\cC$ at height less than $h$, one can lift the configuration by $k$, and then plant downward spikes in an $\epsilon_\beta$-subset of $\cC$. In the 3D Ising model, the large deviation rate of the height above/below a site in the bulk is not governed by a deterministic shape (such as the single column spike in the SOS model), but rather by a distribution over complex oscillations; hence, we instead consider interfaces with $\bar M_S^\downarrow < h_n^*-k$, where no oscillation dipped below a certain height.

One obvious difficulty in carrying out the above approach is due to having, in lieu of the simplified case of one given ceiling~$\cC$, an abundance of randomly located ceilings. One needs to reveal (here and in what follows, we use the term reveal or expose to refer to conditioning on the realizations of the objects under consideration) their boundaries, but not their interior, in order to appeal to the bounds of~\cite{GL20}, and do so with care, as the boundary of one ceiling may overlap with that of another. Moreover, a single \emph{wall}---a connected component of non-ceiling faces---may give rise to multiple ceilings at different heights. However, these details can be handled by exposing the outermost walls and reevaluating the ceiling landscape given the exposed walls (and the ceilings they nest): we may identify a slab at height $h_n^*-\sfh-1-k$ where the total area in the ceilings $\cC_i$ exceeds $\theta^k n^2$, and attempt to apply the above argument to each of the $S_i$'s separately, for a bound of $\exp(-\sum |S_i|e^{-\alpha_{h_n^*-k}})$.
 
The main obstacle, as hinted above, is actually the upper bound we have for $\mu_n^\mp(\bar M_S^\downarrow < h)$, which is only valid in conjunction with an indicator that every wall in the bulk of $S$ must be relatively small---namely, the diameter of every such wall should be at most $\exp(c h)\lesssim n^{\epsilon_\beta}$ (see Proposition~\ref{prop:max-thick} and the event $\cG^\fD_{S^\circ}$ in the key estimate~\eqref{eq:max-thick-upper-large-h} in it). The source of this constraint is that, in order to prove said bound in~\cite{GL20}, the region $S$ was partitioned into boxes $B_i$ of side length $L$, whose interiors $B_i^\circ$ were inspected carefully (while ignoring $B_i\setminus B_i^\circ$); the constraint on the maximum diameter of any wall meant that no wall in $B_i^\circ$  could reach~$\partial B_i$, leading to the desired independence of local oscillations between $B_i^\circ$'s. We stress that this obstacle is real, rather than a limitation of the proof technique: a single wall of size $(c/\beta) n\log n$ can raise/lower an area of order $n^2$ by height $(c/\beta)\log n$, forming correlations between the heights in its interior as it nests sites in ceilings. The cost of such a wall is $\exp(-c'n\log n)$, which can still be larger than $\mu_n^\mp(\Ifloor)$, and indeed, this is precisely the mechanism of entropic repulsion, where most of the interface is raised by a single large wall...

To overcome this obstacle, we \emph{reveal every mesoscopic wall}---in the sense of having diameter larger than a threshold of $e^{c(h_n^*-k)}\vee \log n$; see~\eqref{eq:meso-wall}---in $\cI$ and carry the above approach in this conditional space. This is achieved in Lemma~\ref{lem:area-of-ceilings-I-dagger}, the key to showing~\eqref{eq:A-subcritical-h}. The following are worthwhile mentioning from its proof:
\begin{enumerate}[(1)]
\item Conditioning on all mesoscopic walls can reveal a complicated set of walls nesting one another, supporting a  collection of ceilings $\cC_i$ (the full interface $\cI$ will modify these $\cC_i$'s via nested microscopic walls). Whereas we would like to appeal to Proposition~\ref{prop:max-thick} simultaneously for every $\cC_i$ which is at height $h_n^*-\sfh-1-k$, there are two prerequisites on its interior $S_i$ to qualify an application of that proposition: (i) $S_i$ should be simply-connected;  (ii) the isoperimetric dimension of $S_i$ should be bounded by an absolute constant. Item~(ii) is handled by our a priori regularity estimates, as we only treat $S_i$'s with area at least $n^{1.9}$ and boundary at most $n^{1+o(1)}$. However,  Item~(i) is simply false when one mesoscopic wall nests another...  Our remedy for this is to further expose the walls along a minimal collection of faces that connect the boundaries of all $S_i$'s, which will make the unrevealed portion of the $S_i$'s simply-connected by definition. 
\item Controlling the set of faces that connect the $S_i$'s is imperative: too large of a set may potentially deform the regularity of the $S_i$'s, increasing their isoperimetric dimension beyond the scope of Proposition~\ref{prop:max-thick}. Here we rely on the fact that there is a total of $n^{1+o(1)}$ faces in the union of the boundaries of the~$S_i$'s.
A classical bound on the \emph{Euclidean Traveling Salesman Problem} (TSP) then shows that we can connect an arbitrary face from each one via a minimum set of size  at most $n^{3/2+o(1)}$ faces. Thus, revealing the (microscopic, by the conditioning) walls along said set modifies the area and boundary of each $S_i$ by at most $n^{3/2+o(1)}$, keeping its isoperimetric dimension (we had $|S_i|\geq n^{1.9}$ and $|\partial S_i|\leq n^{1+o(1)}$) in check.
\item  Recall that in the simplified outline one needed a sharp upper bound for $\mu_n^\mp(\bar M^\downarrow_S < h_n^*-k \mid \bar M^\downarrow_S< h_n^*) $.  In the actual proof, we must instead estimate
$\mu_n^\mp(\bar M^\downarrow_S < h_n^*-k \mid \bar M^\downarrow_S< h_n^*\,,\cG^\cD_{S^\circ}) $,
where $\cG^\cD_{S^\circ}$ is the event that no walls in $S^\circ$, the bulk of $S$, are mesoscopic. We infer this bound from a sharp upper bound on  $\mu_n^\mp(\bar M^\downarrow_S < h_n^*-k \,,\cG^\cD_{S^\circ})$ and a sharp lower bound on $\mu_n^\mp( \bar M^\downarrow_S< h_n^*\,,\cG^\cD_{S^\circ}) $. The former is provided by~\eqref{eq:max-thick-upper-large-h} in Proposition~\ref{prop:max-thick}, as discussed above. The latter is available from~\eqref{eq:max-thick-lower-large-h}, which is given in terms of a slightly different event $\cG^\fm_{S^\circ}$, yet in our application this event implies the required event $\cG^\cD_{S^\circ}$.	
\end{enumerate}

In Section~\ref{sec:lower-bound} we give a short proof (Claim~\ref{clm:lower-bound-1-eps}) for entropic repulsion at $\sfh<(1-\epsilon_\beta)h_n^*$ (away from the critical threshold); see the discussion following it for more on the obstacles that arise at $\sfh=(1+o(1))h_n^*$.

\subsubsection{Idea of proof of~\eqref{eq:A-supercritical-h}} 
We now describe our approach for showing that in the regime of $\sfh \ge h_n^*$, all but $\epsilon_\beta n^2$ faces of $\cI$ are \emph{at most} at height zero (the fact that all but $\epsilon_\beta n^2$ faces of $\cI$ are \emph{at least} at height zero follows straightforwardly from the FKG inequality and rigidity under the unconditional measure $\mu_n^\mp$). 

Consider a ceiling $\cC$ at height $\hgt(\cC)\geq 1$ whose interior $S$ has area at least $\theta n^2$. To demonstrate the basic idea that lets rigidity prevail over entropic repulsion in this range of $\sfh$, suppose that the interface gives rise to this ceiling by means of a cylindrical wall $W$ of the form $\partial S \times [0,\hgt(\cC)]$. In this situation, one can compare the original interface $\cI$ with the interface $\cI'$ obtained by truncating the height of the wall $W$ to be $\hgt(\cC)-1$: the weight of $\cI'$ will increase by about $\exp((\beta-C) |\partial S|)$, so this simple Peierls argument can rule out the existence of such a ceiling $\cC$
as long as it does not violate the hard floor constraint.
 By definition, $\cI'\in\Ifloor$ if and only if $\bar M_S^\downarrow$, the maximal downward oscillation inside $\cC$, satisfies $\bar M_S^\downarrow < \hgt(\cC)+\sfh$. The probability of the latter event is  exponentially small in $n$, yet it can still be outweighed by the energy gain due to truncating~$W$.

Concretely, let $\bI_W\subset\Ifloor $ be the set interfaces containing the above wall $W$, and further let $\tilde\bI_W\subset\bI_W$ be the subset of interfaces which in addition satisfy $\bar M_S^\downarrow \leq h_n^*$ (so that, in particular, $\bar M_S^\downarrow < \hgt(\cC)+\sfh$). We may reduce $\mufloor_n^\sfh(\cI\in\tilde\bI_W \mid \cI\in\bI_W)$, akin to the discussion in the previous section, to $\mu_n^\mp(\bar M_S^\downarrow \leq h_n^* \mid \bar M_S^\downarrow \leq \hgt(\cC)+\sfh)$, and appeal to Proposition~\ref{prop:max-thick}---however this time we are interested in a sharp lower bound on this quantity. To that end, we can use the sharp lower bound on $\mu_n^\mp(\bar M_S^\downarrow \leq h_n^*) \geq \exp(-(1+\epsilon_\beta)|S|e^{\alpha_{h_n^*+1}})$ from~\eqref{eq:max-thick-lower-large-h} in that proposition (only the upper bound~\eqref{eq:max-thick-upper-large-h} had the extra event $\cG^\cD_{S^\circ}$ that greatly complicated matters), and simply drop the conditioning (only decreasing the probability in doing so). By the definition of $h_n^*$ in~\eqref{eq:h-star-def} and the fact that the sequence is known~(\cite{GL19a}) to have $\alpha_{h+1}\geq \alpha_h + 4\beta-C$, we have $\alpha_{h_n^*+1} \geq e^{-2\beta+C} /n$, and combining this with the Peierls bound on $\mufloor_n^\sfh(\tilde\bI_W)$ shows that $\mufloor_n^\sfh(\bI_W)$ is bounded from above by approximately
\[ \exp\left(-(\beta-C)|\partial S| + |S|e^{-2\beta+C}/n\right) \leq 
\exp\big(-\big(\beta-C + e^{-2\beta+C'}\big)|\partial S|\big) \,,\]
using here  the isoperimetric inequality $|S| \leq |\partial S| n/4$. 
Notice that this argument did not actually need $S$ to have area $\theta n^2$; e.g., if the boundary of $S$ were at least $\log n$, one could already rule out $\cC$ via a union bound; our actual proof (applicable to a general wall $W$ as opposed to the  cylinder from the toy example above), along the same vein, will rule out any ceiling whose supporting wall $W$ contains at least~$n^{9/10}$ faces.

The main obstacle in using this approach to rule out general ceilings $\cC$ at positive height is that, in the 3D Ising model, the only (tractable) approach one has to shrinking the height of a subset of the interface is the deletion of a collection of walls: even a single wall may have a complicated landscape of overhangs and nested walls near its boundary, so devising a notion of truncating its height to arrive at a valid interface, while gaining the energy of $|\partial S|$ without an entropic cost, is highly nontrivial. (E.g., how do we reconcile the effect of having walls nested in $W$ shift due to the truncation and collide with $W$? Deleting those would impact the entropy, whereas modifying them in any other way can, in addition, lead to similar ripple effects.)

Furthermore, even if we decide to delete the entire wall $W$, our Peierls map cannot consist of that alone: 
  due to the long-range interactions that the deletion of $W$ and the resulting shifts induce between interface faces, this operation must also be followed by the deletion of certain other walls that interact too strongly with $W$---either the \emph{groups of walls} in Dobrushin's original work~\cite{Dobrushin72a}, or the one-sided \emph{wall clusters} in~\cite{GL20}. 
Roughly put, the wall cluster criterion says that if $W$ nests a wall $W'$ whose number of faces is larger than its distance to $W$, then we must also tag $W'$ for deletion, and process its inner walls recursively (see Def.~\ref{def:wall-cluster}). 
Hence, if we delete $W$ and its wall cluster $\bW$, and then aim to analyze $\bar M_S^\downarrow$, we have the following problem: whereas revealing only $W$ and its exterior gave no information on the interface in $S$ (so we were free to apply Proposition~\ref{prop:max-thick}), the act of revealing its wall cluster $\bW$ robs us of that feature, by imposing the constraint that the size of every wall $W'$ that remained unexposed in $S$  should not exceed its distance to~$\bW$... 

Our remedy for this problem is establishing a certain monotonicity principle, which may find other uses.
Intuitively, conditioning on wall sizes being below various thresholds, albeit complicated (here, the threshold for a given wall depends on its distance to the wall cluster), should only yield better control over  oscillations. Making this intuition precise is the observation that  results relying on Peierls maps that only \emph{deleted walls}---such were, e.g., key results in~\cite{GL20} that we use here---can also be established in the above conditional setting (if an interface satisfied the wall size constraints, so will the interface corresponding to a subset of its walls). As such, we can extend said results from~\cite{GL20}, and notably the lower bounds of Proposition~\ref{prop:max-thick}, to this setting (see Proposition~\ref{prop:max-wall-cluster}).
This observation and the extensions it gives rise to appear in Section~\ref{sec:extension-estimates-in-ceiling}.

In Section~\ref{sec:upper-bound} we give a short proof (Claim~\ref{clm:upper-bound-1+eps}) for rigidity at height zero for $\sfh\geq(1+\epsilon_\beta)h_n^*$ (away from the critical threshold); see the discussion following it for more on the obstacles that arise at $\sfh=(1+o(1))h_n^*$.

\subsection{Outline of paper}\label{subsec:outline-of-paper}
In Section~\ref{sec:prelim}, we overview the notation we use in the paper and recall many of the key definitions and tools from~\cite{Dobrushin72a} and~\cite{GL20} used in our analysis. In Section~\ref{sec:extension-estimates-in-ceiling}, we extend several of the preliminary estimates of~\cite{GL20} to also hold under a monotone conditioning event, which will later be needed for showing~\eqref{eq:A-supercritical-h}. In Section~\ref{sec:rough-bounds-no-floor}, we establish an a priori bound on the probability that the interface lies above the floor at height~$0$, and use it to prove preliminary regularity estimates for the interface above a floor at~$-\sfh$. In Section~\ref{sec:lower-bound}, we treat the case of $\sfh < h_n^*-1$, establishing~\eqref{eq:A-subcritical-h} of Theorem~\ref{thm:1}. Then in Section~\ref{sec:upper-bound}, we handle $\sfh \ge h_n^*$, establishing~\eqref{eq:A-supercritical-h} of Theorem~\ref{thm:1}. 

\section{Preliminaries}\label{sec:prelim}
In this section, we recall the decomposition of the interface into walls and ceilings and their groupings into wall clusters. We then recall the results of~\cite{GL20} that we will require in this paper regarding the law of the maximum oscillation inside a ceiling, conditionally on all walls outside that ceiling. Throughout this, we will also introduce any notation that will be used throughout the paper.  

\subsection{Notation}\label{subsec:notation} We begin by describing the underlying geometry on which we will be working, and any related graph notation we will use throughout the paper. 

\subsubsection*{Lattice notation}
The underlying graphs we consider throughout this paper are rectangular subsets of $\mathbb Z^3$. To be precise, $\Z^3$ is the integer lattice graph with vertices at $(x_1,x_2,x_3)\in \Z^3$ and edges between nearest neighbor vertices (at Euclidean distance one). A \emph{face} of $\Z^3$ is the open set of points bounded by four edges (or four vertices) forming a square of side length one, lying normal to one of the coordinate directions. A \emph{cell} of $\Z^3$ is the set of points bounded by six faces (or eight vertices) forming a cube of side length one. 

Because our interest is the behavior of the \emph{interface} separating plus and minus spins, it will be convenient for us to consider Ising configurations as assignments of $\pm 1$ spins to the vertices of the \emph{dual} graph $(\Z^3)^*= (\Z+\frac 12)^3$; these are naturally identified with the cells of $\mathbb Z^3$ for which they are the midpoint.

More generally, we will frequently identify edges, faces, and cells with their midpoints.
A subset $\Lambda \subset \Z^3$ specifies an edge, face, and cell collection via the edges, faces, and cells of $\mathbb Z^3$ all of whose bounding vertices are in $\Lambda$. We will denote the resulting edge set by $\sE(\Lambda)$, face set by $\sF(\Lambda)$, and cell set by $\sC(\Lambda)$.  

Two edges are adjacent if they share a vertex, two faces adjacent if they share a bounding edge, and two cells adjacent if they share a bounding face. 
We denote adjacency by the notation $\sim$.
It will also be useful to have a notion of connectivity in $\R^3$ (as opposed to $\Z^3$); we say that an edge/face/cell is $*$-adjacent, denoted $\sim^*$, to another edge/face/cell if they share a bounding vertex. A connected component of faces (respectively $*$-connected component of faces) is a maximal set of faces such that for any pair of faces in that set, there is a sequence of adjacent faces (resp., $*$-adjacent) faces starting at one face and ending at the other. 

We use the notation $d(A,B)=\inf_{x\in A, y\in B} d(x,y)$ to denote the Euclidean distance in $\mathbb R^3$ between two sets $A,B$. 
We then let $B_r(x) = \{y: d(y,x)\le r\}$. When these balls are viewed as subsets of edges/faces/cells, we include all those edges/faces/cells whose midpoint falls in $B_r(x)$. 

\subsubsection*{Subsets of $\Z^3$} As mentioned, the primary subsets of $\Z^3$ with which we will be concerned are of the form of cubes and cylinders. In view of that, define the centered $n\times m \times h$ box,  
\begin{align*}
\Lambda_{n,m,h} :=\llb -\tfrac n2,\tfrac n2\rrb \times \llb -\tfrac m2,\tfrac m2\rrb \times \llb -\tfrac h2,\tfrac h2\rrb \subset \Z^3\,,
\end{align*}
where
\begin{align*}
\llb a,b\rrb := \{\lfloor a\rfloor,\lfloor a\rfloor +1,\ldots,\lceil b\rceil-1,\lceil b\rceil\}\,.
\end{align*}
We can then let $\Lambda_n$ denote the special case of the cylinder $\Lambda_{n,n,\infty}$. The (outer) boundary $\partial \Lambda_n$ of the cell set $\sC(\Lambda_n)$ is the set of cells in $\sC(\Z^3)\setminus \sC(\Lambda_n)$ adjacent to a cell in $\sC(\Lambda_n)$. 

We are also going to use a dedicated notation for horizontal slices and half-spaces of $\mathbb Z^3$. For any $h \in \Z$ let $\cL_h$ be the subgraph of $\Z^3$ having vertex set $\Z^2\times \{h\}$  and correspondingly define edge and face sets $\sE(\cL_h)$ and $\sF(\cL_h)$. For a half-integer $h\in \Z + \frac 12$, let $\cL_h$ collect the faces and cells in $\sF(\Z^3) \cup \sC(\Z^3)$ whose midpoints have half-integer $e_3$ coordinate $h$. A certain such set which will recur is $\cL_0$ and its restriction to $\sF(\Lambda_{n})$ which we denote by $\cL_{0,n}$.  

Finally we use $\cL_{>h} = \bigcup_{h'>h} \cL_{h'}$ and $\cL_{<h} = \bigcup_{h'<h} \cL_{h'}$, and similarly $\cL_{\ge h}$ and $\cL_{\le h}$, for half-spaces. 

\subsubsection*{Projections onto $\cL_0$}
Throughout the paper, we will refer to a face as \emph{horizontal} if its normal vector is $\pm e_3$, and as \emph{vertical} if its normal vector is one of $\pm e_1$ or $\pm e_2$. 

For a face $f\in \sF(\Z^3)$, its \emph{projection} is the edge or face given by 
$$\rho(f)= \{(x_1,x_2,0):(x_1,x_2,s)\in f \mbox{ for some $s\in \R$}\}\subset \cL_0\,.$$ Specifically, the projection of a \emph{horizontal} face is a face in $\sF(\cL_0)$, while the projection of a \emph{vertical} face is an edge in $\sE(\cL_0)$. 
The projection of a collection of faces $F$ is given by $\rho(F) := \bigcup_{f\in F} \rho(f)$, which may consist both of edges and faces of $\cL_0$. 

\subsection{The Ising model}\label{subsec:Ising-model}
An Ising configuration $\sigma$ on a subset $\Lambda \subset \Z^3$ is an assignment of $\pm1$-valued spins to the cells of $\Lambda$, i.e., $\sigma \in \{\pm 1\}^{\sC(\Lambda)}$.  For a finite connected subset $\Lambda\subset \Z^3$, the Ising model on $\Lambda$ with boundary conditions $\eta \in \{\pm 1\}^{\sC(\Z^3)}$ is the probability distribution over $\sigma \in \{\pm 1\}^{\sC(\Lambda)}$ given by 
\begin{align}\label{eq:Ising-def}
\mu_{\Lambda}^{\eta} (\sigma) &  \propto \exp \left[ - \beta  \cH(\sigma)\right]\,, \qquad \mbox{where}\qquad \nonumber \\ \cH(\sigma) & =  \sum_{\substack{ v,w\in \sC(\Lambda) \\  v\sim w }} \one\{\sigma_v\neq \sigma_w\} +  \sum_{\substack{ v\in \sC(\Lambda), w\in \sC(\Z^3)\setminus \sC(\Lambda) \\  v\sim w}} \one\{\sigma_v\neq \eta_w\}\,. 
\end{align}
Throughout this paper, we will be considering the \emph{Dobrushin boundary conditions}, denoted $\eta = \mp$, where $\eta_w = -1$ if $w$ is in the upper half-space ($w_3 > 0$) and $\eta_w = +1$ if $w$ is in the lower half-space ($w_3 <0$). 

\subsubsection*{Infinite-volume measures} Though~\eqref{eq:Ising-def} is defined only on finite graphs, the definition can be extended to infinite graphs via a consistency criterion known as the \emph{DLR conditions} (see e.g., the book~\cite{Georgii} for a definition and discussion. 
On $\Z^d$, such \emph{infinite-volume measures} arise as weak limits of finite-volume measures, say $n\to\infty$ limits of the Ising model on boxes of side length $n$ with prescribed sequences of boundary conditions. At low temperatures $\beta>\beta_c(d)$, the Ising model on $\Z^d$ admits multiple infinite-volume Gibbs measures $\mu^+_{\Z^3}$ and $\mu^-_{\Z^3}$ obtained by taking plus and minus boundary conditions on boxes of side-length $n$ and sending $n\to\infty$. An important consequence of the work of~\cite{Dobrushin72a} was that when $d\ge 3$, the weak limit $\mu_{\Z^3}^{\mp}:= \lim_{n\to\infty} \mu_{n,n,n}^{\mp}$ is a non-translation-invariant DLR measure, and is thus distinct from any mixtures of $\mu_{\Z^32}^+$ and $\mu_{\Z^32}^-$.

\subsection{Interfaces under Dobrushin boundary conditions}\label{subsec:dobrushin-definitions}
Having defined Ising configurations and the Ising measure with Dobrushin boundary conditions, let us now formally define the \emph{interface} separating the plus and minus phases. 

\begin{definition}[Interfaces]\label{def:interface} For a domain $\Lambda_{n,m,h}$ with Dobrushin boundary conditions, and an Ising configuration
$\sigma$ on $\sC(\Lambda_{n,m,h})$, the \emph{interface} $\cI=\cI(\sigma)$
is defined as follows:  
\begin{enumerate}
\item Extend $\sigma$ to a configuration on $\sC(\mathbb Z^3)$ by taking $\sigma_v=+1$ if $v\in \cL_{<0}\setminus \sC(\Lambda_{n,m,h})$ and $\sigma_v = -1$ if  $v\in \cL_{>0}\setminus \cC(\Lambda_{n,m,h})$. 
\item Let $F(\sigma)$ be the set of faces in $\sF(\mathbb Z^3)$ separating cells with differing spins under $\sigma$.
\item  Let $\cI^{\infty}(\sigma)$ be the (maximal) $*$-connected component of faces in $F(\sigma)$ containing $\cL_0 \setminus \sF(\Lambda_{n,m,h})$ in $F(\sigma)$. (This is also the unique infinite $*$-connected component in $F(\sigma)$.)
\item The interface $\cI(\sigma)$  is the restriction of $\cI^\infty(\sigma)$ to $\sF(\Lambda_{n,m,h})$. 
\end{enumerate}
\end{definition}

Taking the $h\to\infty$ limit $\mu_{n,m,h}^\mp$ to obtain the infinite-volume measure $\mu_{n,m,\infty}^\mp$, the interface defined is easily seen to stay finite almost surely. Thus, $\mu_{n,m,\infty}^\mp$-almost surely, the above process also defines the interface for configurations on all of $\sC(\Lambda_{n,m,\infty})$. 

\begin{remark}\label{rem:interface-spin-config}
Every (finite) interface uniquely defines a configuration with exactly one $*$-connected plus component and exactly one $*$-connected minus component. For every $\cI$, we can obtain this configuration, which we call $\sigma(\cI)$, by iteratively assigning spins to $\cC(\Lambda_{n,m,h})$, starting from $\partial \Lambda_{n,m,h}$ and proceeding inwards, in such a way that adjacent sites have differing spins if and only if they are separated by a face in $\cI$. 
 Informally, $\sigma(\cI)$ is indicating the sites that are in the ``plus phase" and ``minus phase" given the interface~$\cI$.
\end{remark}

\subsection{The Ising interface conditioned on a floor}
In this paper, we study the behavior of the interface $\cI(\sigma)$ when conditioned on a \emph{floor event}, i.e., conditioned on the interface lying above a certain horizontal plane (the floor). For us, the floor will either reside at height zero, or into the lower half-space, so that we are examining the entropic repulsion of the interface into the upper half-space. More precisely, we use the following notation to denote the Ising measure with a floor at height $-\sfh$ for $\sfh\ge 0$:  
\begin{align*}
    \mufloor_n^{\sfh}(\sigma) = \mu_n^\mp(\sigma \mid \cI \subset \cL_{\ge -\sfh}) = \mu_n^\mp(\sigma \mid \Ifloor)\,.
\end{align*}

\subsection{Decomposition of the interface: walls and ceilings}\label{subsec:walls-ceilings}
Having defined the key objects of interest in this paper, we now begin to collect the main tools for its analysis. We begin by recalling the classical decomposition of~\cite{Dobrushin72a} of the Ising interface into \emph{walls} and \emph{ceilings}. 

\begin{definition}[Walls and ceilings]\label{def:ceilings-walls}
A face $f\in \cI$ is a \emph{ceiling face} if it is horizontal and there is no $f'\in \cI$, $f'\ne f$ such that $\rho(f)=\rho(f')$. A face $f\in \cI$ is a \emph{wall face} if it is not a ceiling face. 

A \emph{wall} is a $*$-connected component of wall faces and a \emph{ceiling} is a  $*$-connected component of ceiling faces. 
\end{definition}

Intuitively speaking, the walls capture all the oscillations of the interface off of the lowest energy (horizontally flat) interface, while the ceilings capture the flat stretches of the interface. In particular, every ceiling has a unique \emph{height}, denoted $\hgt(\cC)$, since all faces in the ceiling have the same $x_3$ coordinate.  

\begin{definition}
Throughout the paper, we let $\fC(\cI)$ be the collection of all ceilings of $\cI$, and for every height $h\in \mathbb Z$, we let $\fC_{h}(\cI) = \{\cC\in \fC(\cI): \hgt(\cC) = h\}$. 
\end{definition}

The projections of walls, like non-crossing loop collections in $\mathbb Z^2$, satisfy certain important nesting relations.    

\begin{definition}[Nesting of walls]\label{def:nesting-of-walls}
For a wall $W$, the complement (in $\cL_{0}$) of its projection, denoted
$$\rho(W)^c:= (\sE(\cL_0)\cup \sF(\cL_0)) \setminus \rho(W)\,,$$
splits into one infinite component, and some finite ones. An edge or face $u\in \sE(\cL_0)\cup \sF(\cL_0)$ is said to be \emph{interior} to (or \emph{nested} in) a wall $W$, denoted by $u\Subset W$, if $u$ is not in the infinite component of $\rho(W)^c$. 
A wall $W'$ is nested in a wall $W$, denoted $W'\Subset W$, if every element of $\rho(W')$ is interior to $W$. Similarly, a ceiling $\cC$ is nested in a wall $W$ if every element of $\rho(\cC)$ is interior to $W$. 
\end{definition}

We can then identify the connected components of $\rho(W)^c$ with the ceilings incident to $W$. 

\begin{lemma}[{\cite{Dobrushin72a}}]\label{lem:wall-ceiling-bijection}
For a projection of the walls of an interface, each connected component of that projection (as a subset of edges and faces) corresponds to a single wall. Moreover, there is a 1-1 correspondence between the ceilings adjacent to a standard wall $W$ and the connected components of $\rho(W)^c$. Similarly, for a wall $W$, all other walls $W'\neq W$ can be identified to the connected component of $\rho(W)^c$ they project into, and in that manner they can be identified to the ceiling of $W$ to which they are interior.
\end{lemma}

The above correspondence can be made more transparent by introducing the following notion. 

\begin{definition}
For a wall $W$, the ceilings incident to $W$ can be decomposed into \emph{interior ceilings} of $W$ (those ceilings identified with the finite connected components of $\rho(W)^c$), and a single exterior ceiling, called the $\emph{supporting ceiling}$ of $W$, identified with the infinite connected component of $\rho(W)^c$. 
\end{definition}

\begin{definition}\label{def:hull}
The hull of a ceiling $\cC$, denoted $\hull{\cC}$ is the minimal simply-connected set of horizontal faces containing $\cC$. The hull of a wall, $\hull{W}$ is the union of $W$ with the hulls of its interior ceilings. 
\end{definition}

\begin{observation}
For every interior ceiling $\cC$ of $W$, the projection of its hull, $\rho(\hull{\cC})$, is exactly the finite component of $\rho(W)^c$ it projects into. On the other hand, the projection of the hull of the floor of $W$ is all of $\cL_{0,n}$. Finally, the set $\rho(\hull{W})$ is the union of $\rho(W)$ with all the finite components of $\rho(W)^c$. 
\end{observation}

Finally, we can assign the index points of $\cL_{0,n}$ the walls of an interface $\cI$ as follows. 

\begin{remark}\label{rem:indexing-walls}
Given an interface $\cI$, for every face $x\in \sF(\cL_0)$, assign $x$ the wall $W$ of $\cI$ if $x\Subset W$ and $x$ shares an edge with $\rho(W)$. If there is no $W$ for which this is the case, let $W_x = \trivincr$. Importantly, this labeling scheme is such that $x$ is only assigned one wall, but the same wall may be assigned to many index faces. 
\end{remark}

 We also introduce a notion of \emph{a nested sequence of walls} which will recur. 
 
 \begin{definition}\label{def:nested-sequence-of-walls}
To any edge/face/cell $x$, we can assign a \emph{nested sequence of walls} $\fW_x = \bigcup_{s} W_{u_s}$ consisting of all walls nesting $\rho(x)$ (by Definition~\ref{def:nesting-of-walls}, this forms a nested sequence of walls). 
\end{definition}

\subsection{The standard wall representation}\label{subsec:standard-wall-representation}
A key property of the wall and ceilings decomposition of~\cite{Dobrushin72a} is that neither the ceilings, nor crucially the vertical positions of the walls, are needed to reconstruct the interface. This is formalized by the following notion of standard walls, also defined in~\cite{Dobrushin68}. 

\begin{definition}[Standard walls]\label{def:standard-walls}
A wall $W$ is a \emph{standard wall} if there exists an interface $\cI_W$ such that $\cI_W$ has exactly one wall given by $W$.
A collection of standard walls is \emph{admissible} if any two standard walls in the collection have pairwise vertex disjoint projections. 
\end{definition}

\begin{definition}[Standardization of walls]\label{def:std-of-wall}
 For every wall $W$ of $\cI$ , we can define its \emph{standardization} $\Theta_{\textsc {st}}W$ which is the translate of the wall by $(0,0,-s)$ where $s$ is the height of its supporting ceiling. (We leave the dependence of $s$, and therefore the dependence of $\Theta_{\textsc{st}}W$ on the rest of $\cI$ to be contextually understood.)
\end{definition}

We then have the following important bijection between interfaces and their standard wall representation, defined as the collection of standard walls given by standardizing all walls of $\cI$.

\begin{lemma}[{\cite{Dobrushin72a}}]\label{lem:interface-reconstruction}
The map sending an interface to its \emph{standard wall representation} is a bijection between the set of all valid interfaces and the set of all admissible collections of standard walls. 
\end{lemma}

We note the following important observation based on the bijection given by Lemma~\ref{lem:interface-reconstruction}.  

\begin{observation}\label{obs:1-to-1-map-faces}
Consider interfaces $\cI$ and $\cJ$, such that the standard wall representation of $\cI$ contains that of $\cJ$ (and additionally has the standard walls $\Theta_{\textsc{st}}\bW=(\Theta_{\textsc{st}}W_1,\ldots,\Theta_{\textsc{st}}W_r)$). There is a 1-1 map between the faces of $\cI \setminus \bW$ and the faces of $\cJ \setminus \bH$ where $\bH$ is the set of faces in $\cJ$ projecting into $\rho (\bW)$. Moreover, this bijection can be encoded into a map $f\mapsto \theta_{\udarrow} f$ that only consists of vertical shifts, and such that all faces projecting into the same connected component of $\rho(\bW)^c$ undergo the same vertical shift. 
\end{observation}

\subsection{The induced distribution on interfaces}\label{subsec:cluster-expansion}
Aside from the decomposition of the Ising interface into walls and ceilings, the other key tool used by~\cite{Dobrushin72a} to establish rigidity of the interface was an expression for the Ising distribution over interfaces as a perturbation, by means of the cluster expansion~\cite{Minlos-Sinai}, of a Gibbs measure on interfaces with weight $\exp( - \beta |\cI|)$. The perturbative term takes into account the interaction of bubbles of the low-temperature Ising configurations in the minus phase above and plus phase below the interface, with the interface itself.  

Here and throughout the paper, let $\mu_n^{\mp} = \mu_{\Lambda_n}^\mp = \lim_{h\to\infty} \mu_{n,n,h}^\mp$.

\begin{theorem}[{\cite[Lemma 1]{Dobrushin72a}}]
\label{thm:cluster-expansion} There exist $\beta_0>0$, $\bar c>0$, $\bar K>0$, and a function $\g = \g_\beta$ such that the following holds for every $\beta>\beta_{0}$. For every interface $\cI$, 
\begin{align*}
\mu_n^\mp(\cI)& \propto \exp\bigg(-\beta |\cI| + \sum_{f\in\cI}\g(f,\cI)\bigg)\,,
\end{align*}
and for every $\cI, \cI'$ and $f\in \cI$ and $f'\in\cI'$,
\begin{align}
|\g(f,\cI)| & \leq\bar{K} \label{eq:g-uniform-bound} \\
|\g(f,\cI)-\g(f',\cI')| & \leq \bar K e^{-\bar{c}\br(f,\cI;f',\cI')} \label{eq:g-exponential-decay}
\end{align}
where $\br(f,\cI;f',\cI')$ is the largest radius around
the origin on which $\cI-f$ ($\cI$ shifted by the midpoint of the face $f$) is \emph{congruent} to $\cI'-f'$. That is to say,
\begin{align*}
\br(f,\cI; f', \cI') := \sup \big\{r: (\cI - f) \cap B_{r}(0) \equiv (\cI' - f') \cap B_r(0)\big\}\,,
\end{align*}
where the congruence relation $\equiv$ is equality as subsets of $\R^3$.  
\end{theorem}

We will say that the radius $\br(f,\cI; f',\cI')$ is \emph{attained by} a face $ g \in \cI$ (resp., $g' \in \cI'$) of minimal distance to $f$ (resp., $f'$) if the presence of that face prevents $\br(f, \cI; f',\cI')$ from being any larger.

\subsection{Excess energy of interfaces and walls}\label{subsec:excess-energies}
Given the representation of Theorem~\ref{thm:cluster-expansion}, we see that the ratio between the weight of $\cI$ and the weight of $\cI'$ will have a term of the form $\exp( - \beta (|\cI| - |\cI'|))$. This difference in sizes of the interfaces can be thought of as the \emph{excess energy} of one interface over the other, and will recur throughout the paper. 

\begin{definition}[Excess energy]
\label{def:excess-energy}
For two interfaces $\cI,\cJ$,
the \emph{excess energy} of $\cI$ with respect to $\cJ$,
denoted $\fm(\cI;\cJ)$, is given by $$\fm(\cI;\cJ):= |\cI|-|\cJ|\,.$$
Evidently, for any valid interface $\cI$, we have that $\fm(\cI; \cL_{0,n})\geq 0$.   
\end{definition}

We relate this notion back to the walls and ceilings, and assign excess energies to walls as follows. 

\begin{definition}\label{def:excess-energy-properties}
For a wall $W$, recall that $\cI_{\Theta_{\textsc{st}}W}$ is the interface whose only wall is the standardization of $W$, and define its excess energy by 
\begin{align*}
    \fm(W)= \fm(\cI_{\Theta_{\textsc{st}}W}; \cL_{0,n})\,.
\end{align*}
This excess energy can alternatively be expressed as $\fm(W) = |W| - |\sF(\rho(W))|$; one then finds that 
\begin{align}\label{eq:excess-energy-wall-relations}
\fm(W) \geq \frac{1}{2} |W|\,, \qquad \mbox{and} \qquad \fm(W) \geq |\rho(W)| = |\sE(\rho(W)|+ |\sF(\rho(W))|\,.
\end{align} 
Evidently, any two faces $x,y\in \cL_{0,n}$ both nested in $W$ must satisfy $d(x,y) \leq \fm(W)$. 
\end{definition}

For a collection of walls $\bW$, we define $\fm(\bW) = \sum_{W\in \bW} \fm(W)$. It then becomes evident that the excess energy of an interface, $\fm(\cI;\cL_{0,n})$ is exactly given by the excess energy of its wall collection. 

\subsection{Wall clusters and rigidity inside ceilings}\label{subsec:wall-clusters}
Up to this point, all of the above consisted of simple reformulations of concepts from the classical work~\cite{Dobrushin72a}. The main result there was exponential tails on the wall through a face $x\in \cL_{0,n}$ for any $x$. In~\cite{GL20}, these ideas were extended and a concept called \emph{wall clusters} was introduced to establish that these exponential tails hold even conditionally on arbitrary nesting walls. 

Towards the definition of wall clusters, we first define a notion of \emph{close nesting} between walls. 

\begin{definition}\label{def:closely-nested}
A wall $W'$ is closely nested in a wall $W$ if $W'$ is nested in $W$ ($W' \Subset W$) and 
\[
d_\rho(W, W') := d(\rho(W),\rho(W'))  \le \fm(W')\,.
\]
\end{definition}

\begin{definition}\label{def:wall-cluster}
For a wall $W$, define the \emph{wall cluster} $\mathsf{Clust}(W)$ as follows:
\begin{enumerate}
    \item Initialize $\mathsf{Clust}(W)$ with $W$.
    \item Iteratively, add to $\mathsf{Clust}(W)$ all walls $W''$ that are closely nested in some wall $W'\in \mathsf{Clust}(W)$. 
\end{enumerate}
\end{definition}

\begin{remark}
The key upshot of wall clusters is that unlike the analogous grouping scheme in~\cite{Dobrushin72a} (called \emph{groups of walls}), the wall cluster of $W$ is completely nested in $W$, i.e., $\rho(\Clust(W))\subset \rho(\hull{W})$.
\end{remark}

In~\cite{GL20} we used the following bounds on the effect on the Ising measure of removing a wall cluster at a point $x\in \cL_{0,n}$, and on the number of possible wall clusters through $x$. 

\begin{lemma}[{\cite[Lemma 3.10]{GL20}}]\label{lem:wall-cluster-weights}
There exists $C>0$ such  that for every $\beta>\beta_0$ the following holds. For every fixed $\cI$, and every $x\in \cL_{0,n}$, if $\cJ$ is the interface obtained (per Lemma~\ref{lem:interface-reconstruction}) by removing $\Theta_{\textsc{st}}\Clust(W_{x})$ from the standard wall representation of $\cI$, then 
\begin{align*}
   \Big|\log\frac{\mu_n^\mp (\cI)}{\mu_n^\mp(\cJ)} + \beta\fm(\Clust(W_{x}))\Big| \le C \fm(\Clust(W_x))\,.
\end{align*}
\end{lemma}

\begin{lemma}[{\cite[Lemma 3.13]{GL20}}]\label{lem:number-of-wall-clusters}
Fix $x\in \cL_{0,n}$, and fix \emph{any} set of walls $(W_z)_{z\in A}$ for some $A\subset \cL_{0,n}$ with $x\notin A$. There exists a universal constant $C>0$ such that the number of wall clusters $\Clust(W_x)$ compatible with $(W_z)_{z\in A}$, and having $\fm(\Clust(W_x)) = M$ is at most $C^M$. 
\end{lemma}

Combining the above,~\cite[Section 3]{GL20} established rigidity in a set $S \subset \cL_{0,n}$ even conditionally on the shape of the interface outside $S$. In order to phrase this formally, let us introduce some notation. For a collection of walls $\bW$, let $\bI_\bW$ be the set of all interfaces having $\Theta_{\st}\bW$ in its standard wall representation. Notice that if $\bW = (W_z)_{z\notin S}$ and $\rho(\bW) \cap S = \emptyset$, then for every $\cI \in \bI_{\bW}$, there is a single ceiling $\cC_{\bW}$ having $\rho(\cC_{\bW}) \supset S$. In particular, given the event $\{\cI\in\bI_{\bW}\}$, the minimal energy interface is then the one for which 
$W_x = \emptyset$ for all $x\in S$. The next results give exponential tails of walls of $x\in S$ and height oscillations above $x\in S$ about $\hgt(\cC_\bW)$, conditionally on $\bI_{\bW}$. 

\begin{theorem}[{\cite[Thm.~3.1]{GL20}}]
\label{thm:rigidity-inside-wall}
There exists $C>0$ so that for all $\beta>\beta_0$ the following holds. Fix any two admissible collections of walls $(W_1, \ldots, W_r)$, and $\bW = (W_z)_{z\in A}$ such that $\rho(\bW) \cap  \rho(\bigcup_{i\le r} \hull{W}_i)= \emptyset$. Then 
\[
\mu_n^\mp(\bI_{W_1,\ldots,W_r} \mid \bI_\bW) \leq \exp\Big(-(\beta-C)\sum_{i\le r}\fm(W_i)\Big)\,.
\]
\end{theorem}

As a consequence we obtain the following for nested sequences of walls inside $S$. In what follows, for a set $S\subset \cL_{0,n}$ and a site $x\in S$, let $\fW_{x,S}$ be the collection of walls nesting $x$, and themselves fully nested in $S$. Noticing that the number of possible nested wall collections $\fW_{x,S}$ having $\fm(\fW_{x,S}) = M$ is at most $C^M$ for some universal constant $C$, and enumerating over these possible choices, we obtain the following corollary. 

\begin{corollary}[{\cite[Cor.~3.2]{GL20}}]\label{cor:nested-sequence-inside-a-ceiling}
There exists $C>0$ so that for all $\beta>\beta_0$ the following holds.  Let $S_n\subset \cL_{0,n}$ be such that $\cL_{0}\setminus S_n$ is connected, and let $\bW_n = (W_z)_{z\notin S_n}$ be an admissible collection of walls
with $\rho(\bW_n) \subset S_n^c$. 
For every $x\in S_n$ and every $r\ge 1$,
\begin{align*}
    \mu_n^\mp\left( \fm(\fW_{x,S_n}) \ge r \mid \bI_{\bW_n}
    \right) \leq \exp ( -  (\beta - C)r)\,.
\end{align*}
\end{corollary}

We will also use the following extension of Corollary~\ref{cor:nested-sequence-inside-a-ceiling} to the case of $x_1,\ldots,x_N\in S_n$ for some $N\leq r$. 

\begin{corollary}\label{cor:nested-sequence-of-multiple-faces-in-a-ceiling}
In the setting of Corollary~\ref{cor:nested-sequence-inside-a-ceiling}, if $x_1,\ldots,x_N\in S_n$ for some $N\leq r$ then we have
\[ \mu_n^\mp\Big(\fm\big(\bigcup_{i\leq N}\fW_{x_i,S_n}\big)\geq r \mid \bI_{\bW_n}\Big) \leq \exp(-(\beta-C)r)\,.\]
\end{corollary}
 
\begin{proof}
The number of possible collections of nested sequences of walls $\cup_{i\le N} \fW_{x_i,S_n}$ having total excess energy $\fm(\cup_{i\le N} \fW_{x_i,S_n}) = M$ can be seen to be at most $C^{N+M}$ for some universal constant $C$ as follows: there are $2^{N+M}$ many ways to partition the excess energy into $M_1,\ldots,M_N$ amongst the $N$ points, and then one can systematically enumerate over the wall collections $\fW_{x_i,S_n}\setminus \bigcup_{j<i}\fW_{x_i,S_n}$ having excess area $M_i$ for each $i$, costing in total a further $C^M$. Summing Theorem~\ref{thm:rigidity-inside-wall} over all such collections and all $M\ge r$, one then obtains the desired.
\end{proof}

\subsection{The law of the maximum height oscillation in a ceiling}\label{subsec:law-of-max-in-ceiling}
A key ingredient in our proof will be tail bounds on the maximum height oscillations inside a ceiling and conditionally on the interface outside that ceiling. Towards this us introduce some notation. 
For each face $x$, let $\hgt_\cI(x) = \{h: (x_1,x_2,h)\in \cI\}$. 
For a subset $S\subset \cL_{0,n}$, define 
\begin{align*}
       M_S^\uparrow(\cI) := \max_{\substack{f\in S}} \max \hgt_\cI(f)\,, \qquad \mbox{and} \qquad     M_S^\downarrow(\cI) := \min_{\substack{f\in S}} \min \hgt_\cI(f)\,.
\end{align*}
We will be especially interested in studying the behavior of these quantities when conditioning on $\bI_{\bW_n}$ for some collection $\bW = \{W_z:z\notin S\}$ such that $\rho(\bW) \subset S^c$. In this situation, in the interface $\cI_{\bW}$, there is a single ceiling $\cC_{\bW}$ for which $S\subset \rho(\cC_\bW)$; this will be the reference height with which the oscillations inside $h$ will be compared: as such, define the maximal upwards and downwards oscillations about $\hgt(\cC_\bW)$ as 
\begin{align*}
    \bar M_S^{\uparrow}(\cI) = M_S^\uparrow(\cI) - \hgt(\cC_{\bW})\,,\qquad \mbox{and} \qquad \bar M_S^\downarrow(\cI) = \hgt(\cC_{\bW}) - M_S^\downarrow(\cI)\,.
\end{align*}
Also, for a subset $S\subset \cL_{0,n}$ (understood contextually), recall from~\cite{GL20} the following two events for $A \subset S$:
\begin{equation}\label{eq:G-event} 
\cG^{\fm}_A(r)= \bigcap_{x\in A}\left\{ \fm(\fW_{x,S})< r \right\}\,,\qquad 
\cG^{\mathfrak{D}}_A(r)= \bigcap_{x\in A}\left\{ \diam(\fW_{x,S})< r \right\}
\,.
\end{equation}
We begin with providing certain cruder bounds on the behavior of $\bar M^{\uparrow}_S(\cI), \bar M^\downarrow_S(\cI)$ quantities for arbitrary sets $S$, that can be proved using only the conditional exponential tail bounds obtained in Section~\ref{subsec:wall-clusters}.   

The following bounds were primarily formulated in~\cite{GL20} for upwards oscillations $\bar M_{S_n}^{\uparrow}$, but since their application in this paper is for downwards oscillations, we express them for $\bar M_{S_n}^{\downarrow}$: the two formulations are equivalent by reflection symmetry of $\Lambda_n$ about $\cL_{0}$. 

\begin{proposition}[{\cite[Prop. 6.1]{GL20}}]\label{prop:max-generic}
There exist $\beta_0,C>0$ so the following holds for all $\beta>\beta_0$. Let $S_n\subset\cL_{0,n}$ be a simply-connected set such that $|S_n|\to\infty$ with $n$, and let $\bW_n = \{ W_z : z\notin S_n\}$ be such that $\rho(\bW) \subset S_n^c$. Then for every $h = h_n \ge 1$,
\begin{align} 
\mu^\mp_{n}\left( \bar M_{S_n}^\downarrow< h \given \bI_{\bW_n}\right) \geq 
\mu^\mp_{n}\left( \cG^{\fm}_{S_n}(4h) \given \bI_{\bW_n}\right) &\geq \exp\big(- |S_n| e^{-(4\beta -C)h}\big)\,, \label{eq:max-lower-crude}
\end{align}
and
\begin{align}\label{eq:max-upper-crude}
\mu^\mp_{n}\left(\bar M_{S_n}^\downarrow < h \mid \bI_{\bW_n}\right) &\leq  \exp\left(- |S_n| e^{-(4\beta+C)h}\right) \,.
\end{align}
Moreover, identical bounds hold if $\bar M_{S_n}^{\downarrow}$ is replaced by $\bar M_{S_n}^{\uparrow}$. 
\end{proposition}

For ``nice" sets $S$, the above bounds on the distributions of $\bar M^{\downarrow}_{S_n}$ and $\bar M^\uparrow_{S_n}$ can be sharpened significantly so that the lower and upper bounds essentially match; indeed this was the main result of~\cite{GL20}. It relies on an identification of the exact rate of a deviation of height $h$ above a point $x\in S_n$ as $\alpha_h$ from~\eqref{eq:alpha-h-def}; without any conditioning that goes back to~\cite{GL19a,GL19b}. To state this refined form, which is crucial to the identification of the height $h_n^*$ in our main theorems, we first formalize what we mean by ``nice" sets. 

\begin{definition*}[isoperimetric dimension of face sets]
A simply-connected subset of faces $S\subset \sF(\cL_{0,n})$ is said to have isoperimetric dimension at most $d$, denoted $\isodim(S) \leq d$, if $|\partial S| \leq |S|^{(d-1)/d}$.
\end{definition*}

For sets $S_n$ with a uniformly bounded $\isodim(S_n)$ we have the following much finer estimates.

\begin{proposition}[{\cite[Prop. 6.2]{GL20}}]\label{prop:max-thick} 
There exist $\beta_0,\kappa_0>0$ such that the following holds for every fixed $\beta>\beta_0$. Let $S_n\subset\cL_{0,n}$ be a sequence of simply-connected sets such that $\isodim(S_n)\leq \sqrt \beta$ and 
 $\lim_{n\to\infty}|S_n|=\infty$. Let $\bW_n = (W_z)_{z\notin S_n}$ be such that $\rho(\bW_n) \subset S_n^c$.
For every $ \sqrt{\log |S_n| } \leq h \leq \frac1{\sqrt\beta}\log |S_n|$,
\begin{align}\label{eq:max-thick-lower-large-h} 
\mu^\mp_{n}\left( \bar M_{S_n}^\downarrow< h \given \bI_{\bW_n}\right) \ge \mu^\mp_{n}\left( \bar M_{S_n}^\downarrow < h \,,\,\cG^\fm_{S_n\setminus S_{n,h}^\circ}
(4h)\,,\, \cG^{\fm}_{S^\circ_{n,h}}
(5h) \mid \bI_{\bW_n}\right) &\geq \exp\left(
- (1+\epsilon_\beta)|S_n| e^{-\alpha_h}\right)\,,
\end{align}
and
\begin{align}\label{eq:max-thick-upper-large-h}
\mu^\mp_{n}\left(\bar M_{S_n}^\downarrow < h\,,\,\cG^{\mathfrak D}_{S_{n,h}^\circ}(e^{\kappa_0 h})  \mid \bI_{\bW_n}\right) &\leq  \exp\left(-(1-\epsilon_\beta) |S_n| e^{-\alpha_h}\right)\,,
\end{align}
where $S_{n,h}^\circ := \{x\in S_n : d(x,\partial S_n) \geq e^{2\kappa_0 h}\}$. Moreover, identical bounds hold if $\bar M_{S_n}^{\downarrow}$ is replaced by $\bar M_{S_n}^{\uparrow}$.
\end{proposition}

\section{Extension of estimates within a ceiling to be conditional on no big walls}\label{sec:extension-estimates-in-ceiling}
In this section we use the observation that the proof of Theorem~\ref{thm:rigidity-inside-wall} is monotone w.r.t.\ wall sizes since it hinges on maps that only delete walls (see~\eqref{eq:Phi-monotonicity} below) in order to elevate Theorem~\ref{thm:rigidity-inside-wall}, and consequently also the lower bounds in Propositions~\ref{prop:max-generic} and~\ref{prop:max-thick}, to apply conditionally on events of the following form:
\begin{equation}\label{eq:def-E-eta}\cE^\eta_{S} := \bigcap_{z\in S} \{ \fm(W_z)\leq \eta_z\}\quad\mbox{for}\quad \eta=\{\eta_z\}_{z\in S}\,.\end{equation}
This extension will be important in the proof of~\eqref{eq:A-supercritical-h} as hinted in Section~\ref{subsec:proof-ideas}. 

The proof of Theorem~\ref{thm:rigidity-inside-wall} in~\cite{GL20} relied on a Peierls-type argument: a map $\Phi$ 
that deletes wall clusters.
Extending it to be conditional on $\cE_S^\eta$ rests on the fact that $\Phi$ clearly satisfies 
\begin{equation}\label{eq:Phi-monotonicity} \fm(W_z(\Phi(\cI)) \leq \fm(W_z(\cI))\qquad\mbox{for every $z$ and every $\cI$}\,,\end{equation}
and so, in particular, if $\cI\in\cE^\eta_S$ for some $\eta$ then necessarily $\Phi(\cI)\in\cE^\eta_S$. We now formalize the extension and include a proof for completeness. 

\begin{theorem} In the setting of Theorem~\ref{thm:rigidity-inside-wall}, for any $\eta=\{\eta_z\}_{z\in A^c}$ we have
\[\mu_n^\mp(\bI_{W_1,\ldots,W_r}\mid\bI_\bW\,,\,\cE^\eta_{A^c}) \leq \exp\Big(-(\beta-C)\sum_{i\leq r}\fm(W_i)\Big)\,.\]
\end{theorem}
\begin{proof}
Consider a map $\Phi= \Phi_{(W_i)_i}$ which takes an interface $\cI\in \bI_{W_1,\ldots,W_r}\cap\bI_\bW \cap \cE^\eta_{A^c}$, deletes $\bigcup_i \Theta_\st \Clust(W_i)$ from its standard wall representation, and from the resulting standard wall collection, constructs the interface $\Phi(\cI)$. On the one hand, the map satisfies a weight gain of (see~\cite[Lemma 3.10]{GL20})
\begin{align*}
    \Big|\log \frac{\mu_n^\mp (\cI)}{\mu_n^\mp(\Phi(\cI))} + \beta \fm(\cI;\Phi(\cI))\Big| \le C\fm(\cI;\Phi(\cI))= C\fm(\bigcup_i \Clust(W_i))\,.
\end{align*}
On the other hand, for fixed $(W_i)_{i}$, the number of sets $\bigcup_i \Clust(W_i)$ is bounded by $\exp(C\fm(\bigcup_i \Clust(W_i)))$ (see~\cite[Lemma 3.13]{GL20}). 
Furthermore, the resulting interface $\Phi(\cI)$ belongs to $\bI_{\bW} \cap \cE_{A^c}^{\eta}$ since it only decreases the wall collection, and only deletes walls confined to $A^c$. If we let $M = \fm(\bigcup_i \Clust(W_i))$ and apply the map $\Phi$ pointwise to $\cI\in \bI_{W_1,\ldots,W_r}\cap \bI_\bW\cap\cE_{A^c}^\eta$, mapping them to interfaces $\cJ \in \bI_\bW\cap \cE_{A^c}^{\eta}$, we obtain
\begin{align*}
    \mu_n^\mp(\bI_{W_1,\ldots,W_r},\bI_\bW, \cE_{A^c}^{\eta}) \le \sum_{\cJ\in \bI_\bW,\cE_{A^c}^\eta} \mu_n^\mp(\cJ) \sum_{k\ge M} e^{ - (\beta - C)k} \le e^{- (\beta - C)M} \mu_n^\mp(\bI_\bW,\cE_{A^c}^\eta)\,.
\end{align*}
Dividing both sides by the probability on the right-hand side then yields the desired conditional probability bound. 
\end{proof}

\begin{corollary}\label{cor:nested-sequence-inside-a-ceiling-E-eta}
In the setting of Corollary~\ref{cor:nested-sequence-inside-a-ceiling}, for any $\eta=\{\eta_z\}_{z\in S_n}$ we have
\[ \mu_n^\mp(\fm(\fW_{x,S_n})\geq r \mid \bI_{\bW_n}\,,\,\cE^\eta_{S_n}) \leq \exp(-(\beta-C)r)\,.\]
\end{corollary}
We stress that the monotonicity in excess energy of $\Phi$ does not carry to the more complicated maps in~\cite{GL19a,GL19b,GL20} where the quantities $\alpha_h$ governing the large deviation rates in Proposition~\ref{prop:max-thick} arise. Still, we are able to extract the following extension of the lower bounds~\eqref{eq:max-lower-crude} and~\eqref{eq:max-thick-lower-large-h} stated in Propositions~\ref{prop:max-generic} and~\ref{prop:max-thick}, which will be needed in Section~\ref{sec:upper-bound}. 
(In what follows, we retain the definition of $S_{n,h}^\circ$ from the latter proposition in terms of $2\kappa_0$ rather than $\kappa_0$---placed there due to the role of $\kappa_0$ in the upper bound~\eqref{eq:max-thick-upper-large-h}, which we do not include here---for the sake of consistency.)
\begin{proposition}\label{prop:max-wall-cluster}
There exist $\beta_0,C,\kappa_0>0$ so the following holds for all $\beta>\beta_0$. Let $S_n\subset\cL_{0,n}$ be a simply-connected set such that $|S_n|\to\infty$ with $n$, and let $\bW_n = \{ W_z : z\notin S_n\}$ be such that $\rho(\bW_n) \subset S_n^c$. \begin{enumerate}
    \item For every $h = h_n \ge 1$ and every $\eta$,
\begin{align} 
\mu^\mp_{n}\left( \bar M_{S_n}^\downarrow< h \given \bI_{\bW_n}\,,\,\cE^\eta_{S_n}\right) &\geq \exp\big(- |S_n| e^{-(4\beta -C)h}\big)\,, \label{eq:max-lower-crude-wallcluster}
\end{align}
\item Further suppose that $\isodim(S_n)\leq \sqrt \beta$. For every $ \sqrt{\log |S_n| } \leq h \leq \frac1{\sqrt\beta}\log |S_n|$ and every $\eta$ satisfying that
     $\eta_z\geq 5h$ for all $z\in S_{n,h}^\circ := \{x\in S_n : d(x,\partial S_n) \geq e^{2\kappa_0 h}\}$, 
\begin{align}\label{eq:max-thick-lower-large-h-wallcluster} 
\mu^\mp_{n}\left( \bar M_{S_n}^\downarrow < h  \mid \bI_{\bW_n}\,,\,\cE^\eta_{S_n}\right) &\geq \exp\left(
- (1+\epsilon_\beta)|S_n| e^{-\alpha_h}\right)\,.
\end{align}
\end{enumerate}
\end{proposition}
\begin{proof}
The proof will follow the same arguments used to establish~\eqref{eq:max-lower-crude}, and~\eqref{eq:max-thick-lower-large-h} in~\cite[Proofs of (6.1) and (6.3)]{GL20}; hence we will only describe the modifications required to adapt it to the conditional space given $\cE_{S_n}^\eta$.

The proof of~\eqref{eq:max-lower-crude} used the fact that $\{\bar M_{S_n}^\downarrow< h\} \supset \{\cG_{S_n}^{\fm}(4h)\}$ (defined in~\eqref{eq:G-event}) and shows that
\[\mu_n^\mp(\cG^{\fm}_{S_n}(4h) \mid \bI_{\bW_n}) \geq \exp(-|S_n|e^{-(4\beta-C)h})\]
by iteratively exposing $\fG_{x,S_n}$ (the sequence of nested walls $\fW_{x,S_n}$ which nest $x$ in $S_n$, as well as the walls nested in this sequence of walls) along a corresponding filtration $\cF_i$.  One uses the fact that 
\begin{align}\label{eq:G-hat-event}
    \cG_{S_n}^{\fm}(4h)= \bigcap_{x\in S_n} \widehat G_x\,, \qquad \mbox{where}\qquad \widehat G_x=  \bigcap_{u\in \rho(\smhull\fG_{x,S_n})} \{\fm(\fW_{u,S_n})<4h\}\,,
\end{align}
and that, when revealing $\fG_{x,S_n}$, either $\diam(\fG_{z_i,S_n})\geq 4h$, which has probability at most $e^{-(4\beta -C)h}$ under $\mu_n^\mp(\cdot\mid\cF_i)$, or there are $O(h^2)$ sites $z\in\rho(\hull\fG_{x,S_n})$, and $\mu_n^\mp(\fm(\fW_{z,S_n})\geq 4h \mid\cF_i)\leq e^{-(4\beta-C)h}$ for any given~$z$, both bounds due to Corollary~\ref{cor:nested-sequence-inside-a-ceiling}. Combined, $\mu_n^\mp(\widehat G_x \mid\cF_i)\geq 1-O(h^2 e^{-(4\beta -C)h} )\geq \exp(-e^{-(4\beta-C')h})$. 
Applying the exact same argument while appealing to Corollary~\ref{cor:nested-sequence-inside-a-ceiling-E-eta} in lieu of Corollary~\ref{cor:nested-sequence-inside-a-ceiling} now shows that~\eqref{eq:max-lower-crude-wallcluster} holds conditionally on $\cE_{S_n}^\eta$.

For the proof of~\eqref{eq:max-thick-lower-large-h}, we first tiled $S_{n,h}^\circ$ by $L\times L$ boxes $Q_i$ with $L\sim \frac{1}{2}e^{2\kappa_0 h}$, with $Q$ denoting their union. We then applied the above analysis to bound the probabilities of $\widehat G_z$ for $z\in S\setminus Q$. Thereafter, breaking $Q_i$ into its bulk $Q'_i\subset Q_i$ and remainder $Q_i\setminus Q'_i$, one gives a lower bound on an event 
\[D_i :=\bigcap_{z\in Q'_i}\widehat H_z \cap \bigcap_{z\in Q_i\setminus Q'_i}\widehat G^{\fm}_z(4h) \]
where $\widehat H_z$ is some event involving the pillar of $z$ in $S_n$ and is a subset of $\{\fm(\fW_{z,S_n})<5h\}$. In light of this fact, $D_i$ is already a subset of the event $\bigcap_{z\in Q_i} \{\fm(\fW_{z,S_n})<5h\}$. Since $\eta_z \ge 5h$ for all $z\in S_{n,h}^\circ$, this means $D_i$ is already a subset of the event $\mathcal E_{S_n}^\eta$, and thus 
$$\mu_n^{\mp}(D_i\mid \cF_i,\cE_{S_n}^\eta) \ge \mu_n^\mp(D_i\mid\cF_i)\,.$$
With this observation, the bounds from the proof in ~\cite{GL20} can be applied as is.
\end{proof}

\section{Basic bounds using the Ising measure without a floor}\label{sec:rough-bounds-no-floor}

The aim of this section is to establish a priori estimates on the number of total wall faces in the interface, the number of wall faces in any collection of $N\le n$ distinct walls, and the number of faces in ceilings that are small, say smaller than $n^{1.9}$, or not thick in that they have isoperimetric dimension much larger than $2$. To prove these we first lower bound the probability that the interface lies above a floor at height $0$, which gives one side of the bound of Proposition~\ref{prop:prob-positive-upper-lower}. 

\subsection{A map to lift the entire interface by \texorpdfstring{$k$}{k}}
We begin by introducing a simple but recurring tool we will use: a map to lift the entirety of the interface up by a height of $k$. Unlike height function models, even this simple map cannot be done in a bijective manner at the level of Ising interfaces, but the multiplicity of the map and the interactions it induces in the interface via the $\g$ term in Theorem~\ref{thm:cluster-expansion} are mild. 

\begin{definition}\label{def:shift-map}
For any $k\ge 1$, define the map $\Phi_{k}^{\uparrow}$ as the following map: for an interface $\cI$, 
\begin{enumerate}
    \item Let $\theta_\uparrow \cI$ be the shift of $\cI$ by the vector $(0,0,k)$;
    \item Let $B_k  = \sF(\partial \cL_{0,n}\times \llb 0,k\rrb)$, and let $\cI' = \theta_\uparrow \cI \oplus B_k$ (where $\oplus$ denotes the symmetric difference);
    \item Let $\Phi_k^{\uparrow}(\cI)$ be given by taking $\cI'$ and removing all finite $*$-connected components of $\cI' \cup (\cL_0 \setminus \cL_{0,n})$. 
\end{enumerate}
\end{definition}

The main use of this map is the following estimate on its effect on the mass of any set of interfaces when passed through $\Phi_k^\uparrow$, showing it only costs roughly $\exp( 4\beta kn)$ to lift the interface by $k$. 

\begin{proposition}\label{prop:shift-map-effect}
    Fix $\beta>\beta_0$. Consider any set of interfaces $A$, and any $k\ge 1$. Then, 
    \begin{align*}
        \frac{\mu_n^{\mp}(A)}{\mu_n^\mp(\Phi_k^\uparrow(A))}\le \exp( 4(\beta+ C) k n)\,.
    \end{align*}
Moreover, this holds under $\mufloor_n^\sfh$ for any $\sfh \ge 0$. 
\end{proposition}

We begin by arguing that the map $\Phi_k^\uparrow$ is well-defined and describing some of its key properties. 

\begin{lemma}\label{lem:shift-map-properties}
For every $k\ge 1$, $\Phi_k^\uparrow$ is well-defined on the set of all interfaces in $\Lambda_n$. For every $\cI$ it has $\fm(\Phi_k^\uparrow(\cI);\cI) \le 4kn$.
Furthermore, if $\cI \in \Ifloor$, then $\Phi_k^{\uparrow}(\cI)\setminus B_k \subset \cL_{\geq k-\sfh}$, and
 for every ceiling $\cC\in \mathfrak C(\cI)$, its shift up by height~$k$ is a subset of a ceiling of $\Phi_k^{\uparrow}(\cI)$.  
\end{lemma}

\begin{proof}
We begin with the well-definedness of $\Phi_k^{\uparrow}$ on the set of all interfaces on $\Lambda_n$. Consider the spin configuration obtained by taking $\sigma(\cI)$ and shifting it by $(0,0,k)$; this is exactly the spin configuration whose set of separating faces are $\cI' \oplus B_k$. Taking that spin configuration and flipping the spins in all \emph{bubbles}, i.e., finite $*$-connected sets of minus or plus spins, then yields a configuration $\Phi_k^{\uparrow}(\sigma(\cI))$, say, whose set of separating faces is exactly $\Phi_k^{\uparrow}(\cI)$. Moreover, that set of separating faces corresponds to an interface per Remark~\ref{rem:interface-spin-config} as the spin configuration has no finite $*$-connected plus or minus components.

Next we turn to the excess energy generated by the map. Notice that the map consists of the addition of at most $4kn$ faces in step (2), and then removes some faces in step (3). As such, we have for every $\cI$ that  $\fm(\Phi_k^\uparrow(\cI);\cI) = |\Phi_k^\uparrow(\cI)| - |\cI| \le 4kn$. 

Turning now to the properties of the range of the map, consider any $\cI \in \Ifloor$. The map $\Phi_k^\uparrow$ raises all faces of $\cI$ by $k$, and only \emph{adds} faces of $B_k$. Namely, $\Phi_k^\uparrow(\cI) \subset \theta_\uparrow \cI \cup B_k$, and $\theta_\uparrow  \cI \in \Ifloor[\sfh-k]$.  

To prove the last property, recall that a horizontal face $f$ is a ceiling face in $\cI$ if and only if all spins of $\sigma(\cI)$ in the column through $f$ are plus below $f$, and minus above $f$ (see~\cite[Observation 4.10]{GL19b}). Consider a ceiling face $f\in \cI$; from earlier in this proof, in the spin configuration $\sigma(\Phi_k^\uparrow(\cI))$, this entire column is shifted up vertically by $k$ (shifting $f$ to $f+(0,0,k)$) and then all bubbles in the resulting spin configuration are deleted. However, no bubbles can intersect the column through $f + (0,0,k)$ because if they did, that would violate the property that all spins below $f+(0,0,k)$ are plus and all spins above are minus. As such, no spins are flipped in this column, and the face $f+(0,0,k)$ must be a ceiling face of $\Phi_k^\uparrow(\cI)$. If the set of ceiling faces of $\cI$ are a subset of those of $\Phi_k^\uparrow(\cI)$, then evidently any connected component of them in $\cI$ will be a subset of a connected component of them in $\Phi_k^\uparrow(\cI)$, concluding the proof.  
\end{proof}

We now turn to the proof of  Proposition~\ref{prop:shift-map-effect}.   

\begin{proof}[\textbf{\emph{Proof of Proposition~\ref{prop:shift-map-effect}}}]
The proof of the proposition consists of taking into account the weight change from application of the map $\Phi_k^\uparrow$, as well as its multiplicity.

We begin with considering the change in the weight of an interface in $A$ when mapped through $\Phi_k^\uparrow$. More precisely, we first show that for every $\cI$ and every $k\ge 1$, we have 
\begin{align}\label{eq:shift-map-weights}
    \Big|\log \frac{\mu_n^{\mp}(\Phi_k^\uparrow(\cI))}{\mu_n^\mp(\cI)} + \beta \fm(\Phi_k^\uparrow(\cI);\cI) \Big| \le C\big(kn + |\fm(\Phi_k^\uparrow(\cI);\cI)|\big)\,.
\end{align}
To show~\eqref{eq:shift-map-weights}, fix any $\cI$ and for ease of notation let $\cJ = \Phi_k^\uparrow(\cI)$. For $f\in \cI$, let $\theta_\uparrow f$ be its vertical shift by $k$. By Theorem~\ref{thm:cluster-expansion}, the left-hand side of~\eqref{eq:shift-map-weights} is at most  
\begin{align*}
    \Big|\sum_{f\in \cI} \g(f,\cI) - \sum_{f'\in \cJ}\g(f',\cJ)\Big| &  \le  \sum_{f: \theta_\uparrow f\in \theta_\uparrow \cI \setminus \cJ} |\g(f,\cI)| + \sum_{f'\in \cJ \setminus \theta_\uparrow \cI} |\g(f',\cJ)| \\
    & \quad + \sum_{f: \theta_\uparrow f\in \theta_\uparrow \cI \cap \cJ} |\g(f,\cI) - \g(\theta_\uparrow f, \cJ)|\,.
\end{align*}
By~\eqref{eq:g-uniform-bound}, the first sum on the right-hand side is at most $\bar K |\theta_{\uparrow} \cI\setminus \cJ|$ and the second sum is at most $\bar K |\cJ \setminus \theta_\uparrow \cI|$, so that together, they contribute at most $\bar K |\theta_\uparrow \cI \oplus \cJ|$. Turning to the last term,  by~\eqref{eq:g-exponential-decay}, and the fact that the radius $\br(f,\cI; \theta_\uparrow f,\cJ)$ is the same as $\br(\theta_\uparrow f;\theta_\uparrow \cI; \theta_\uparrow f,\cJ)$ which must be attained at a face of $\theta_\uparrow \cI \oplus \cJ$,
\begin{align*}
    \sum_{f: \theta_\uparrow f\in \theta_\uparrow \cI \cap \cJ} |\g(f,\cI) - \g(\theta_\uparrow f, \cJ)| \le \sum_{f'\in \theta_\uparrow \cI \cap \cJ} \bar K e^{ - \bar c \br(f';\theta_\uparrow \cI; f',\cJ)} \le \sum_{f'\in \theta_\uparrow \cI \cap \cJ} \sum_{g\in \theta_\uparrow \cI \oplus \cJ} \bar K e^{ - \bar c d(f',g)}\,.
\end{align*}
This is evidently at most $C \bar K |\theta_\uparrow \cI \oplus \cJ|$. To conclude, notice that 
\begin{align*}
      |\theta_\uparrow \cI \oplus \cJ| = |\cJ \setminus \theta_\uparrow \cI| + |\theta_\uparrow \cI \setminus \cJ|\,,  \qquad \mbox{and}\qquad    \fm(\cJ;\cI) = |\cJ \setminus \theta_\uparrow \cI| - |\theta_\uparrow \cI \setminus \cJ|\,.
\end{align*}
Rearranging the latter equality and taking absolute values, we see that $|\theta_\uparrow \cI \setminus \cJ|\le |\cJ\setminus \theta_\uparrow \cI| + |\fm(\cJ;\cI)|$. Since $|\cJ\setminus \theta_\uparrow \cI|\le 4kn$, we then get $|\theta_{\uparrow} \cI \oplus \cJ|\le 8kn + |\fm(\cJ;\cI)|$, yielding~\eqref{eq:shift-map-weights}.

We next consider the multiplicity of the map $\Phi_k^\uparrow$. The key claim here is that for every $M\ge 0$, for every $\cJ$ in the range of $\Phi_k^\uparrow$, we have 
\begin{align}\label{eq:shift-map-multiplicity}
   \big|\big\{\cI \in (\Phi_k^{\uparrow})^{-1}(\cJ): |\fm(\cJ;\cI)| = M\big\}\big| \le  \exp\big(C(kn+M)\big)\,.
\end{align}
To show this, consider the set of faces $\cB = \cI' \setminus \cJ$, i.e., all bubbles deleted in step (3) of Definition~\ref{def:shift-map}. We first claim that this serves as a witness to the pre-image $\cI$, i.e., for every fixed $\cJ$, given $\cB = \cB(\cI)$ one can uniquely reconstruct $\cI$. This is evidently done by noticing that  $\cI' = \cJ \cup \cB$, then taking $\cI' \oplus B_k$, and shifting it by the vector $(0,0,-k)$ to obtain $\cI$. 

Moreover, for an interface $\cI$ having $|\fm(\cJ;\cI)| = M$, we must have $|\cB| \le 4kn + M$ since 
\begin{align*}
    |\cI'|  - |\cI|  \le 4kn\,, \qquad \mbox{and} \qquad \fm(\cJ; \cI) =  |\cI'| - |\cI| - |\cB| \,.
\end{align*}
It therefore suffices for us to enumerate over possible face sets $\cB$ having at most $4kn + M$ many faces. We first note that every $*$-connected component of $\cB$ must intersect $B_k$; indeed since $\theta_{\uparrow}\cI \oplus \cI' \subset B_k$, if a $*$-connected component of $\cB$ doesn't intersect $B_k$, then its shift by $(0,0,-k)$ was a finite $*$-connected component of $\cI$ bounding a bubble, and therefore $\cI$ would not have been an admissible interface. 

In order to enumerate over the choice of $\cB$, we first choose some $N\le M$ for the number of $*$-connected components of $\cB$ and choose a subset of $N$ faces $f_1,\ldots,f_N$ amongst $B_k$ for some representative face of $B_k$ for each  of those $N$ components. We then choose corresponding values $M_1,\ldots,M_N$ such that $\sum M_i \le 4kn + M$ dictating how many faces belong to each of those components, and finally for each $f_i$, we enumerate over the possible $*$-connected components of faces in $\sF(\Lambda_n)$ of size $M_i$ containing $f_i$. In total, this enumeration counts at most 
\begin{align*}
    M\,  \binom{4kn}{M}\, 2^{N+4kn + M}\, C^{4kn+M}
\end{align*}
many such choices, which is at most $C^{4kn+M}$ for some other $C$, establishing~\eqref{eq:shift-map-multiplicity}. 

We are now in position to conclude the proof of the proposition. For a set $A$ of interfaces, we can rewrite 
\begin{align*}
    \mu_n^\mp(A) =  \sum_{\cI\in A}\mu_n^\mp(\cI) = \sum_{\cJ\in \Phi_k^\uparrow(A)} \mu_n^\mp (\cJ) \sum_{-\infty< M\le 4kn} \sum_{\substack{\cI\in A \cap (\Phi_k^\uparrow)^{-1}(\cJ) \\ \fm(\cJ;\cI) = M}} \frac{\mu_n^\mp(\cI)}{\mu_n^\mp(\cJ)}\,.
\end{align*}
Applying~\eqref{eq:shift-map-weights} and~\eqref{eq:shift-map-multiplicity}, we find 
\begin{align*}
    \mu_n^\mp(A) \le \mu_n^\mp(\Phi_k^\uparrow(A))\sum_{-\infty < M \le 4kn} e^{ \beta M + C(kn + |M|)} \le e^{(4\beta +C)kn}\mu_n^\mp(\Phi_k^\uparrow(A))\,.
\end{align*}
Dividing both sides out by $\mu_n^\mp(\Phi_k^\uparrow(A))$ yields the desired result. 
The fact that this holds under $\mufloor_n^\sfh$ follows from the fact that if $A \subset \Ifloor$, then $\Phi_k^\uparrow(A) \subset \Ifloor$ per Lemma~\ref{lem:shift-map-properties}. 
\end{proof}

\subsection{Lower bound on the probability of lying above the floor}
With the map of Definition~\ref{def:shift-map} in hand, we can now prove the following simple lemma giving a lower bound on the probability of the interface under $\mu^\mp_n$ to lie above height $0$. 
This will give the lower bound from Proposition~\ref{prop:prob-positive-upper-lower}, whereas the matching upper bound up to a factor of $1+\epsilon_\beta$ will be a consequence of our estimates in~\S\ref{sec:lower-bound} (namely, Theorem~\ref{thm:lower-bound}). 

\begin{lemma}\label{lem:prob-positive}
There exists $\beta_0>0$ and a sequence $\epsilon_\beta\to 0$ as $\beta\to\infty$ such that, for every $\beta>\beta_0$,
\begin{align*}
    \mu_n^\mp(\Ifloor[0]) \geq \exp( - (1+\epsilon_\beta) n\log n)\,.
\end{align*}
\end{lemma}
\begin{proof}
Let $h_0 = (4\beta - C_0)^{-1}\log n$, where $C_0>0$ is the absolute constant from Proposition~\ref{prop:max-generic}.
Recalling that $\Ifloor[h_0] = \{\cI: M_n^\downarrow\leq h_0\}$, we have by~\eqref{eq:max-lower-crude} (with $S_n=\cL_{0,n}$) that
\[ \mu_n^\mp(\Ifloor[h_0]) \geq \exp(-n^2 e^{-(4\beta-C_0)h_0}) = e^{-n}\,.\]
At the same time, we can compare the weight of $\Ifloor[h_0]$ to that of $\Ifloor[0]$ by application of the map $\Phi_{h_0}^{\uparrow}$ since we have $\Phi_{h_0}^\uparrow(\Ifloor[h_0]) \in \Ifloor[0]$ by Lemma~\ref{lem:shift-map-properties}. In particular, by Proposition~\ref{prop:shift-map-effect}, 
\begin{align*}
    \mu_n^\mp(\Ifloor[h_0]) \le e^{(4\beta + C')h_0 n}\mu_n^\mp(\Ifloor[0])\,.
\end{align*}
Combining this with the aforementioned lower bound of $\mu_n^\mp(\Ifloor[h_0]) \geq e^{ - n}$ concludes the proof.
\end{proof}

\subsection{Wall faces in total and in linearly many walls}
Our first application of Lemma~\ref{lem:prob-positive} will establish that the interface $\cI\sim\mufloor_n^{\sfh}$ must contain at most $e^{-2\beta}n^2$ wall faces w.h.p., as given by the following lemma. (N.B.\ it is easy to infer the weaker upper bound of $(C/\beta) n^2$ on the number of such faces from the cluster expansion representation of Theorem~\ref{thm:cluster-expansion}.
To do so, one compares all such interfaces to the completely flat one (that is, analyze the energy gain and multiplicity loss incurred in the map that deletes every wall), with a multiplicity of $\exp(C(n^2 +m))$ for interfaces with $m$ excess faces competing with a probability gain of $\exp(-(\beta-C) m)$.)
\begin{lemma}\label{lem:n2-excess}
There exist $\beta_0,C>0$ such that, for every $\beta>\beta_0$, 
\[
\mu_n^\mp\Big(\sum\{\fm(W)\,:\;\mbox{$W$ is a wall in $\cI$}\} \geq e^{-2\beta} n^2\Big) \leq \exp(-C e^{-4\beta} n^2)\,.
\]
Consequently, the same bound holds true under $\mufloor_n^{\sfh}$ for any $\sfh\geq 0$.
\end{lemma}
\begin{proof}
Order the faces of $\cL_{0,n}$ using labels $1,\ldots,n^2$ by scanning $\cL_{0,n}$ in a connected manner, row by row. 
Roughly put, we will reveal $\{W_{x}\}_{x\in\cL_{0,n}}$ in this order, skipping any $x$ that belongs to $\rho(\smhull W_{x'})$ for some already revealed $W_{x'}$. Formally:
\begin{enumerate}[\indent1.]
    \item \label{it:exposure-init} Initialize $F=\cL_{0,n}$ as the set of unexplored faces.
    \item \label{it:exposure-step} While $F \neq \emptyset$, repeatedly 
        reveal $W_{x}$ for $x\in F$ having the smallest label in $F$,
        and delete $\rho(\hull W_{x})$ from~$F$.
\end{enumerate} (Note that, in this way, our set of revealed walls remains always connected to $\partial\Lambda_n$.) This process terminates at $F=\emptyset$ with a collection of nonempty walls $W_{x}\neq \varnothing$ whose hull is yet unexplored; processing these walls in an arbitrary order, apply to each such $W_{x}$ the same exposure procedure (Step~\ref{it:exposure-step}) from an initial $F = \rho(\hull W_{x})$.
(As before, the set of revealed walls is always connected to $\partial \Lambda_n$.)
Let $(\cF_j)_{j=1}^{n^2}$ be the corresponding filtration.

If $x_{j}$ is about to be revealed in step $j$ of the process, and $S_{j-1}$ is the connected component of faces that contains $x_{j}$ in  
$\cL_{0,n}\setminus \rho(\bigcup_{k<j} (W_{x_{k}} \cup\{x_{k}\}))$, then the fact that the projections of the walls revealed thus far are connected to $\partial\cL_{0,n}$ supports an application of Corollary~\ref{cor:nested-sequence-inside-a-ceiling}, which implies that
\[ \mu_n^\mp(\fm(W_{x_{j}})\geq r\mid\cF_{j-1}) \leq \mu_n^\mp(\fm(\fW_{x_{j},S_{j-1}})\geq r\mid\cF_{j-1}) \leq e^{-(\beta-C)r}\,.\]
It follows that $\sum_{i} \fm(W_{x_i})$ is is stochastically dominated by $\sum_{i=1}^{n^2}(\xi_i-1)$ where the $\xi_i$'s are i.i.d.\ Geometric($p$) for $p=1-e^{-(4\beta-C)}$, using that every $W_x\neq\varnothing$ has $\fm(W_x)\geq 4$. Thus, bounding $\sum_i \xi_i$ via the correspondence between the negative binomial and binomial distributions,
\[ \mu_n^\mp\left(\sum\{\fm(W):\mbox{$W$ is a wall in $\cI$}\}> \lfloor e^{-2\beta} n^2\rfloor \right) \leq 
 \P\left(\bin(\lfloor(1+e^{-2\beta})n^2\rfloor,p)<n^2\right)\,.
\]
The latter binomial random variable has mean $\mu \in [(1+\frac12e^{-2\beta})n^2,(1+e^{-2\beta})n^2]$ provided $\beta$ is large enough, whence the probability that it is less than $\mu-a$ for $a=\frac12e^{-2\beta}n^2$ is at most $\exp(-\frac12 a^2/\mu)$. This establishes the claimed bound for $\sum_W\fm(W)$ under $\mu_n^\mp$, and the analogous bound under $\mufloor_n^{\sfh}$ follows from Lemma~\ref{lem:prob-positive}.
\end{proof}

The next application of Lemma~\ref{lem:prob-positive} rules out having $n$ distinct walls, each with diameter at least $\log n$.
\begin{lemma}
\label{lem:n-disjoint-walls}
There exist $\beta_0,C>0$ so that the following holds for every $\beta>\beta_0$. Let $\mathfrak{E}_n$ be the event that there exist $N\leq n$ distinct faces $x_1,\ldots,x_N\in\cL_{0,n}$ such that $\fm(\bigcup_{i}\fW_{x_i})\geq (6/\beta) n\log n$ in the interface $\cI$.
Then $\mu_n^\mp(\mathfrak{E}_n)\leq \exp(-(5-\epsilon_\beta) n\log n)$, and consequently, $\mufloor_n^{\sfh}(\mathfrak{E}_n) \leq \exp(-3 n\log n)$.
\end{lemma}
\begin{proof}
Fix a set $A$ of at most $n$ distinct faces in $\cL_{0,n}$. 
We infer from Corollary~\ref{cor:nested-sequence-of-multiple-faces-in-a-ceiling} that
\[ \mu_n^\mp\left(\fm(\mbox{$\bigcup_{x\in A} \fW_{x}$})\geq (6/\beta) n\log n\right) \leq e^{-(\beta-C)(6/\beta)n\log n} \leq e^{-(6-\epsilon_\beta) n\log n}
\]
provided $\beta_0$ (thus also $\beta$) is large enough. Substituting the fact that $\binom{n^2}{|A|}\leq \binom{n^2}n\leq \exp(n\log n+n)$ yields
\begin{align*}
  \mu_n^\mp(\mathfrak{E}_n) &\leq \sum_{N\leq n}\binom{n^2}N e^{-(6-\epsilon_\beta) n\log n}\leq e^{-(5-\epsilon'_\beta) n\log n}\,,
\end{align*}
and the conclusion for $\mufloor_n^{\sfh}$ follows from  Lemma~\ref{lem:prob-positive}.
\end{proof}

\subsection{Ceilings with a sub-linear area or a large isoperimetric dimension} We will next use Lemma~\ref{lem:prob-positive} to rule out an excessive total area in ceilings whose hull satisfies $|\cC| \leq c_\beta n^2/\log^2 n$ (Lemma~\ref{lem:small-ceil}), as well as in ceilings with $|\cC|\geq n $ and $\isodim \geq 4$ (Lemma~\ref{lem:thin-ceil}), by showing the corresponding events have probabilities smaller than $\exp(-(1+\epsilon_\beta) n\log n)$. (N.B. that the threshold $|\cC| = O(n^2/\log^2 n)$ is correct when aiming at an error probability of $\exp(-c n\log n)$ for a total excess area of $c n^2$, as the probability of finding a single such ceiling is some $\exp(-c n/\log n)$, and $c\log^2 n$ many such ceilings  would be needed for a total area of $c n^2$.) 
\begin{lemma}\label{lem:small-ceil}
There exist $\beta_0,C>0$ so that for every fixed $\beta>\beta_0$ and every $1\leq A \leq 2\log_2 n - \sqrt{\log n}$, \[ \mu_n^\mp\left(\sum\big\{|\cC|\,:\; \cC\mbox{ is a ceiling with } 2^{A-1}\leq |\cC|\leq\big(\tfrac{n}{e^\beta\log n}\big)^2 \big\} \geq \tfrac{C}{e^\beta A}n^{2}\right) < e^{-5 n\log n}\,,
\]
and the same bound holds under $\mufloor^\sfh_n$.
\end{lemma}

\begin{proof}
We begin by ruling out the contribution of ceilings $\cC$ with $|\cC|\leq L_1$ for $L_1:= n^2/\log^8 n$ via the bound
\begin{equation}\label{eq:polylog-below-macro} \mu_n^\mp\left(\sum\big\{| \cC|\,:\; \cC\mbox{ is a ceiling with }2^{A-1}\leq |\cC|\leq L_1\big\} \geq \tfrac{C }{e^\beta  A}n^{2}\right) \leq \exp(- \epsilon_\beta n \log^2 n)\,,\end{equation}
for some sequence $\epsilon_\beta$ vanishing as $\beta\to\infty$.
Partitioning the set of ceilings in $\cI$ into sets $\mathfrak{A}_1,\mathfrak{A}_2,\ldots$ given by \[ \mathfrak{A}_k = \{ \cC \,:\; 2^{k-1} \leq |\cC| < 2^k \}\,,\] we will argue that, for a suitable absolute constant $C_0>0$, for each $k=1,\ldots,\lceil\log_2 L_1\rceil$ we have
\begin{equation}\label{eq:polylog-below-macro-per-k} \mu_n^\mp\bigg(\sum_{\cC\in\mathfrak{A}_k}|\cC| \geq (\epsilon_{\beta,k} n)^2 \bigg) \leq \exp\bigg(- (\beta-C)\frac{(\epsilon_{\beta,k}n)^2} {2^{k/2+1}} \bigg)\quad\mbox{for}\quad \epsilon_{\beta,k}:=\frac{C_0}{e^{\beta/2} k}\,.\end{equation}
This, as we will later see, will readily imply~\eqref{eq:polylog-below-macro} by a union bound over $k$. 

To establish~\eqref{eq:polylog-below-macro-per-k}, denote the ceilings of $\mathfrak{A}_k$ by $\cC_1,\ldots,\cC_N$ for some $N$, and suppose that the event under consideration holds, i.e.,
$\sum_{i=1}^N |\cC_i| \geq (\epsilon_{\beta,k} n)^2$. 
Further let
\[ \chi := \frac{\sum_{i=1}^N |\cC_i|}{(\epsilon_{\beta,k}n)^2}\qquad(\;\geq 1\;)\,.\]

It then follows, by the definition of $\mathfrak{A}_k$, that 
\[ N \leq  \chi  (\epsilon_{\beta,k} n)^2 2^{1-k}\,.\]
In addition, for each $i$ we have $|\partial \hull\cC_i| \geq 4 |\hull\cC_i|^{1/2} \geq 4 |\cC_i|^{1/2}$ by the isoperimetric inequality in $\Z^2$, whence again the definition of $\mathfrak{A}_k$ implies that
\[ \sum_{i=1}^N |\partial\hull\cC_i| \geq 
4\sum_{i=1}^N |\cC_i|2^{-k/2} \geq \chi(\epsilon_{\beta,k} n)^2 2^{2-k/2}\,.
\]
For any prescribed sequence of faces $x_1,\ldots,x_N$ such that $x_i\in\rho(\cC_i)$ for each $i$, we can appeal to Corollary~\ref{cor:nested-sequence-of-multiple-faces-in-a-ceiling} to find 
that for any $r \geq N$,
\[ \mu_n^\mp\bigg(\fm\Big(\bigcup_{i\leq N} \fW_{x_i}\Big) \geq r\bigg) \leq \exp(-(\beta-C)r)\,.\]
Since $\fm(\bigcup_{i=1}^N \fW_{x_i})\geq \frac12 \sum_{i=1}^N |\partial\hull\cC_i|$ (each face of $\partial\hull\cC_i$ corresponds to a distinct face in its supporting wall~$W$, which may correspond in this way to a face in at most one other $\partial\hull\cC_j$), we can take $r = \chi (\epsilon_{\beta,k}n)^2 2^{1-k/2}$ (which is at least $N 2^{k/2} \geq N$ by the above bound on $N$) and
conclude that
\begin{equation}\label{eq:ceilings-union-bound} \mu_n^\mp\bigg(\sum_{\cC\in\mathfrak{A}_k} |\cC| \geq (\epsilon_{\beta,k} n)^2\bigg) \leq \sum_\chi \sum_{N \leq  \chi (\epsilon_{\beta,k} n)^2 2^{1-k}} \binom{n^2}{N} \exp\left(-(\beta-C) \chi (\epsilon_{\beta,k} n)^2 2^{1-k/2}\right)\,. 
\end{equation}
Using the bound $\sum_{j\leq p m}\binom{m}j \leq \exp(\mathsf{H}(p)m)$ where $\mathsf{H}(p)=p\log \frac1p+(1-p)\log\frac1{1-p}$ is the binary entropy function (with natural base)---valid for every $m$ and $p\leq \frac12$ 
(when $k\ge 2$, $p\le 1/2$ because $\sum_{i=1}^{N} |\cC_i|\le n^2$, and when $k=1$, $p>1/2$ would violate distinctness of the constituent ceilings)---yields that 
\[ \mu_n^\mp\bigg(\sum_{\cC\in\mathfrak{A}_k} |\cC| \geq (\epsilon_{\beta,k} n)^2\bigg) \leq \sum_\chi \exp\left(
\big(\mathsf{H}( \chi\epsilon_{\beta,k}^2 2^{1-k}) 
-2(\beta-C)  \chi \epsilon_{\beta,k}^2 2^{-k/2}\big)n^2
\right)\,.
\]
Hence, in order to establish~\eqref{eq:polylog-below-macro-per-k} it suffices to show that if the absolute constant $C_0>0$ is large enough then
\begin{equation}\label{eq:entropy-vs-excess-org} \mathsf{H}(\chi \epsilon_{\beta,k}^2 2^{1-k}) \leq \frac32(\beta-C)\chi \epsilon_{\beta,k}^2 2^{-k/2}\qquad\mbox{for every $k\geq 1$}\,,\end{equation}
 noting the sum over at most $n^2$ values of $\chi \in [1,\epsilon_{\beta,k}^{-2}]$ adds a $2\log n$ term to the exponent,  which is readily absorbed in the constant $C$ from~\eqref{eq:polylog-below-macro-per-k} since $ (\epsilon_{\beta,k}n)^2 2^{-(k/2+1)}$ is of order at least $n\log ^2 n$ for all $k\leq \lceil \log_2 L_1\rceil$.
Indeed, using that $\mathsf{H}(p)\leq p(\log\frac1p+1)$ for any $0<p<1$, it suffices to show that, for every $k\geq 1$,
\begin{equation}\label{eq:entropy-vs-excess-gen-eps} 2\log(1/\epsilon_{\beta,k}) + \log(1/\chi)+ (\log 2) (k-1)+1\leq \frac34(\beta-C)2^{k/2}\,.
\end{equation}
By the definition of $\epsilon_{\beta,k}$  and the fact that $\chi\geq 1$, it suffices to show, for every $k\geq 1$,
\begin{equation}\label{eq:entropy-vs-excess} \beta + (\log 2) (k-1)+2\log k + 1 - 2\log C_0\leq \frac34(\beta-C)2^{k/2}\,.
\end{equation}
For every $k\geq 1$ we have $\frac34  2^{k/2} \beta \geq 3 \cdot 2^{-3/2} \beta \geq \frac{21}{20}\beta$; thus, in order to establish~\eqref{eq:entropy-vs-excess-org} it suffices to show that
\[ 
(\log 2) (k-1)+2\log k + 1 - 2\log C_0\leq \frac1{20}(\beta-C)2^{k/2}\,,
\]
and indeed this easily holds for every large enough $k$, whereas selecting a large enough $C_0$ allows us to assume that $k$ is sufficiently large. (We note in passing that the $\beta$ term in the left-hand of~\eqref{eq:entropy-vs-excess} is the crux of our final threshold  $Ce^{-\beta}n^2$ in~\eqref{eq:polylog-below-macro}, which corresponds to $(\epsilon_{\beta,1} n)^2$. A term of $2\beta$ on the left-hand of~\eqref{eq:entropy-vs-excess}, for instance, would violate this inequality for $k=1$ and a large enough $\beta$.)

Having established~\eqref{eq:polylog-below-macro-per-k}, we observe that $\epsilon_{\beta,k}^2 2^{-k/2}$ is decreasing in $k$, whence a union bound over $k$ shows that $\sum_{\cC\in\mathfrak{A}_k}|\cC| \leq (\epsilon_{\beta,k} n)^2$ holds for all $k=1,\ldots,\lceil\log_2 L_1\rceil$ except with probability at most 
\[ \lceil \log_2 L_1\rceil\exp\left(-(\beta-C)C_0^2 e^{-\beta} \lceil\log_2 L_1\rceil^{-2} \tfrac1{2\sqrt 2} L_1^{-1/2} n^2\right) \leq \exp\left(-e^{-\beta} n\log^2 n \right)\,,\]
where the $O(\log L_1)$ prefactor as well as the absolute constants in the exponent were offset by the $\beta-C$ term.
On this event, the fraction of $n^2$ in ceilings with $2^{A-1}\leq |\cC|\leq L_1$ is at most $ \sum_{k\geq A} \epsilon_{\beta,k}^2 \leq C_0^2 e^{-\beta} \sum_{k\geq A}k^{-2}$, thereby giving~\eqref{eq:polylog-below-macro}.

Letting $L_2 := e^{-2\beta} n^2/\log^2 n$, we will treat ceilings with $L_1 \leq \cC \leq L_2$ by
establishing that
\begin{equation}\label{eq:polylog-just-below-macro} \mu_n^\mp\left(\big|\bigcup\big\{\rho(\cC)\,:\; \cC\mbox{ is a ceiling with }L_1\leq |\cC|\leq L_2\big\}\big| \geq C e^{-\beta}n^{2}\right) \leq \exp(- c_0\beta n \log n)\end{equation}
for some absolute constants $C,c_0>0$.
This will follow from showing that, for all $k=\lceil\log_2 L_1 \rceil,\ldots,\lceil\log_2 L_2\rceil$,
\begin{equation}\label{eq:polylog-just-below-macro-per-k} \mu_n^\mp\bigg(\sum_{\cC\in\mathfrak{A}_k}|\cC| \geq (\bar\epsilon_{\beta,k} n)^2 \bigg) \leq \exp\Big(- (\beta-C)\frac{(\bar\epsilon_{\beta,k} n)^2}{ 2^{k/2+1}}\Big)\quad\mbox{for}\quad \bar\epsilon_{\beta,k}:=\frac{e^{-\beta/2}}{\lceil \log_2 L_2\rceil - k + 8}\,.\end{equation}
As argued above~\eqref{eq:polylog-just-below-macro-per-k} will follow once we show the analog of~\eqref{eq:entropy-vs-excess-gen-eps} w.r.t.\ the modified quantity $\bar\epsilon_{\beta,k}$, which here translates into showing that for every $\lceil \log_2 L_1\rceil \leq k \leq \lceil 
\log_2 L_2\rceil$,
\[ \beta + 2\log(\lceil\log_2 L_2\rceil - k + 8) + (\log 2) (k-1)+1 \leq \frac34(\beta-C)2^{k/2}\,.\]
This indeed holds as the left-hand is $O(k) = O(\log n)$ while the right-hand has order $2^{k/2}\geq \sqrt{L_1} = n^{1-o(1)}$ for each such $k$,  and~\eqref{eq:polylog-just-below-macro-per-k} is obtained. To identify the dominant term in the union bound over $k$, note that
\[ \frac{\bar\epsilon_{\beta,k}^2 2^{-k/2}}{\bar\epsilon_{\beta,k+1}^2 2^{-(k+1)/2}}=\sqrt2 \Big(1-\frac1{\lceil\log_2 L_2\rceil -k+8}\Big)^2 \geq \sqrt2 \cdot (\tfrac78)^2 > 1\,,
\]
whence $\sum_{\cC\in\mathfrak{A}_k}|\cC| \leq (\bar\epsilon_{\beta,k} n)^2$ holds for all $k=\lceil\log_2 L_1\rceil,\ldots,\lceil\log_2 L_2\rceil$ except with probability 
\[  \lceil\log_2 L_2\rceil\exp\left(-(\beta-C) \tfrac1{64} e^{-\beta} \tfrac1{2\sqrt 2} L_2^{-1/2} n^2\right) \leq \exp\left(-c_0 \beta n\log n \right)\,,\]
using the definition of $L_2$. This establishes~\eqref{eq:polylog-just-below-macro}, thereby completing the proof.
\end{proof}
\begin{lemma}\label{lem:thin-ceil}
There exist $\beta_0,C>0$ so that for every fixed $\beta>\beta_0$,
\begin{equation}\label{eq:thin-ceil} \mu_n^\mp\bigg(\bigcup\Big\{\rho(\hull \cC)\,:\; \cC\mbox{ is a ceiling with }|\hull\cC|\geq n\mbox{ and }\isodim(\hull\cC)\geq 4\Big\} \geq n^{5/3}\bigg) \leq \exp(- (\beta-C) n^{7/6})\,,\end{equation}	
and the same bound holds under $\mufloor_n^\sfh$.
\end{lemma}
\begin{proof}
Let $\cC_1,\ldots,\cC_N$ be the outermost ceilings satisfying $|\hull\cC|\geq n$ and $\isodim(\hull\cC)\geq 4$. Clearly, $N\leq n$ by the lower bound on the area of each of these pairwise disjoint ceilings.
Furthermore, by the assumption on the isoperimetric dimension, $|\partial \hull\cC_i| \geq |\hull\cC_i|^{3/4}$ for every $i$; thus, whenever $\sum_{i=1}^N |\hull\cC_i| \geq n^{5/3}$ we can conclude that
\[ \sum_i |\partial\hull\cC_i|\geq 
n^{-1/2} \sum_i |\hull\cC_i|\geq n^{7/6}\,.\]
By the same argument the led to~\eqref{eq:ceilings-union-bound}, it now follows that the probability in the left-hand of~\eqref{eq:thin-ceil} is at most
\[\sum_{N \leq n} \binom{n^2}{N} \exp\left(-(\beta-C)n^{7/6}\right)\leq \exp\left(-(\beta-C')n^{7/6}\right)\,,
\]
where we absorbed $\sum_{N\leq n}\binom{n^2}N \leq \exp\big((1+o(1))n\log n\big)$ into the larger constant $C'$, as required.
\end{proof}

\section{Bounding the interface histogram below a given height}\label{sec:lower-bound}
Our objective in this section is to obtain the following bound on the height histogram of the interface in the presence of a floor at height $-\sfh$, capturing the essence of entropic repulsion.

\begin{theorem}\label{thm:lower-gen-h}
Let $L_{n,k}$ be the set of $x\in\cL_{0,n}$ such that $\cI\cap (\{x\}\times\R)$ intersects $\cL_{<h_n^*-\sfh-k}$ or is not a singleton.  
There exist $\beta_0,C>0$ so that for all $\beta>\beta_0$ and every $\sfh<h_n^*-1$ and $k\geq 1$, 
\[ \mufloor_n^{\sfh}\left(|L_{n,k}|\geq Ce^{-\beta}n^2\right) \leq \exp\left(-e^{(\beta-C)k}n\right)\,.\]
\end{theorem}
This will be a consequence of the following estimate on the total area in macroscopic and near macroscopic ceilings below height $h_n^*-\sfh-1$.
\begin{theorem}\label{thm:lower-bound}
There exist $\beta_0,C>0$ so that for every fixed $\beta>\beta_0$, if $\sfh < h_n^* - 1$ and
\[ \fC_{j}(\cI) = \{\cC \mbox{ is a ceiling in $\cI$ with }\hgt(\cC)=j\}
\]
then for every $1\leq k \leq h_n^*-1$, 
\[ \mufloor_n^{\sfh}\bigg(\sum\{|\cC|\,:\; \cC\in\fC_{h_n^*-\sfh-k-1}(\cI)\,, |\cC|\geq  n^{1.9} \} \geq e^{-\beta k}n^{2}\bigg) \leq e^{- \left(e^{(\beta-C)k} \,\wedge\, 3\log n \right)n}\,.
\]
\end{theorem}

\begin{remark}
The threshold of $e^{-\beta k}n^2$ on $\sum|\cC|$ in Theorem~\ref{thm:lower-bound} could be  extended via the same proof into $e^{-a\beta k}n^2$ for any fixed $1<a<2$ (with a bottleneck at $a=2$ due to our reliance on $e^{-a\beta h_n^*} = O(n^{3/2})$, whereas $e^{-2\beta h_n^*}$ is only guaranteed to be $O(n^{3/2+\epsilon_\beta})$ by the best upper bound 
we have for $h_n^*$).
\end{remark}

Observe that the proof of Theorem~\ref{thm:lower-gen-h} follows directly from  Theorem~\ref{thm:lower-bound} along with our results from~\S\ref{sec:rough-bounds-no-floor}:
By Lemma~\ref{lem:small-ceil}, the total area in $\rho(\cC)$ for ceilings $\cC$ with $|\cC|\leq n^2/\log^3 n$ is at most $Ce^{-\beta}n^2$ except with probability $\exp(-5 n \log n)$; thus, for  fixed $\bar k\geq 1$, the size of $\bigcup\{\rho(\cC): \cC \in \fC_{<h_n^*-\sfh-\bar k}\}$ is at most $e^{-\beta \bar k}n^2$ except with probability $\exp(-e^{(\beta-C)\bar k } n)$ via Theorem~\ref{thm:lower-bound} and a union bound over $k=\bar k,\ldots,h_n^*-1$; and there are at most $e^{-2\beta}n^2$ wall faces in $\cI$ except with probability $\exp(-c_\beta n^2)$ by Lemma~\ref{lem:n2-excess}.

The general principle behind the proof is the standard competition between energy and entropy which propels the interface to height $h_n^*-1$. It will be illuminating to consider the case $\sfh=0$ and show, via a straightforward Peierls argument which served as the basis of the sharp results for the SOS model in~\cite{CLMST14,CLMST16}, that the interface in $\mufloor_n^0$ is propelled to height $(1-\epsilon_\beta)h_n^*$. We do so in the next claim, and then explain why one cannot hope for such an argument, albeit effective for SOS, to be applicable for the Ising model.

\begin{claim}\label{clm:lower-bound-1-eps}
There exist $\beta_0,C_0>0$ such that, if $h_0 = \lfloor(4\beta+C_0)^{-1}\log n\rfloor$, then for every $\beta>\beta_0$ and $k\geq 2$,
\[ \mufloor_n^0\bigg(\sum\{|\cC|\,:\; \cC\in\fC_{h_0-k}(\cI) \} \geq e^{-2\beta k}n^{2}\bigg) \leq \exp\left(- e^{\beta k} n\right)\,.
\]
\end{claim}
\begin{proof}
Let $k\geq 1$. For $\cI\in\Ifloor[0]$, let \[A_0(\cI)= \bigcup\{\rho(\cC):\cC\in\fC_{h_0-k-1}(\cI)\}\quad,\quad 
A_1(\cI)= \bigcup\{\rho(\cC):\cC\in\fC_{h_0-k}(\cI)\},\,\] and let \[\bI = \{\cI : |A_0(\cI)|\geq e^{-2\beta k}n^2\}\,.\] 
Further let $\bI' = \Phi_1^\uparrow(\bI)$, where we recall from Definition~\ref{def:shift-map} which effectively shifts the interface up by $1$. In particular,
$|A_1(\cI')|\geq e^{-2\beta k} n^2$ for every $\cI'\in\bI'$.

For $\cI'\in\bI'$ and a subset $A \subset A_1(\cI')$, let $\bJ_{\cI'}$ be the interfaces obtained from $\cI'$ by setting all spins of $\sigma(\cI')$ in $A\times (\frac12+\Z_+)$ to be minus (modifying $h_0-k$ spins above each $x\in A$). As $\cI'\cap((\cL_{0,n}\setminus\partial\cL_{0,n})\times\R)) \subset \cL_{\geq 1}$ for any $\cI\in\Ifloor[0]$, by Lemma~\ref{lem:shift-map-properties}, we can recover $\cI'$ from $\cJ\in\bJ_{\cI'}$ (reading $A$ off from~$\cJ\cap\cL_0$). This implies that the interface sets $\{\bJ_{\cI'}:\cI'\in \bI'\}$ are pairwise disjoint.

 Moreover, if $|A|=m$ then it is easy to verify that for some absolute constant $C>1$,
 $$\mufloor_n^0(\cJ) \geq \exp(-(\beta+C)4(h_0-k) m)\mufloor_n^0(\cI'),,$$ 
since any such~$\cJ$ has $\fm(\cJ;\cI') \le 4(h_0-k)m$ by construction.
 Taking $C_0=4C$ in the definition of $h_0$, this translates into $\mufloor_n^0(\cJ) \geq p^{m}\mufloor_n^0(\cI')$ where
 $p=p_{k,n} = e^{(4\beta +C_0)k}/n$ (noting that $p_{k,n}\leq p_{h_0,n} \leq 1$), and so,
 \begin{align*}1 &\geq \sum_{\cI'\in\bI'}\sum_{\cJ\in\bJ_{\cI'}}\mufloor_n^0(\cJ) \geq  \sum_{\cI'\in\bI'}\mufloor_n^0(\cI') \!\! \sum_{A\subset A_1(\cI')}\!\! p^{|A|} 
 =  \sum_{\cI'\in\bI'} \mufloor_n^0(\cI') \left(1+p\right)^{|A_1(\cI')|}
 \\
 &\geq \exp\left( \tfrac12 p e^{-2\beta k} n^2 \right)\mufloor_n^0(\bI') = \exp(\tfrac12 e^{(2\beta + C_0)k} n )\mufloor_n^0(\bI')\,,
\end{align*}
using $1+p \geq \exp(p/2)$ for $p\leq 1$ in the inequality between the lines. By Proposition~\ref{prop:shift-map-effect} we have  $\mufloor_n^0(\bI) \leq \mufloor_n^0(\bI') \exp(4(\beta+C')n)$ for another $C'>0$, completing the proof provided $\beta_0$ is large enough.
\end{proof}
The proof of Claim~\ref{clm:lower-bound-1-eps} captured the effect of entropic repulsion: if an interface $\cI$ is such that $A_0(\cI)$, the projections of the ceilings at height $j=h_0-k$, is too large, one can compare it to the set of (many) interfaces $\cJ$ obtained by elevating all heights via a single $n\times n$ slab of $(+)$ spins (a cost of $\exp(-(4\beta-C) n)$ in energy) while replacing any subset of $A_0$ by spikes reaching down to height $0$ (a gain of $\exp(e^{-(4\beta-C)j}|A_0|)$ in entropy); comparing the two competing terms, the entropy term wins when $j< h_0$, giving the claim.
Unfortunately, the correct rate $\alpha$ of large deviations in the 3D Ising model exceeding height $j$ is realized not by a deterministic shape such as a spike of height $j$ but rather by a \emph{distribution} over pillars $\cP_x$ that behave as decorated random walks (e.g., typically $\diam(\rho(\cP_x))$ has order $\sqrt j$ for such $\cP_x$). For one to be able to modify~$\cI$ to have such downward pillars $\cP_x$ at a subset $A\subset A_0$, they would need to all fit in the corresponding ceilings and with one another; then, one would need to account for their randomness and interactions among themselves and with existing walls. Even more problematic is the fact that comparing the interface $\cJ$ with a given pillar $\cP_x$ of height $j$ to an interface $\cI$ without it via cluster expansion  (Theorem~\ref{thm:cluster-expansion}) incurs a cost of $\exp(\bar K j)$ due the interaction term $\g$. This would result in an error of $n^{\epsilon_\beta}$---thus not improving on Claim~\ref{clm:lower-bound-1-eps}.

One could instead begin by examining the interfaces $\cI'$ obtained from $\cI$ by elevating all height by $1$ yet without planting any spikes or pillars. As mentioned above, $\mufloor_n^0(\cI')$ is at least $\exp(-4(\beta-C)n)\mufloor_n^0(\cI)$, due to the additional $n\times n$ slab of $(+)$ spins. The latter interfaces have their minima within $A_0\times \R$ not reaching to height $0$, which one would expect to have probability at most $\exp(-e^{-\alpha_h}|A_0|)$ via the estimate~\eqref{eq:max-thick-upper-large-h} in Proposition~\ref{prop:max-thick} (proved via the approximate Domain Markov property for ceilings in~\cite{GL20}). If this probability outweighs the $\exp(-4(\beta-C)n)$ factor, we will obtain the desired lower bound.

However, the event $\cG_{S_{n,h}^\circ}$ in~\eqref{eq:max-thick-upper-large-h} complicates matters, and is a real (rather than a technical) obstacle: Mesoscopic walls nested in a ceiling may encapsulate additional ceilings and modify their heights in a complicated way (whereby a single wall supports multiple ceilings in different heights), making some of them more favorable, by the exact same entropic repulsion mechanism that propels the interface to height $h_n^*-1$. Thus, we cannot preclude the existence of such walls, and must resort to an analysis of their subtle effect.

Our approach for handling the delicate effect of mesoscopic walls involves conditioning on them---possibly revealing large ceilings at different heights in the process---then reevaluating the new landscape of ceilings, and applying the aforementioned entropic repulsion argument for a specific height in this conditional space.

Consider an integer $1 \leq k < h_n^*$. With $\kappa_0>0$ the absolute constant from Proposition~\ref{prop:max-thick}, define 
\begin{equation}\label{eq:meso-wall}
\bW^\dagger(\cI) = \left\{ W \,:\; W \mbox{ is a wall of $\cI$ with }\diam(\rho(W)) \geq  (e^{ \kappa_0 (h_n^*-k)}\,\vee\,\log n)\right\}
\end{equation}
(note the dependence on $k$) and associate to every $\cI$ the interface $\cI^\dagger$ consisting only of these walls, i.e.,
\begin{equation}\label{eq:I-dagger-def}
\cI^\dagger := \cI_{\bW^\dagger(\cI)}\,.
\end{equation}
\begin{remark}
\label{rem:ceilings-of-W-dagger}	
If $W' \Supset W$ for some $W \in \bW^\dagger$ then $W'\in\bW^\dagger$ since in that case $\diam(\rho(W'))> \diam(\rho(W))$. (This would not be the case were our inclusion criterion for $\bW^\dagger$ instead been phrased e.g.\ in terms of $\fm(W)$.) Consequently, every face of a ceiling $\cC\in\cI$ supported by some $W\in\bW^\dagger(\cI)$ is also a ceiling face in $\cI^\dagger$ at the exact same height (whereas $\cC$ may be contained in a strictly larger  ceiling $\cC^\dagger\in \cI^\dagger$ with $\hgt(\cC^\dagger)=\hgt(\cC)$).
\end{remark}
With this in mind, we will later use that if the set of ceilings $\fC_{h_n^*-\sfh-k-1}(\cI)$ with $|\hull\cC|\geq n^{1.9}$ is $\{\cC_1,\ldots,\cC_N\}$, then $\fC_{h_n^*-\sfh-k-1}(\cI^\dagger)$ includes ceilings $\{\cC^\dagger_1,\ldots,\cC^\dagger_{N'}\}$ for some $N' \leq N$ where every $\cC_j$ must have $\cC_j\subset\cC^\dagger_{i_j}$ for some $i_j$.  (N.B.\ that $N'\leq N$ as a single ceiling in $\cI^\dagger$ may include multiple $\cC_{j}$'s, as well as extra horizontal wall faces from $\cI$ which were deleted in the transition to $\cI^\dagger$, but distinct ceilings in $\cI^\dagger$ must correspond to disjoint such sets of ceilings in $\cI$ by definition of a ceiling as a $*$-connected component of ceiling faces.)

The key reduction in the proof of Theorem~\ref{thm:lower-bound} is the next bound on the total area of a certain class of ceilings in $\fC_{h_n^*-\sfh-1}(\Phi_k^\uparrow(\cI)^\dagger)$, which we will thereafter use to bound the total area of ceilings in $\fC_{h_n^*-\sfh-k-1}(\cI)$.

\begin{lemma}
	\label{lem:area-of-ceilings-I-dagger}
There exist absolute constants $\beta_0,C>0$ so that, for all $\beta>\beta_0$ and $0\leq \sfh < h_n^*-1$, the following holds.  Let $1 \leq k < h_n^*$, define $\cI^\dagger$ as in~\eqref{eq:I-dagger-def}, and let
\[ \bC^\dagger(\cI) = \left\{ \cC^\dagger\in\fC_{h_n^*-\sfh-1}(\cI^\dagger)\,:\; |\cC^\dagger|\geq n^{1.9}\,,\,\bar M^\downarrow_{\cC^\dagger}(\cI) < h_n^*-k\right\}\,.
\]
 Then for every sufficiently large enough~$n$ we have
\[ \mufloor_n^{\sfh}\bigg(\sum_{\cC^\dagger\in\bC^\dagger(\cI)} |\cC^\dagger| \geq e^{-\beta k}n^{2}\,,\,\fm(\bW^\dagger(\cI))\leq n\log n \bigg) \leq e^{-\left(e^{(\beta-C)k} \,\wedge\, n^{1/5}\right)n}\,.
\]
Furthermore, with $\gamma:=n e^{-\alpha_{h_n^*}}$ 
this remains true when replacing $e^{(\beta-C)k}$ in the right-hand by $\gamma e^{(3\beta-C)k}$, and moreover, it holds  conditionally on the event $\{\bW^\dagger(\cI)=\bW_n\}$ for any $\bW_n$ such that $\fm(\bW_n)\leq n\log n$.
\end{lemma}
(We emphasize that the definition of $\bC^\dagger$ looked at $\bar M_{\cC^\dagger}$ in $\cI$ as opposed to in $\cI^\dagger$.)

We postpone the proof of this lemma in order to show how to derive Theorem~\ref{thm:lower-bound} from it.
\subsection{Proof of Theorem~\ref{thm:lower-bound} modulo Lemma~\ref{lem:area-of-ceilings-I-dagger}}
Consider some $\cI\in\Ifloor$. Since every $\bW^\dagger(\cI)$ consists of distinct walls, each of which has size at least $\log n$, if $\sum_{W\in\bW^\dagger(\cI)}\fm(W)\geq \frac12n\log n$, then any subset of $|\bW^\dagger(\cI)|\wedge n$ many of those walls also has total $\fm$ at least $\frac12 n\log n$. Thus, by Lemma~\ref{lem:n-disjoint-walls} with $N = |\bW^\dagger(\cI)|\wedge n$,
\[ \mufloor_n^{\sfh}(\sum\{\fm(W)\,:\;W\in\bW^\dagger(\cI)\}\geq \tfrac12 n\log n\}) \leq e^{-3n\log n}\,.\]
It will therefore suffice to show that
\begin{equation}\label{eq:Ik-bound}\mufloor_n^h(\bI_k)\leq \exp(-(e^{(\beta-C)k} n)
\end{equation}
 for
\[ \bI_k =
\bigg\{\cI\in\Ifloor\,:\;\sum_{\cC\in\fC_{h_n^*-\sfh-k-1}(\cI)} |\cC|\one_{\{|\cC|\geq n^{1.9}\}}\geq e^{-\beta k}n^{2}~,~
\fm(\bW^\dagger(\cI))\leq \tfrac12 n\log n\bigg\}\,.\]
(N.B.\ we are aiming to bound $\sum |\cC|$, throughout this theorem, as opposed to $\sum|\hull\cC|$; indeed, the latter can include an array of ceilings at different heights, which would be treated by sets $\bI_k$'s as above for different~$k$'s.) 
For each $\cI\in\bI_k$, define
\[ \tilde\cI = \Phi_k^\uparrow(\cI)\]
where $\Phi_k^\uparrow$ is the map given in Definition~\ref{def:shift-map}, which we recall  translates $\cI$ by $(0,0,k)$, adds a subset of the~$4kn$ wall faces $\sF(\partial\cL_{0,n}\times\llb0,k\rrb)$, and then deletes some other subset of wall faces.
By the definition of $\cI\in\bI_k$,
\begin{equation}\label{eq:bW-tilde-bound}\sum_{W\in\bW^\dagger(\tilde\cI)}\fm(W) \leq \sum_{W\in\bW^\dagger(\cI)}\fm(W) +4kn \leq \tfrac12n\log n + 4kn < n\log n \,,\end{equation}
using here that $4kn \leq (1+\epsilon_\beta)\beta^{-1}n\log n$. 

Denote by $\bC $ the set of every ceiling $\cC\in \fC_{h_n^*-\sfh-k-1}(\cI)$ that satisfies $|\cC|\geq n^{1.9}$, and further let $\tilde\bC$ be the analogous subset of the ceilings $\fC_{h_n^*-\sfh-1}(\tilde\cI)$. As no ceiling face was deleted from $\theta_\uparrow\cI$ in forming~$\tilde\cI$, we see that every $\cC\in\bC$ must satisfy $\rho(\cC)\subset\rho(\tilde\cC)$ for some $\tilde \cC\in\fC_{h_n^*-\sfh-1}(\tilde\cI)$, and so $|\tilde\cC|\geq n^{1.9}$ as well, thus $\tilde\cC\in\tilde\bC$. In particular, if $\cI\in\bI_k$ then $\tilde\cI$ satisfies
\[ \sum_{\tilde\cC\in\tilde\bC}|\tilde\cC| \geq e^{-\beta k}n^2\,.\]

Now, if $\tilde\cC\in\tilde\bC$ then its supporting wall $W$ 
must satisfy $\diam(\rho(W))\geq \diam(\tilde\cC)\geq \frac12|\tilde\cC|^{1/2} \geq \frac12 n^{0.95}$, implying that $W\in \bW^\dagger(\tilde\cI)$ for every sufficiently large $n$ (as we have $\exp(\kappa_0(h_n^*-k))\leq n^{\epsilon_\beta}$ whereas $\epsilon_\beta < \frac12$, say, provided $\beta_0$ is suitably large). In particular (invoking Remark~\ref{rem:ceilings-of-W-dagger}), every face in $\tilde\cC\in\tilde\bC$ must be part of some ceiling $\cC^\dagger\in\tilde\cI^\dagger$ at the same height of $h_n^*-\sfh-1$ (which, once again, satisfies $|\cC^\dagger| \geq |\tilde\cC|\geq n^{1.9}$).
Moreover, by Lemma~\ref{lem:shift-map-properties} we know that $\tilde\cI\cap ((\cL_{0,n}\setminus\partial\cL_{0,n})\times\R) \subset \cL_{\geq k-\sfh}$; thus, for every ceiling $\cC^\dagger\in\fC_{h_n^*-\sfh-1}(\tilde\cI^\dagger)$,
\[ M_{\cC^\dagger}^\downarrow(\tilde\cI) \leq (h_n^*-\sfh-1)-(k-\sfh) = h_n^*-k-1\,,\]
which, when combined, imply that $\cC^\dagger\in\bC^\dagger(\tilde\cI)$ as defined in Lemma~\ref{lem:area-of-ceilings-I-dagger}. In conclusion,
\[ \sum_{\cC^\dagger\in\bC^\dagger(\tilde\cI)}|\cC^\dagger| \geq e^{-\beta k} n^2\qquad\mbox{and}\qquad \fm(\bW^\dagger(\tilde\cI))< n\log n\]
holds for all $\tilde\cI\in\Phi_k^\uparrow(\bI_k)$.
We may thus bound $\mufloor_n^\sfh(\Phi_k^\uparrow(\bI_k))$ via
Lemma~\ref{lem:area-of-ceilings-I-dagger}, and combine it with  Proposition~\ref{prop:shift-map-effect} to get that
\[ \mufloor_n^{\sfh}(\bI_k) \leq e^{(4\beta+C)k n}\, \mufloor_n^{\sfh}(\Phi_k^\uparrow(\bI_k)) \leq 
 \exp\left[ \left((4\beta+C)k -\big(\gamma e^{(3\beta-C)k}\wedge n^{1/5}\big)\right) n
\right]\,.
\]
Since $k = O(\log n)$, the case where $\gamma e^{(3\beta-C)k}\geq n^{1/5}$ immediately leads to $\mufloor_n^{\sfh}(\bI_k)\leq \exp[-(1-o(1))n^{6/5}]$, thus it remains the treat the converse case.
Recalling that $\gamma $ satisfies 
\[ \exp(-2\beta-e^{-4\beta})\leq \gamma <\exp(2\beta)\,, \]
we see that $\gamma e^{(3\beta-C)k}\geq e^{(\beta-C)k}$ (with room to spare: we could have taken $C'\exp[\beta+(3\beta-C)(k-1)]$), so
\[
\mufloor_n^{\sfh}(\bI_k) \leq \exp\left[(4\beta+C)kn - e^{(\beta-C)k}n\right] \leq \exp\left[-e^{(\beta-C')k}n\right]\,,
\]
where replacing the constant $C>0$ by some larger absolute constant $C'>0$ in the last transition allowed us to absorb the term $(4\beta+C)kn$. We have thus established~\eqref{eq:Ik-bound}, as required.
\qed

\subsection{Proof of Lemma~\ref{lem:area-of-ceilings-I-dagger}}
Let $\cI\sim\mufloor_n^{\sfh}$, and reveal the wall collection $\bW^\dagger(\cI)$, which by our assumption satisfies $\fm(\bW^\dagger(\cI))\leq n\log n$. We would like to obtain an upper bound on the probability that there are $\epsilon_\beta n^2$ faces in ceilings of $\bC^\dagger(\cI)$ (for an appropriate $\epsilon_\beta$). As our large deviation estimates in Proposition~\ref{prop:max-thick} are only applicable to simply-connected regions, extra care must be taken to modify the ceilings of $\cI^\dagger$ into a simply-connected subset. To do so, we will expose a minimal additional collection of walls, indexed by $\cP^\dagger$ below.

Towards that purpose, define  the random variable $\cP^\dagger$, measurable given $\bW^\dagger(\cI)$, to be
\[ \cP^\dagger=\cP^\dagger(\cI) := \mbox{minimum face subset of $\cL_{0,n}$ such that $\cP^\dagger(\cI) \cup \rho(\bW^\dagger(\cI)) \cup \partial\cL_{0,n}$ is $*$-connected}\]
and the collection $\cS^\dagger_\sfh$ (also measurable given $\bW^\dagger(\cI)$, as the latter determines $\cI^\dagger$) to be
\[ \cS^\dagger_\sfh = \cS^\dagger_
\sfh(\cI) := \left\{ S^\dagger = \rho(\cC^\dagger)\setminus\cP^\dagger\,:\; \cC^\dagger\in\fC_{h_n^*-\sfh-1}(\cI^\dagger)~,~
|\cC^\dagger|\geq  n^{1.9}\right\}\,.\]
(For $\cP^\dagger(\cI)$ to be well-defined, if there exist multiple face subsets of $\cL_{0,n}$, all with the same  cardinality, that turn $\rho(\bW^\dagger(\cI))\cup\partial\cL_{0,n}$ to be $*$-connected, choose the identity of $\cP^\dagger(\cI)$ according to a predefined arbitrary lexicographic ordering of all subsets of $\cL_{0,n}$.)
We will show that
\begin{equation}\label{eq:S-dagger-bound}  \mufloor_n^{\sfh}\bigg(\sum_{S^\dagger\in\cS^\dagger_\sfh} |S^\dagger|\one_{\{\bar M^\downarrow_{S^\dagger}(\cI) < h_n^*-k\}} \geq e^{-\beta k}n^{2} \;\Big|\; \bW^\dagger(\cI)\bigg) \one_{\{\fm(\bW^\dagger(\cI))\leq n\log n\}}\leq e^{-\left(\gamma e^{(3\beta-C)k} \,\wedge\, n^{1/5}\right)n}\,,
\end{equation}
where here we used $\bar M^\downarrow_{S^\dagger}$ to denote the minimum height of $\cI$ within $S^\dagger$ after centering it by 
$\hgt(\cC^\dagger)=h_n^*-\sfh-1$. 

It is important to stress that for every ceiling $\cC^\dagger$ in $\cI^\dagger$, the only obstacles preventing it from being simply-connected would be walls from $\bW^\dagger(\cI)$ nested in it; hence, by the minimality of $\cP^\dagger$,
\[
\rho(\cC^\dagger)\setminus\cP^\dagger\mbox{ is simply-connected for every ceiling $\cC^\dagger$ in $\cI^\dagger$}\,.
\]
(If $\cC^\dagger$ is supported by $W\in\bW^\dagger(\cI)$ and $P=\cP^\dagger\cap\hull\cC^\lowdag$ connects  $x, y\in W$ and  some $W'\in\bW^\dagger(\cI)$ nested in~$\cC^\dagger$, one can delete the part of~$P$ connecting~$x$ to $W$, retaining connectivity of $W,W'$ via the other part of~$P$.)

Recall that each $\cC^\dagger$ has $|\cC^\dagger|\geq n^{1.9}$ and $|\partial\cC^\dagger| \leq n\log n$ (the latter due to our assumption on $\bW^\dagger(\cI)$).
To address the effect of subtracting the set $\cP^\dagger$ from the ceilings $\cC^\dagger$ of $\cI^\dagger$, we will appeal to the following classical result concerning Euclidean Traveling Salesman Problems (TSPs). (For the $d$-dimensional analog of this result, see, e.g., the references in~\cite[\S2]{Steele90}, as well as Lemma~1 there.) 

\begin{fact}\label{fact:tsp}[{\cite{Verblunsky51}}; see also, e.g.,~{\cite[\S2]{Steele90}}]
For every $m$ points $x_1,\ldots,x_m$ in $[0,1]^2$ there exists a path connecting them of length at most $\sqrt{2.8 m}+2$.
\end{fact}

For completeness, we include the short classical argument giving the bound $2\lceil\sqrt m\rceil$: one divides $[0,1]^2$ into strips of height $1/\sqrt m$ each, and proceeds from the bottom strip to the top via switchbacks, connecting the points within each strip from left-to-right, then right-to-left, etc. If we were to use only horizontal and vertical lines (as an upper bound), and there are $k_i$ points in strip $i$, then the total length associated with strip $i$ would be at most $1 + k_i /\sqrt m$, whence the overall total length is at most $\lceil\sqrt m\rceil + \sum k_i /\sqrt m \leq 2\lceil \sqrt m\rceil $.

By our assumption that $\fm(\bW^\dagger(\cI))\leq n\log n$, we can apply Fact~\ref{fact:tsp} to the faces of $\rho(\bW^\dagger(\cI))\cup \partial\cL_{0,n}$ with $m = O(n\log n)$, rescaling $[0,1]^2$ to $\cL_{0,n}$. This yields that there exists a face subset $\cP^\dagger \subset \cL_{0,n}$ connecting $\rho(\bW^\dagger(\cI))\cup \partial\cL_{0,n}$ of size
\[
|\cP^\dagger| = O(n\sqrt{m}) = O( n^{3/2}\sqrt{\log n}) < n^{1.6}
\]
for large enough $n$. In the first equality, the extra $n$ came from rescaling of $[0,1]^2$ to $\cL_{0,n}$. Also the discretization effect of faces amounts to at most $2m=O(n\log n)$ additional faces at the endpoints of each of the paths.

With this bound in hand, we may proceed to show~\eqref{eq:S-dagger-bound}. Write $\bC^\dagger(\cI)=\{\cC^\dagger_1,\ldots,\cC^\dagger_N\}$ for some $N$, as well as $S^\dagger_i = \rho(\cC_i^\dagger)\setminus \cP^\dagger$, and note that we may assume that
\begin{equation}\label{eq:sum-Si-hypothesis} \sum_{i=1}^N |S_i^\dagger| \geq e^{-\beta k} n^2\,,\end{equation}
or else there is nothing left to show. 
For each $z\in \cP^\dagger$, let $U_z$ be $\rho(\cC^\dagger)$ for the ceiling $\cC^\dagger$ in $\cI^\dagger$ with $z\in\rho(\cC_i^\dagger)$, and reveal 
 the walls in $\fG_{z,U_z}$ (where we recall $\fG_{x,S}$ is the sequence of nested walls $\fW_{x,S}$ which nest $x$ in $S$, as well as the every wall nested in this sequence of walls), denoting the set of all walls revealed so far by~$\bW_0$.
For every $i=1,\ldots,N$, 
let
\[ \tilde S_i = S_i^\dagger \setminus \bigcup_{z\in\cP^\dagger} \rho(\smhull\fG_{z,U_z})\,,\]
which is simply-connected since $\hull S_i^\lowdag=S_i^\dagger$ (and $\tilde S_i $ was obtained by deleting from $S_i^\dagger$, at the very most, a collection of simply-connected regions each adjacent to its boundary).
Observe that every $W$ revealed as part of $\fG_{z,U_z}$ for $z\in\cP^\dagger$ necessarily has $\diam(\rho(W)) \leq \exp(\kappa_0 (h_n^*-1-k))\vee \log n \leq \exp(\kappa_0 h_n^*)\vee\log n \leq n^{\epsilon_\beta}$ (otherwise it would have been part of $\bW^\dagger(\cI)$). 
Thus, when obtaining $\tilde S_i$ from $S^\dagger_i$, we subtracted at most $|\cP^\dagger|n^{\epsilon_\beta} \leq n^{1.6+\epsilon_\beta}$ from its area, and similarly added at most $n^{1.6+\epsilon_\beta}$ to its perimeter.
Recalling that $S_i^\dagger$ itself was obtained by deleting $\cP^\dagger$ from $\cC_i^\dagger$ which had $|\cC_i^\dagger|\geq n^{1.9}$ and $|\partial\cC_i^\dagger| \leq n\log n$, it follows that
\[  |\tilde S_i| \geq |\cC_i^\dagger| - |\cP^\dagger|n^{\epsilon_\beta} \geq n^{1.9} - n^{1.6+\epsilon_\beta} \geq \frac12 n^{1.9}\,,\]
and
\[ |\partial \tilde S_i | \leq |\partial \cC^\dagger_i| + n^{1.6+\epsilon_\beta} \leq (1+o(1))n^{1.6+\epsilon_\beta}\,.\]
Combining these inequalities, while recalling that $\tilde S_i$ is simply-connected, shows that for large enough $n$,
\[\isodim(\tilde S_i)\leq 7\] 
provided that $\beta_0$ is large enough (whence $\epsilon_\beta$ becomes small enough).
Finally, note that 
\begin{align}\label{eq:sum-Si-minus-tilde-Si}
\sum_{i=1}^N |S^\dagger_i \setminus \tilde S_i| \leq |\cP^\dagger|(e^{\kappa_0 h_n^*}\vee\log n) \leq n^{1.6+\epsilon_\beta} < \frac12 e^{-\beta k}n^2\,,
\end{align}
where the last inequality is by the fact that $e^{-\beta k}n^2 \geq n^{7/4-\epsilon_\beta}$ for large $n$, since 
$h_n^* \leq \frac1{4\beta - C} \log n + C'$ and so
$e^{-\beta k} \geq e^{-\beta h_n^*} \geq n^{-1/4-\epsilon_\beta}$. 
Combining this with the (deterministic) fact that
\[ \sum_{i=1}^N |S_i^\dagger|\one_{\{\bar M^\downarrow_{S_i^\dagger}<h_n^*-k\}} \leq 
 \sum_{i=1}^N \Big(|\tilde S_i|\one_{\{\bar M^\downarrow_{\tilde S_i}<h_n^*-k\}} + |S^\dagger_i\setminus \tilde S_i|\Big)\,,
\]
it therefore suffices to show that
\[ \mufloor_n^{\sfh}(\mathfrak{E}\mid\bW_0) \leq e^{-(\gamma e^{(3\beta-C)k}\wedge n^{1/5})n}\qquad\mbox{where}\qquad \mathfrak{E} = \Big\{\sum_{i=1}^N |\tilde S_i|\one_{\{\bar M^\downarrow_{\tilde S_i}(\cI)<h_n^*-k\}}\geq  \frac12 e^{-\beta k}n^2\Big\}\,.\]
To this end, we now wish to successively reveal the walls in $\tilde S_i$.
Let $\cF_i$ be the filtration corresponding to initially revealing $\bW_0$, then proceeding to reveal $\bW_i = \{W_z: z\in \tilde S_i\}$ at step $i=1,\ldots,N$. Recalling the definition of the event $\cG$ playing a role in Proposition~\ref{prop:max-thick}, given by
\[ \cG_A(r) = \bigcap_{z\in A} \{ \diam(\rho(W_z)) \leq r\vee\log n\})\,,\]
we stress that every $z\in\cL_{0,n}$ whose $W_z$ was not revealed as part of $\bW_0$ is conditioned to satisfy $\cG_z(e^{\kappa_0\hbar})$ for
\[ \hbar = h_n^*-k\,.\]
By taking a supremum at every step $i$ over the walls outside $\tilde S_i$, it therefore remains to bound the conditional probability that the event $\mathfrak{E}$ occurs given on $\bW_0$, via
\begin{align}\label{eq:mufloor-Psi-bound}
  \mufloor_n^{\sfh}(\mathfrak{E}\mid \bW_0) \leq \prod_{i=1}^N \mufloor_n^{\sfh}( \bar M_{\tilde S_i}^\downarrow<\hbar \mid\cG_{\cL_{0,n}\setminus\{\rho(W):W\in\bW^\dagger\}}(e^{\kappa_0 \hbar})\,,\, \cF_{i-1}) \leq \prod_{i=1}^N \Psi_i
\end{align} 
for
\[ \Psi_i := \sup_{\substack{\bW=\{W_z:z\notin \tilde S_i\} \\ \rho(\{\hull W : W\in\bW\})\subset \tilde S_i}}\mufloor_n^{\sfh}\left(\bar M_{\tilde S_i}^\downarrow < \hbar \mid \cG_{\tilde S_i}(e^{\kappa_0 \hbar})\,,\bI_{\bW} \right)\,.
\]
Recall that if
$S_i^\dagger = \rho(\cC_i^\dagger)\setminus \cP^\dagger$ for some ceiling $\cC_i^\dagger \in \fC_{h_n^*-\sfh-1}(\cI^\dagger)$ then by definition 
$\hgt(\cC_i^\dagger)=h_n^*-\sfh-1$. Therefore, when looking at $\Psi_i$, the implicit conditioning in $\mufloor_n^{\sfh}$ that $\cI\in\Ifloor$ amounts to having $ \bar M_{\tilde S_i}^\downarrow < h_n^*$ via~$\mu_n^\mp$.
With this in mind, for every $\bW=\{W_z \,:\;z\notin \tilde S_i\}$ with $ \rho(\{\hull W : W\in\bW\}\subset \tilde S_i$ we have
\begin{align}
     \mufloor_n^{\sfh}\left(\bar M_{\tilde S_i}^\downarrow < \hbar \mid \cG_{\tilde S_i}(e^{\kappa_0 \hbar})\,,\,\bI_{\bW} \right)
& =
\mu_n^\mp\left(\bar M_{\tilde S_i}^\downarrow < \hbar \mid \bar M_{\tilde S_i}^\downarrow < h_n^*\,,\,\cG_{\tilde S_i}(e^{\kappa_0 \hbar})\,,\,\bI_{\bW} \right) \nonumber\\
&=
\frac{ 
\mu_n^\mp\left(\bar M_{\tilde S_i}^\downarrow < \hbar\,,\, \cG_{\tilde S_i}(e^{\kappa_0 \hbar})\mid\bI_{\bW} \right)
}{
\mu_n^\mp\left(\bar M_{\tilde S_i}^\downarrow < h_n^* \,,\, \cG_{\tilde S_i}(e^{\kappa_0 \hbar})\mid\bI_{\bW} \right)
}\,.  \label{eq:Psi-rewritten}
\end{align}

First consider $1\leq k \leq h_n^*-1-\sqrt{2\log n}$, so that $\sqrt{2\log n} \leq \hbar \leq h_n^*-2$, whence the criterion for $\cG_z(e^{\kappa_0\hbar})$  is dominated by the $\exp(\kappa_0 \hbar)$ term (exceeding, e.g., $\log^2 n$ for every large~$n$).
As $\tilde S_i$ is simply-connected, has $\isodim(\tilde S_i)\leq 7$ and $\hbar \geq \sqrt{2\log n} \geq (\log |\tilde S_i|)^{1/2}$,  we may use~\eqref{eq:max-thick-upper-large-h} from Proposition~\ref{prop:max-thick} for $h=\hbar$, yielding
\[ \mu_n^\mp(\bar M^\downarrow_{\tilde S_i}<\hbar\,,\,\cG_{\tilde S_i}(e^{\kappa_0 \hbar})\mid\bI_{\bW}) \leq \exp(-(1-\epsilon_\beta)|\tilde S_i|e^{-\alpha_\hbar})\,,\]
where we used that the event $\cG_{\tilde S_i}(e^{\kappa_0 \hbar})$ implies  $\cG_{\tilde S_i^\circ}(e^{\kappa_0 \hbar})$ for the subset $\tilde S_i^\circ\subset \tilde S_i$ specified in that proposition.
At the same time, invoking~\eqref{eq:max-thick-lower-large-h} from the same proposition, this time for $h=h_n^*$, we get
\[ \mu_n^\mp(\bar M^\downarrow_{\tilde S_i}<h_n^*\,,\,\cG_{\tilde S_i}(e^{\kappa_0 \hbar})\mid\bI_{\bW}) \geq \exp(-(1+\epsilon_\beta)|\tilde S_i|e^{-\alpha_{h_n^*}})\,,\]
here using that $\cG_{\tilde S_i}(e^{\kappa_0 \hbar})$ is implied by $\cG^{\fm}_{\tilde S_i}(5 h_n^*)$ (and hence also by $\cG^{\fm}_{\tilde S_i}(4 h_n^*)$) since $e^{\kappa_0 \hbar}> \log^2 n > 5 h_n^*$ for every sufficiently large $n$.
Substituting the last two displays in~\eqref{eq:Psi-rewritten}, we find that
\[ \Psi_i \leq \exp\left(-\left ((1-\epsilon_\beta)e^{-\alpha_\hbar}-(1+\epsilon_\beta)e^{-\alpha_{h_n^*}}\right)|\tilde S_i| \right)\,.
\]
We know from~\cite[Corollary~5.2]{GL19b} that
\[ \alpha_{\hbar}+\alpha_{k}-\epsilon_\beta \leq \alpha_{\hbar+k} \qquad\mbox{and}\qquad\alpha_k \geq (4\beta-C)k \,,\] 
which, since $\hbar+k=h_n^*$ and we defined
 $\gamma := n \exp(-\alpha_{h_n^*})$, 
implies that
\begin{align*} \frac{\gamma}n = e^{-\alpha_{h_n^*}} \leq e^{-\alpha_k + \epsilon_\beta} e^{-\alpha_\hbar}\leq e^{-(4\beta-C)k} e^{-\alpha_\hbar}\,.
\end{align*}
Hence, we may absorb the term $(1+\epsilon_\beta)e^{-\alpha_{h_n^*}}=(1+\epsilon_\beta)\gamma/n$ from the above upper bound on $\Psi_i$ into the constant $C>0$ from the lower bound $e^{-\alpha_{\hbar}}\geq e^{(4\beta-C)k}\gamma/n$ (with $k\geq 1$ and $\beta_0$ large) and obtain that 
\begin{equation}\label{eq:Psi-i-upper}
\Psi_i \leq \exp\left(-\gamma e^{(4\beta-C) k} |\tilde S_i|/n\right)\,;
\end{equation}
thus, by plugging in the fact that 
\begin{equation}\label{eq:sum-tilde-Si-hypothesis}
\sum_i |\tilde S_i|\geq \frac12 e^{-\beta k}n^2\,,
\end{equation} obtained from~\eqref{eq:sum-Si-hypothesis} and~\eqref{eq:sum-Si-minus-tilde-Si}, we infer that
\begin{align}\label{eq:Psi-bound-large-h} 
\prod_{i=1}^N \Psi_i &\leq \exp\bigg(-\gamma e^{(4\beta-C)k}\sum_{i=1}^N \frac{|\tilde S_i|}n\bigg) \leq 
\exp\bigg(-\frac12\gamma e^{(3\beta-C) k} n\bigg)\,.
\end{align}

Next, consider $\frac13 h_n^* \leq k \leq h_n^*-1$. Treating this  regime is significantly easier. (The overlap between the regimes is indicative: the latter can be used to yield a lower bound of $(1-\epsilon_\beta)h_n^*$ on the typical height of~$\cI$.)
With the same starting point~\eqref{eq:Psi-rewritten}, we invoke~\eqref{eq:max-upper-crude} from Proposition~\ref{prop:max-generic} to deduce that
\[ \mu_n^\mp(\bar M^\downarrow_{\tilde S_i}<\hbar\,,\,\cG_{\tilde S_i}(e^{\kappa_0 \hbar})\mid\bI_{W}) \leq
\mu_n^\mp(\bar M^\downarrow_{\tilde S_i}<\hbar\mid\bI_{W}) 
\leq \exp(-|\tilde S_i|e^{-(4\beta-C)\hbar})\,,\]
whereas~\eqref{eq:max-lower-crude} from Proposition~\ref{prop:max-generic} implies that
\[ \mu_n^\mp(\bar M^\downarrow_{\tilde S_i}<h_n^*\,,\,\cG_{\tilde S_i}(e^{\kappa_0 \hbar})\mid\bI_{W}) 
\geq \mu_n^\mp(\cG_{\tilde S_i}(4 h_n^*)\mid\bI_{W})
\geq \exp(-|\tilde S_i|e^{-(4\beta+C)h_n^*})\,,\]
since $(e^{\kappa_0 \hbar}\vee \log n)\geq (4h^*_n \vee \log n)= \log n$. Combining these estimates, and using that $\hbar = h_n^*-k$, yields
\[ \Psi_i \leq \exp\left( -|\tilde S_i| e^{(4\beta - C)k - 2C h_n^*}\right) \leq  
\exp\left( - e^{(4\beta - C)k}|\tilde S_i| n^{-1-\epsilon_\beta}\right)\,.
\]
Together with the bound~\eqref{eq:sum-tilde-Si-hypothesis} and the fact that $k\geq \frac13 h_n^* $, this shows that
\begin{align}\label{eq:Psi-bound-small-h} 
 \prod_{i=1}^N \Psi_i \leq \exp\left( -\tfrac14 e^{(3\beta-C)k} n^{1-\epsilon_\beta}\right) \leq \exp\left(-c n^{5/4-\epsilon_\beta}\right)\,.\end{align}

Combining~\eqref{eq:Psi-bound-large-h} and~\eqref{eq:Psi-bound-small-h} with~\eqref{eq:mufloor-Psi-bound} concludes the proof.
\qed

\begin{remark}
\label{rem:h=hn-1-lower}
One can extend Lemma~\ref{lem:area-of-ceilings-I-dagger}---followed by 
Theorem~\ref{thm:lower-bound}, and in turn, Theorem~\ref{thm:lower-gen-h}---to the case $\sfh=h_n^*-1$ if either $k\geq 2$ or the quantity $\gamma = n \exp(-\alpha_{h_n^*})$---which we recall satisfies $
e^{-2\beta - \epsilon_\beta} \leq \gamma <  e^{2\beta}$---is in a given range, for a suitably modified $\epsilon_\beta^k$ replacing $e^{-\beta k}$.
Consider for instance the regime where
\[ \gamma > \beta^2\qquad\mbox{or}\qquad k\geq 2\,.\]
One then defines $\bC^\dagger(\cI)$ from Lemma~\ref{lem:area-of-ceilings-I-dagger} to be 
\[ \bC^\dagger(\cI) = \left\{ \cC^\dagger\in\fC_{h_n^*-\sfh}(\cI^\dagger)\,:\; |\cC^\dagger|\geq n^{1.9}\,,\,\bar M^\downarrow_{\cC^\dagger}(\cI) \leq h_n^*-k\right\}\,,
\]
and argues that, in place of its conclusion, one has
\[ \mufloor_n^{\sfh}\bigg(\sum_{\cC^\dagger\in\bC^\dagger(\cI)} |\cC^\dagger| \geq c_0\beta^{-k}n^{2}\,,\,\fm(\bW^\dagger(\cI))\leq n\log n\bigg) \leq\exp\Big(-\big(5\gamma \beta^{-k} e^{(4\beta-C)(k-1)} \,\wedge\, n^{1/5}\big)n\Big)\,.
\]
Indeed, this follows by showing that, in lieu of~\eqref{eq:S-dagger-bound}, one has
\begin{equation}\label{eq:S-dagger-h=hn-1}  \mufloor_n^{\sfh}\bigg(\sum_{S^\dagger\in\cS^\dagger_\sfh} |S^\dagger|\one_{\{\bar M^\downarrow_{S^\dagger}(\cI) \leq h_n^*-k\}} \geq e^{-\beta k}n^{2} \;\Big|\; \bW^\dagger(\cI)\bigg) \one_{\{\fm(\bW^\dagger(\cI))\leq n\log n\}}\leq e^{-\left(5\gamma \beta^{-k} e^{(4\beta-C)(k-1)} \,\wedge\, n^{1/5}\right)n}\,.
\end{equation}
Following the same argument used to prove Lemma~\ref{lem:area-of-ceilings-I-dagger} yet with $\hbar = h_n^*+1-k$, instead of the bound~\eqref{eq:Psi-i-upper} one arrives at $\Psi_i\leq \exp(-\gamma C e^{(4\beta-C)(k-1)}|\tilde S_i|/n)$. Taking $c_0 := 5 C$ and plugging in $\sum_{i=1}^N|\tilde S_i|/n \geq c_0 \beta^{-k}n $ implies~\eqref{eq:S-dagger-h=hn-1}. In the application of this lemma towards proving Theorem~\ref{thm:lower-bound}, the exponent on the right of~\eqref{eq:S-dagger-h=hn-1} competes with $(4\beta+C)kn$. For $k\geq 2$, the former is the dominant term regardless of $\gamma$, whereas for $k=1$ it is the dominant term if $5 \gamma/\beta > 4\beta+C$, as in our assumption.
\end{remark}

\subsection{Proof of Proposition~\ref{prop:prob-positive-upper-lower}}\label{sec:pf-prop-positive}
The lower bound in the proposition is precisely the statement of Lemma~\ref{lem:prob-positive}, and it remains to show how to derive the upper bound via Theorem~\ref{thm:lower-bound}. If $\bC$ is the set of ceilings in $\cI$ with $|\cC|\geq n^{1.9}$ and height at least $h_n^*-1$, and $\mathfrak{E}$ is the event that $\sum_{\cC\in\bC}|\cC|\geq A:=(1-2e^{-\beta})n^2$, then Theorem~\ref{thm:lower-bound} together with Lemma~\ref{lem:small-ceil} imply that $\mufloor_n^0(\mathfrak E)=1-o(1)$, implying in particular that \[\mu_n^\mp(\Ifloor[0]) = (1-o(1))\mu_n^\mp(\mathfrak E)\,.\]
We have at most $n^{0.1}$ such ceilings, and if $\fW_1,\ldots,\fW_N$ ($N\leq n^{0.1})$ are the nested sequences of walls  supporting them, then  
\[ \fm\Big(\bigcup_{i=1}^N \fW_i\Big) \geq 4\sqrt{A}  (h_n^*-1) \geq \frac{1-e^{-\beta/2}}{\beta+1/4} n\log n = (1-\epsilon_\beta)\beta^{-1} n\log n\]
since $h_n^* \geq (4\beta+e^{-\beta})^{-1}\log n -C$ for large enough $n$, and
an area of $A$ ceiling faces must be supported by a perimeter of at least $4\sqrt{A}$ vertical wall faces at heights $1,\ldots,h_n^*-1$. 
Using $\sum_{N\leq n^{0.1}}\binom{n^2}N \leq \exp(n^{0.1+o(1)})$, we deduce from Theorem~\ref{cor:nested-sequence-of-multiple-faces-in-a-ceiling} that
\[  \mu_n^\mp(\mathfrak E) \leq \exp\left(-(\beta-C)(1-\epsilon_\beta)\beta^{-1}n\log n\right) \sum_{N\leq n^{0.1}}\binom{n^2}{N} \leq \exp\left(-(1-\epsilon'_\beta)n\log n\right)\,,
\]
and the proof is complete.\qed

\section{Bounding the interface histogram above a given height}\label{sec:upper-bound}

Our goal in this section is to establish the following bound on the fraction of the sites in $\cL_{0,n}$ above/below which the interface is not a singleton at height $0$ whenever $\sfh\geq h_n^*$.

\begin{theorem}\label{thm:upper-gen-h}
Let $U_n$ be the set of $x\in\cL_{0,n}$ such that $\cI\cap(\{x\}\times\R) \neq \{0\}$.
There exist $\beta_0,C>0$ so that for all $\beta>\beta_0$ and every $\sfh\geq h_n^*$, we have
\[ \mufloor_n^{\sfh}\left(|U_{n}|\geq  C e^{-\beta}n^2\right) \leq \exp\left(- e^{-\beta} n\right)\,.
\]
\end{theorem}

This will be a consequence of the following estimate on encountering walls with a near linear excess energy.
\begin{theorem}\label{thm:upper-bound}
There exist $\beta_0,C>0$ so the following holds for all $\beta>\beta_0$. For every $\sfh\geq h_n^*$ and $r\geq n^{9/10}$, 
\[  \mufloor_n^{\sfh}\left(\exists\mbox{ a wall $W$ in $\cI$ with }|\rho(\hull W)|\geq r^2
\right) \leq e^{-(\beta-C)(r \,\wedge\, n\log n)}\,.\]
Furthermore, for any $k\geq 1$,
\[  \mufloor_n^{\sfh}\left(\exists\mbox{ walls $\{W_i\}_{i=1}^k$ in $\cI$ with disjoint hulls and }\min_i |\rho(\hull W_i)|\geq r^2
\right) \leq e^{-(\beta-C)(k r \,\wedge\, n\log n)}\,.\]

\end{theorem}

Theorem~\ref{thm:upper-gen-h} readily follows from Theorem~\ref{thm:upper-bound} and  our results from Section~\ref{sec:rough-bounds-no-floor}; to see this, argue as follows:
Lemma~\ref{lem:small-ceil} guarantees that the total area in $\rho(\hull\cC)$ for ceilings $\cC$ with $|\hull\cC|\leq n^2/\log^3 n$ is at most $Ce^{-\beta}n^2$ except with probability $\exp(-5 n \log n)$. 
To treat the remaining ceilings having $|\hull\cC|>n^2/\log^3 n$, define
\[ \bW_\ell(\cI) = \{\mbox{outermost walls $W$ in $\cI$ with $e^{-\ell} \leq |\rho(\hull W)|/n^2 \leq e^{1-\ell}$}\}\qquad (\ell \in \llb 2\beta, 3\log \log n\rrb)\,.\]
An application of Theorem~\ref{thm:upper-bound} with $k=e^{\ell/2}$ and $r=n e^{-\ell/2}$ shows that $|\bW_\ell(\cI)| \leq e^{\ell/2}$ (reflecting a total area of at most $n^2 e^{1-\ell/2}$) except with probability $\exp(-(\beta-C)n)$. A union bound then shows that, except with probability $\exp(-(\beta-C)n)$, we have $\bigcup_\ell \bigcup_{W\in\bW_\ell} |\rho(\hull W)| \leq Cn^2 e^{-\beta}$.
The probability that a single outermost $W$ exists with $| \rho(\hull W)|\geq e^{-2\beta} n^2 $ (and hence $\fm(W)\geq e^{-\beta}n$) is, by another application of Theorem~\ref{thm:upper-bound}, at most $\exp(-(\beta-C) e^{-\beta} n)$, the dominant term in our final estimate.
Finally, by Lemma~\ref{lem:n2-excess}, the total number of wall faces is at most $e^{-2\beta}n^2$ except with probability $\exp(-c_\beta n^2)$.

It will be illuminating to describe how the analogue of the vanilla SOS approach can yield an estimate on the total area in ceilings at height at least $1$ when $\sfh \geq (1+\epsilon_\beta)h_n^*$
once we combine it with the approximate Domain Markov property for ceilings established in~\cite{GL20}.  More generally, we have the following claim, applicable to any $\sfh\geq 0$ (the above mentioned estimate corresponds to the special case $\sfh = h_1$, which in turn is $(1+\epsilon_\beta)h_n^*$; at the other end, the case $\sfh=0$ gives a bound on $\sum\{|\cC|:\cC\in\fC_{>h_1}(\cI)\}$ for $\cI\in\mufloor_n^0$, a counterpart for Claim~\ref{clm:lower-bound-1-eps}).

\begin{claim}\label{clm:upper-bound-1+eps}
There exist $\beta_0,C_0>0$ such that for every  $h_1 \geq (1+\frac1{\sqrt \beta})\frac1{4\beta-C_0}\log n$ and all $\beta>\beta_0$ and $\sfh\geq 0$,

\[\mufloor_n^\sfh\Big(\sum\{|\cC|:\; \cC\in\fC_{>h_1-\sfh}(\cI)\} \geq (1/\sqrt\beta) n^2\Big) \leq O\left(\exp\left(-3  n \log n\right)\right) \,. \]
\end{claim}
\begin{proof}
For readability, let us prove the hardest case of $\sfh =0$, with the modifications being clear (additive $-\sfh$ everywhere) in the general case. Let $\beta_0,C_0>0$ be given by Proposition~\ref{prop:max-generic}. 
Further set $h_0=\lfloor\frac1{4\beta-C_0}\log n\rfloor$. 
For every interface $\cI$, let $\{\cC_1,\ldots,\cC_N\}$ denote the set of outermost ceilings in $\fC_{>h_1}(\cI)$ with $|\hull\cC|\geq n^{1.9}$ (note that $N\leq n^{0.1}$). Lemma~\ref{lem:small-ceil} allows us to disregard the contribution from ceilings with smaller areas, whereas if  $x_i$ is any face in the wall supporting $\cC_i$ then Lemma~\ref{lem:n-disjoint-walls} (actually its analogue for wall clusters) allows us to preclude any $\cI$ in which
 $\fm(\bigcup_i\Clust(\fW_{x_i}))>(6/\beta) n\log n$.
With this in mind, let $\bI$ be the set of interfaces where $\sum_{i }|\cC_i|\geq (1/\sqrt\beta)n^2$ and in addition $\fm(\bigcup_i \Clust(\fW_{x_i}))\leq 30 n h_0$. 

Reveal the outermost walls of $\cI\in\bI$, and continue the process inductively in supported ceilings, while not revealing any walls nested in the~$\cC_i$'s. 
Applying Proposition~\ref{prop:max-generic} iteratively on $S_i=\rho(\hull\cC_i)$ for each $\cC_i$, we see that for $S = \bigcup_i S_i$ and any realization $\bW$ of the walls outside of $S$ (subject to having $\cI\in\Ifloor[0]$) we have 
\[ \mufloor_n^0(\bar M_S^\downarrow \leq h_0 \mid \bI_\bW) \geq \mu_n^\mp(\bar M^\downarrow_S \leq  h_0 \mid \bI_{\bW})
\geq \exp\big(- |S| (e^{-(4\beta-C_0)(h_0+1)} \big) \geq e^{-|S|/n} \geq e^{-n}\,,\] using that, given the wall set $\bW$, the event $\cI\in\Ifloor$ is implied by $\bar M_S^\downarrow \leq h_1$ (which is implied by $\bar M_S^\downarrow\leq h_0$) in the first inequality. So, if $\bI' $ is the set of $\cI\in\bI$ such that $\bar M_S^\downarrow \leq h_0$ for $S=\bigcup_i \rho(\hull\cC_i)$, then 
$$ \mufloor_n^0(\bI') \geq \exp(-n)\mufloor_n^\sfh(\bI)\,.$$

We now define the following map on interfaces in $\mathbf{I}'$. The map $\Phi$ takes an interface $\cI \in \mathbf{I}'$, and 
\begin{enumerate}
    \item deletes (from the wall representation) every wall in $\bigcup_i \Clust(\mathfrak{W}_{x_i})$;
    \item adds back for every outermost ceiling $\cC$ supported by a wall $W\in \bigcup_i \Clust(\mathfrak{W}_{x_i})$, the ``simplified wall" faces $\tilde W_\cC$ which consists of vertical faces along $\partial \hull \cC$ at heights stretching from the height of the floor of $W$ to $\hgt(\cC)$, \emph{except} if $\cC\in \{\cC_1,\ldots,\cC_N\}$ in which case $\tilde W_\cC$ caps off at height $h_0$ (and if $h_0$ is lower than the height of the floor of $W$, say $\tilde W_\cC$ is empty).   
\end{enumerate}
This is a valid map on $\mathbf{I}'$ because all ceilings and walls not in $\bigcup_{i}\Clust (\mathfrak{W}_{x_i})$ are placed back at the original height they were at, except those nested in $\{\mathcal C_1,\ldots, \mathcal C_N\}$ which get shifted to be based at height $h_0$. But in that case, the fact that $\bar M_S^\downarrow \le h_0$ ensures the positivity constraint is respected. 

We now consider the weight change and entropy lost from this operation. The ceilings $\{\mathcal C_i\}_{i=1}^{N}$ accounted for at least $(\beta^{-1/2} n^2)^{1/2}$ vertical wall faces in each of the slabs $h_0+1,\ldots,h_1$ in $\cI$. The interface $\Phi(\cI)$ has at least this many fewer faces than $\cI$ due to the truncation of $\tilde W_\cC$ at height $h_0$ for each of these ceilings. In the other walls that are modified, there is no gain of faces, because every $W$ modified must have at least $\sum  |\partial \hull \cC|\cdot (\hgt(\cC) - \hgt(\text{floor}(W))$ many faces, where the sum is over ceilings $\cC$ it supports. In particular, 
in the cluster expansion (namely, Theorem~\ref{thm:cluster-expansion}), we have 
$$ \beta\fm(\cI;\Phi(\cI))\geq \beta^{3/4}n (h_1-h_0)   \geq \beta^{1/4} n h_0\,. $$
The contribution from the interaction terms $\g$ in that theorem are at most $\bar K |\bigcup_i \Clust(\fW_{x_i})| \leq 30\bar K n h_0$ (using the definition of wall clusters, as done in the rigidity proof using wall clusters). The number of interfaces in $\mathbf{I}'$ that get mapped to the same $\mathbf{I}$ can be upper bounded by $\binom{n^2}{N}\le e^{O(n^{0.1}\log n)}$ for locating the root points $x_{i}$ and then $C^{30 n h_0}$ for enumerating over the realization of the faces in $\bigcup_{i}\Clust(\mathfrak{W}_{x_i})$. 
Altogether, if $\beta_0 $ is large enough then $$\mufloor_n^0(\bI)\leq \exp(n) \mufloor_n^0(\bI') \leq \exp(-(1-\epsilon_\beta)\beta^{1/4}n h_0)\,,$$ as desired.
 \end{proof}

\begin{remark}
We now explain why one cannot hope for the proof of Claim~\ref{clm:upper-bound-1+eps} to hold all the way to $h_1=(1+o(1))h_n^*$, while pointing out a subtle but important difference from the Peierls argument used in the SOS model. Whereas in the SOS model (being a distribution over height functions) one can always lower a level line by $1$ via decreasing the heights in its interior, the analog of a level line in the Ising model---a ceiling $\cC_0$ supported by a wall $W_0$---might not be (in fact often will not be) consistent with such an operation. 
Indeed, even if we suppose~$\cC_0$ is the only ceiling supported by $W_0$ (to simplify matters), it may be (and in fact often will be the case) that the ceiling $\cC_0$ is a part of a larger connected set $F\subset \cI$ of horizontal faces at height $\hgt(\cC_0)$, which are categorized as wall (rather than ceiling) faces due to the shape of the interface far (say, at distance $\epsilon \log n$) below them. It is the set $F$ that one would want to trim so as to gain $|\partial \cC|$ in energy (since, locally around $\hgt(\cC_0)$, that set is the analogue of an SOS level line). However, such an operation might shift the thermal fluctuations of the interface above $\rho(F\setminus \cC_0)$ and have them clash with other wall faces. Further complicating matters is the fact that the wall $W_0$ might be tilted, whence its horizontal fluctuations (on account of which faces in $F\setminus\cC_0$ are not ceiling faces) are not well understood. To bypass these issues, the proof of Claim~\ref{clm:upper-bound-1+eps} ``straightened'' the walls supporting the ceilings under consideration into cylinders; e.g., in the case described above with a single $\cC_0$, the energetic cost of modifying $W_0$ in $k\sim h_n^*$ slabs would have order at least $k |\partial \cC_0|$. To offset this cost and rule out a ceiling $\cC_0$ with $|\partial\cC_0|\asymp n$, one would need to gain at least $(1/\beta) k n$ deleted walls from the trimming operation; with $\partial \cC_0$ contributing  $O(n)$ such wall faces in every slab being shrunk, this argument would need a leeway of $h_1\geq (1 + c/\beta) k \geq (1+\epsilon_\beta) h_n^*$ slabs.
\end{remark}

Whereas these obstacles are highly nontrivial, in the setting of Theorem~\ref{thm:upper-bound} there is a single wall $W_0$ whose entire wall cluster may be successfully deleted in the regime $\sfh\geq h_n^*$---with the caveat that it will require the extension given in~\S\ref{sec:extension-estimates-in-ceiling} of the bounds on the maximum within a ceiling obtained in~\cite{GL20}.

\subsection{Proof of Theorem~\ref{thm:upper-bound}}
Let us first consider the situation of a single wall $W_0$; the case of multiple walls $W_0^{(1)},\ldots,W_0^{(\ell)}$ will be obtained by  iteratively applying this argument.

We will show the following stronger statement, applicable to any $\sfh\geq 0$.
Recalling that the supporting ceiling of a wall $W_0$ is the ceiling whose faces are adjacent to $\hull W_0$ (which for outermost walls necessarily belongs to $\fC_0$), and letting 
\[ \fI_r := \left\{\cI\,:\;\mbox{$\exists$ a wall $W_0$ in $\cI$ supported by  $\widehat{\cC}_0\in\fC_{\geq h_n^*-\sfh}(\cI)$
such that $\fm(W_0)\geq r$}\right\}\,,\]
we will argue that
\[  \mufloor_n^{\sfh}\left(\fI_r\right) \leq e^{-(\beta-C)(r \,\wedge\, n\log n)}\qquad\mbox{for every $r\geq n^{9/10}$}\,.\]
In the case $\sfh \geq h_n^*$, this will immediately imply the required result, as every $W$ with $\fm(\rho(\hull W))\geq r^2$ is contained in some outermost wall $W_0$ with $\fm(W_0) \geq |\rho(\hull W_0)|^{1/2} \geq |\rho(\hull W)|^{1/2} \geq r$, for which the above estimate holds.

Consider a standard wall $W_0$, as well as a realization of its complete wall cluster $\bF_0 = \mathsf{Clust}(W_0)$, such that $\fm(\bF_0)=r$ for some $r \geq n^{9/10}$.
Further let
\[ \fI_{\bF_0} := \left\{ \cI \,:\; 
\begin{array}{l}
W_0\in\mbox{standard wall collection of $\cI$, its wall cluster is $\bF_0$}\\
\mbox{and it is supported by a ceiling $\widehat\cC_0\in\fC_{\geq h_n^*-\sfh}(\cI)$}
\end{array}
\right\}\,,\]
emphasizing that we did not restrict $\cI\in\Ifloor$.
It will suffice to show for every $\bF_0$ where $\mufloor_n^\sfh(\fI_{\bF_0})>0$ that 
\begin{equation}\label{eq:fI-W0-bound}
\mufloor_n^{\sfh}(\fI_{\bF_0}) \leq e^{-(\beta -C) (r \,\wedge\, n\log n)}\,,
\end{equation}
since there are $O(n^2)=e^{o(r)}$ locations for the placement of $W_0$, and at most $e^{\bar c r}$ wall clusters $\bF$ with $\fm(\bF) \leq r$ for some $\bar c>0$ (independent of $\beta$), whence a union bound over~\eqref{eq:fI-W0-bound} will conclude the proof.

 If $r\geq n\log n$ then this is readily implied by~\cite[Lemma 3.10]{GL20}, as $\mu_n^\mp(\fI_{\bF_0^x}) \leq \exp(-(\beta-C)\fm(\bF_0^x))$ which is then at most $\exp(-(\beta-C)n\log n)$, hence extends to $\mufloor_n^{\sfh}$ via Lemma~\ref{lem:prob-positive}. Assume therefore that $r\leq n\log n$.

Further assume---this time, an assumption that will require extra justification---that $W_0$ satisfies 
\begin{equation}\label{eq:def-good-wall}
 \sum\left\{|\cC|\,:\;\cC\mbox{ is a ceiling of $\cI_{W_0}$ with $\cC\Subset W_0$, $\isodim(\cC)>4$}\right\}\leq n^{7/4}\,.
\end{equation}
(We will later show that, except with probability $O(\exp(-3n\log n))$, every wall $W_z\in\cI$ satisfies this.)

Let $\widehat\cC_0$ denote the ceiling that supports $W_0$, and let 
\[\bC_0 = \left \{\cC \mbox{ is a ceiling of $\cJ = \cI_{\bF_0}$ with $\cC\Subset W_0$ and $\hgt(\cC) \ge h_n^* - \sfh$}\right\} \quad \text{and} \quad S_0  = \bigcup_{\cC\in \mathbf{C}_0} \rho(\cC)\]
(Note that $\bC_0$ consists not of ceilings of $\cI_{W_0}$ but rather those of the interface $\cJ$ comprising the wall cluster~$\bF_0$. Whereas the former are simply-connected, a ceiling $\cC\in\bC_0$ might not be, due to nested walls in $\mathsf{Clust}(W_0)$.) Define $\bar S_0 = S_0 \cup \rho(\bF_0)$ to also include the supporting walls' projections.
Further let
\[ \bar\fI_{\bF_0} := \left\{ \cI\in\fI_{\bF_0}\,:\; \bar M^\downarrow_{S_0}\leq h_n^* 
\right\}\,.
\]
Let $\bV_0 = \{W_z: z\notin \bar S_0\}$ be an arbitrary realization of the set of walls of $\bar S_0^c$ compatible with $\bF_0$ and such that $\cI\cap (\bar S_0^c \times \mathbb R) \subset \cL_{\ge -\sfh}$ and generating ceiling $\widehat C_0$ supporting $W_0$ at height $\hgt(\widehat \cC_0)\ge h_n^* -\sfh$. 

The key to the proof of~\eqref{eq:fI-W0-bound} (and in turn, the entire theorem) is establishing that for any such $\bV_0$,
\begin{equation}
    \label{eq:barI-conditional-lower-bound-floor}
    \mufloor_n^{\sfh}\left(\bar\fI_{\bF_0}\mid \fI_{\bF_0}\cap\bI_{\bV_0}\right) \geq \exp(-C  r)\quad\mbox{for an absolute constant $C>0$.}
\end{equation}
We now observe that for any interface $\cI \in \fI_{\bF_0} \cap \bI_{\bV_0}$, the portions exterior to $ S_0$ are entirely above height $-\sfh$. Therefore, on these events, the event $\bar\fI_{\bF_0}$ implies $\Ifloor$. With that, we can write 
\begin{align*}
 \mufloor_n^\sfh\left(\bar\fI_{\bF_0}\mid \fI_{\bF_0}\cap\bI_{\bV_0}\right) = \frac{\mu_n^\sfh  \left(\bar\fI_{\bF_0} \cap \Ifloor \mid \fI_{\bF_0}\cap\bI_{\bV_0}\right)}{\mu_n^\sfh(\Ifloor\mid \fI_{\bF_0}\cap\bI_{\bV_0})}  =  \frac{\mu_n^\sfh  \left(\bar\fI_{\bF_0} \mid \fI_{\bF_0}\cap\bI_{\bV_0}\right)}{\mu_n^\sfh(\Ifloor\mid \fI_{\bF_0}\cap\bI_{\bV_0})} \ge \mu_n^\sfh  \left(\bar\fI_{\bF_0} \mid \fI_{\bF_0}\cap\bI_{\bV_0}\right)\,.
\end{align*}
Therefore, in order to establish~\eqref{eq:barI-conditional-lower-bound-floor}, it is sufficient to establish 
\begin{equation}\label{eq:barI-conditional-lower-bound}
        \mu_n^{\mp}\left(\bar\fI_{\bF_0}\mid \fI_{\bF_0}\cap\bI_{\bV_0}\right) \geq \exp(-C  r)\quad\mbox{for an absolute constant $C>0$.}
\end{equation}
It is important to stress the effect of conditioning on $\fI_{\bF_0}\cap \bI_{\bV_0}$. Conditioning on $\bI_{\bV_0}$---namely on the walls $\bV_0$ external to $W_0$ or internal to ceilings based at heights below $h_n^* - \sfh$---will not introduce significant complication as far as the extremal behavior of $\cI$ within $\cC\in\bC_0$ is concerned, thanks to our results~\cite{GL20} on 
the rigidity and approximate Domain Markov property for ceilings. However, the conditioning on the realization of the wall cluster, $\fI_{\bF_0}$ does need to be handled with care, as it is equivalent to conditioning on
\[\fm(W_z) \leq d_\rho(z,\partial W_0)
 \qquad\mbox{for every $z\in S_0$}\,,
\]
on top having the walls in $\bF_0$ belong to the interface; in other words, in terms of the event $\cE^\eta_S$ from~\eqref{eq:def-E-eta},
\begin{equation}\label{eq:F0x-conditioning}
\mu_n^{\mp}(\cdot\mid \fI_{\bF_0}\cap\bI_{\bV_0}) = \mu_n^{\mp}(\cdot \mid \cE^\eta_{S_0}\,, \bI_{\bF_0}\cap\bI_{\bV_0})\quad\mbox{ with $\eta_z=d_\rho(z,\partial W_0)$ for all $z\in S_0$}\,.
\end{equation}

Before analyzing $\cI$ in this delicate conditional space towards~\eqref{eq:barI-conditional-lower-bound}, let us show how to infer~\eqref{eq:fI-W0-bound} from~it.
Let $\Phi$ be the map on interfaces $\cI\in\bar\fI_{\bF_0}$ that deletes all of $\bF_0$ from the standard wall representation of~$\cI$. Clearly, $\Phi$ is a bijection ($\bF_0$ being fixed), and the crux of the definition of $\bar\fI_{\bF_0}$ is that $\cI':=\Phi(\cI) \in \Ifloor$ holds for all $\cI\in \bar\fI_{\bF_0}$ (in fact even if we had $\cI$ intersect $S_0\cap\cL_{<-\sfh}$). Indeed, 
$\cI$ does not intersect $\rho(\smhull W_0)^c\times \cL_{<-\sfh}$, and having deleted $\bF_0$ from it to reach $\cI'$, clearly $\cI'$ does not intersect $(S_0)^c\times\cL_{<-\sfh}$, whereas in $S_0\times \R$, since we have $\hgt(\widehat\cC_0)\geq h_n^*-\sfh$ (by the definition of $\bV_0$), it has a minimum height of at least $ h_n^* -\sfh-\bar M^\downarrow_{S_0}(\cI) \geq -\sfh$ by the definition of $\bar\fI_{\bF_0}$. 
Recalling that $\fm(\bF_0)=r$, we infer from~\cite[Lemma 3.10]{GL20} that \begin{align*}
    \mufloor_n^{\sfh}(\bar\fI_{\bF_0} \mid \bI_{\bV_0}) &= 
    \sum_{\cI' : \Phi^{-1}(\cI')\in \bar\fI_{\bF_0}}\mufloor_n^{\sfh}(\Phi^{-1}(\cI') \mid \bI_{\bV_0}) \leq  \sum_{\cI':\Phi^{-1}(\cI')\in\bar\fI_{\bF_0}} e^{-(\beta-C)r}\mufloor_n^{\sfh}( \cI' \mid\bI_{\bV_0})\leq e^{-(\beta-C)r}\,.
\end{align*} 
(Here we only conditioned on $\bI_{\bV_0}$ rather than on $\fI_{\bF_0}\cap\bI_{\bV_0}$.) Via the key inequality~\eqref{eq:barI-conditional-lower-bound-floor}, implied by~\eqref{eq:barI-conditional-lower-bound} which we will establish, this implies
\begin{equation}\label{eq:I_F0-upper-bound} \mufloor_n^{\sfh}(\fI_{\bF_0}\mid \bI_{\bV_0}) = \frac{\mufloor_n^{\sfh}(\bar \fI_{\bF_0} \mid \bI_{\bV_0})}
{\mufloor_n^{\sfh}(\bar \fI_{\bF_0} \mid \fI_{\bF_0} \cap \bI_{\bV_0})} \leq e^{-(\beta - C') r}\,,
\end{equation}
which, given that this holds for every $\bV_0$ compatible with interfaces in $\fI_{\bF_0}$, establishes~\eqref{eq:fI-W0-bound}. 

To prove~\eqref{eq:barI-conditional-lower-bound}, let $\cP\subset\cL_{0,n}$ be a minimum set of faces that makes $\{\rho(W):W\in\bF_0\} \cup \rho(W_0)$ connected (we read $\cP$ deterministically from $\bF_0$ via an arbitrarily predefined tie breaking between sets of equal size). As we next observe, the criterion for a wall $W$ to be part of $\mathsf{Clust}(W_0)$ readily implies that 
\[|\cP|\leq \fm(\bF_0) = r\,;\] 
indeed, initializing $\hat\bF = \{W_0\}$ and $\hat\cP=\emptyset$, consider the process of adding to $\hat\bF$ walls $W\in \bF_0 \setminus \hat\bF$ one at a time, in tandem with adding to $\hat\cP$ the shortest path between $\rho(W)$ and $\rho(\hat\bF)$, until arriving at $\hat\bF=\bF_0$ and $\hat\cP$ which connects every $\{\rho(W):W\in\hat\bF\}$ to $\rho(W_0)$.
Every $W\in \bF_0\setminus \hat\bF$ must satisfy $d_\rho(W,\hat\bF)\leq \fm(W)$
by the definition of $\mathsf{Clust}(W_0)$, hence incurs an addition of at most $\fm(W)$ faces to $\hat\cP$. The process terminates with $|\hat\cP|\leq\fm(\bF_0)$, hence the same bound applies to the minimum face set $\cP$ connecting these walls.

Ordering the faces of $\cP$ as $z_1,z_2,\ldots$, we proceed to expose $\fG_{z_i,S_0}$ for $i=1,2,\ldots$ conditionally on $\fI_{\bF_0}\cap\bI_{\bV_0}$ and the walls already revealed along this process for $j<i$, denoting the corresponding filtration by~$\cF_i$.
Recalling~\eqref{eq:G-hat-event} from the proof of Proposition~\ref{prop:max-wall-cluster}, we again define
\[ \widehat G_x=  \bigcap_{u\in \rho(\smhull\fG_{x,S_0})} \{\fm(\fW_{u,S_0}) < 4h_n^*\}\,,\]
and note that by the exact same argument that followed~\eqref{eq:G-hat-event}, now with $\eta_{z}=d_\rho(z,\partial W_0)$, we have
\[ \mu_n^\mp(\widehat G_{z_i} \mid\cE^\eta_{S_0}\,,\bI_{\bF_0}\cap\bI_{\bV_0}\,, \cF_i) \geq 1-e^{-(4\beta-C)h_n^*}\,.\]
The result of iterating this bound for all $z_i\in\cP$, followed by the bound $1-x\geq e^{-\frac32 x}$ for $x<\frac12$, is that for
\[ \cG^{\fm}_\cP(4h_n^*) = \bigcap_{z\in\cP} \widehat{G}_z\]
we have, in light of~\eqref{eq:F0x-conditioning}, that 
\begin{align}\label{eq:bar-M-along-P}
     \mu_n^{\mp}\left(\cG^{\fm}_\cP(4h_n^*)\mid \fI_{\bF_0}\cap \bI_{\bV_0}\right) &\geq \left(1-e^{-(4\beta-C')h_n^*}\right)^{|\cP|} \geq \exp\left(-|\cP|e^{-(4\beta-C'')h_n^*}\right) \geq \exp\left(- r n^{-1+\epsilon_\beta}\right)\,,
\end{align}
where we used that $ h_n^* \geq (4\beta + e^{-4\beta})^{-1} \log n - C$ (see~\cite[Corollary 5.2]{GL19b}). 
Letting $\bW_\cP = \bigcup\{\hull\fG_z : z\in\cP\}$, note that the interface cannot have $\bar M_{\rho(\cC)}^\downarrow \geq h_n^*$ at $\{z\}\times \R$ for $z\in\rho(\cC)$ without also having $\fm(\fW_{z,\rho(\cC)})\geq 4h_n^*$. Conditional on  $\cG^{\fm}_\cP(4h_n^*)$, we next look at events of the form $\{\bar M_S^\downarrow \leq h_n^*\}$ for $S = \rho(\cC)\setminus \rho(\bW_{\cP})$ with $\cC\in\bC_0$.

Partition $\cC\in\bC_0$ into two subsets $\bC'_0,\bC''_0$ as follows:
\begin{equation*}
\bC'_0 := \{\cC\in\bC_0 :\mbox{$|\cC|\geq n^{7/4}$ and $\isodim(\rho(\cC))\leq 8$}\}\quad,\quad\bC_0'':=\bC_0\setminus \bC_0'\,,
\end{equation*}
consider $S=\rho(\cC)\setminus \rho(\bW_\cP)$ for $\cC\in\bC_0$,
and let $\bW$ be an arbitrary realization of the walls of $S^c$ under $\mufloor_n^\sfh$ that is compatible with $\bI_{\bF_0}\cap \bI_{\bV_0}$, $\cG^{\fm}_\cP(4h_n^*)$ and $\cE^\eta_{\bC_0}$. 
Note that, by the minimality of $\cP$, the set $S$ is simply-connected. 
If $\cC\in\bC_0'$ then $|S|\geq |\cC|-|\cP|\geq (1-o(1))|\cC|$ (recalling that $|\cP|\leq r\leq n\log n$ and $|\cC|\geq n^{7/4}$), whereas $|\partial S| \leq |\partial \cC|+|\cP|\leq (1+o(1))|\cC|^{7/8}$ (as $|\partial \cC|\leq |\cC|^{7/8}$ due to $\isodim(\cC)\leq 8$, while $r\leq n\log n$ and $|\cC|^{7/8}> n^{3/2}$). In particular, $\isodim(S)\leq 8+o(1)$ in this case.
By Proposition~\ref{prop:max-wall-cluster}, we may thus apply~\eqref{eq:max-thick-lower-large-h-wallcluster}, deducing that
\begin{align}\label{eq:bar-M-high-thick-ceilings}
 \mu_n^\mp(\bar M^\downarrow_S \leq h_n^* \mid
    \bI_\bW\,,\, \cE^\eta_S) 
    \geq \exp\left(-(1+\epsilon_\beta)|S| e^{-\alpha_{h_n^*+1}}\right)
    \geq \exp\left(-(1+\epsilon_\beta)|\cC| e^{-\alpha_{h_n^*+1}}\right)
    \,,
\end{align}
noting that we meet the assumptions for every $n$ large enough there since  $h_n^* \geq (4\beta+\epsilon_\beta)^{-1}\log n  \geq \sqrt{\log |S|}$, and  for every $z$ that satisfies $d(z,\partial S)\geq e^{2\kappa_0 h_n^*}$ our constraint in $\eta$ is set to $\eta_z = d(z,\partial W_0) \geq e^{2\kappa_0 h_n^*} > 5h_n^*$.
For the remaining ceilings of $\bC_0$, namely $\bC\in\bC_0''$, we may appeal to~\eqref{eq:max-lower-crude-wallcluster} and derive 
\begin{align}\label{eq:bar-M-high-thin-ceilings} \mu_n^\mp(\bar M^\downarrow_S \leq h_n^* \mid \bI_\bW\,,\, \cE^\eta_S) \geq \exp\left(-|S| e^{-(4\beta-C)h_n^*}\right)
\geq \exp\left(-|\cC| e^{-(4\beta-C)h_n^*}\right)\,.
\end{align}
Combining~\eqref{eq:bar-M-along-P}--\eqref{eq:bar-M-high-thin-ceilings}, we conclude that
\begin{equation}\label{eq:bar-M-total} \mu_n^{\mp}(\bar\fI_{\bF_0} \mid \fI_{\bF_0}\cap\bI_{\bV_0}) \geq \exp\bigg(-n^{-1+\epsilon_\beta}r-(1+\epsilon_\beta)\sum_{\cC\in\bC_0'}|\cC| e^{-\alpha_{h_n^*+1}} - \sum_{\cC\in\bC_0''}|\cC| e^{-(4\beta-C)h_n^*}\bigg)\,,\end{equation}
and it remains to account for the total area of the ceilings in $\bC_0'$ and $\bC_0''$. We will use the following straightforward deterministic fact.
\begin{fact}
\label{fact:number-of-small-ceilings} Let $a,r$ be integers, let $\bW$ be a set of walls with a total of $r$ faces, and suppose there are $m$ distinct ceilings $\{\cC_i\}_{i=1}^m$ in $\cI_{\bW}$ such that $\rho(\cC_i)\subset \bigcup_{W\in\bW}\rho(\hull W)$ and $|\cC_i|\leq a$ for all $i$. Then $\sum_{i=1}^m|\cC_i|\leq 3 r\sqrt{a}$.
\end{fact}
\begin{proof}
As every $\cC_i$ is bounded by wall faces in $\bW$, and each face of $W\in\bW$ bounds at most 2 distinct ceiling faces, we have $\sum_{i=1}^m |\partial \cC_i|\leq 2r$. So, via the isoperimetric inequality $|\cC_i| \leq |\partial\cC_i|^2 $ (with a factor $4$ to spare),
\[ 
\sum_{i=1}^m |\cC_i| \leq \sqrt a \sum_{i : |\partial\cC_i|\leq\sqrt a}|\partial \cC_i| + \left|\left\{i\,:\;|\partial\cC_i|\geq\sqrt a \right\}\right|a \leq  \sqrt{a}r + \frac{2r}{\sqrt{a}}a = 3 r\sqrt{a}\,.
\qedhere
\]
\end{proof}
For $\bC'_0$ we use (along the same vein as the proof of the fact above) that $|\cC|\leq |\partial\cC|n/4$ and 
$\sum |\partial\cC|\leq 2r$ to deduce that 
\[\sum_{\cC\in\bC_0'} |\cC| \leq n r/2\,.\] 
For $\bC_0''$, recall that $\cC_0,\ldots,\cC_L$ are the ceilings of $\cI_{W_0}$ nested in $W_0$ (as opposed to $\bC_0$, which  takes $\mathsf{Clust}(W_0)$ into consideration).
For every $\cC\in\bC_0''$ there is a unique $i\geq 1$ such that $\rho(\cC)\subset \rho(\cC_i)$, so 
\[ \sum_{\cC\in\bC_0''} |\cC|\one_{\{\exists i \,:\; \rho(\cC)\subset\rho(\cC_i)\,,\,\isodim(\cC_i)>4\}}
\leq
\sum_{i} |\cC_i|\one_{\{\isodim(\cC_i)>4\}}\leq n^{7/4}\]
by construction of $W_0$, whereas Fact~\ref{fact:number-of-small-ceilings} (and having $\fm(\bF_0)=r\leq n\log n$) shows that
\[ \sum_{\cC\in\bC_0''}|\cC|\one_{\{|\cC|\leq n^{7/4}\}} = O(n^{15/8}\log n)\,.\]
Finally, if $|\cC|> n^{7/4}$ yet $\rho(\cC)\subset\rho(\cC_i)$ for some $i$ such that $\isodim(\rho(\cC_i))\leq 4$, then $|\partial \cC|\leq |\partial\cC_i|+\fm(\bF_0)$ which is at most $|\cC_i|^{3/4}+r \leq (1+o(1))n^{3/2}$, so $\isodim(\rho(\cC))\leq 7+o(1)$. Therefore, $\cC\in\bC_0'$ for large $n$.

Substituting these bounds in~\eqref{eq:bar-M-total}, along with the facts
$\alpha_{h_n^*+1}\geq \alpha_{h_n^*}+4\beta-C\geq \log n +2\beta-C$ and $h_n^* \geq (4\beta+e^{-4\beta})^{-1}\log n - C$ from~\cite[Corollary~5.2]{GL19b}, it now follows that, for some other absolute $C>0$,
\begin{equation}\label{eq:I_F0-lower-bound}
\mu_n^{\mp}(\bar\fI_{\bF_0} \mid \fI_{\bF_0}\cap\bI_{\bV_0}) \geq \exp\left(-e^{-2\beta+C}r - n^{7/8+\epsilon_\beta}\right)  \geq \exp\left(-C'e^{-2\beta}r\right)\,,
\end{equation}
where the term $n^{-1+\epsilon_\beta}r$ was absorbed into the $e^{-2\beta +C}r$ term, and the final inequality used that $r\geq n^{9/10}.$
This establishes~\eqref{eq:barI-conditional-lower-bound}, thereby concluding the proof for all $W_0$ satisfying~\eqref{eq:def-good-wall}.

It remains to justify the assumption~\eqref{eq:def-good-wall} for $W_0$; recalling that we have $\fm(W_0)\leq n\log n$, this property will be obtained as a consequence of our results from Section~\ref{sec:rough-bounds-no-floor}, as we have the following:
\begin{enumerate}[(a)]
    \item if $\{\cC_i\}$ is the set of ceilings of $\cI\sim\mufloor_n^\sfh$ with $|\hull\cC_i|\geq n$ and $\isodim(\hull\cC_i)\geq 4$, then $| \bigcup_i\rho(\hull\cC_i)| < n^{5/3}$ 
    with probability at least $1-\exp(-(\beta-C)n^{7/6})$ by Lemma~\ref{lem:thin-ceil}; 
    \item if $W_0$ is a wall with $\fm(W_0)\leq n\log n$, and $\{\cC_i\}$ is the set of ceilings of $\cI_{W_0}$ that are nested in $W_0$ and satisfy $|\cC_i|\leq n$, then Fact~\ref{fact:number-of-small-ceilings} shows that, deterministically, $\sum_i |\cC_i| = O(n^{3/2}\log n)$. 
\end{enumerate}
Combining these (along with a union bound on $W_0$ and its location) we see that $W_0$ satisfies~\eqref{eq:def-good-wall} (we arrived at an upper bound of $n^{5/3}+n^{3/2+o(1)} < n^{7/4}$) except with probability $O(\exp(-n^{7/6-o(1)}))$, as claimed.

To obtain the result for a collection of walls $W_0^{(1)},\ldots,W_0^{(\ell)}$, observe that our bound on the probability of finding a given such $W_0$ was conditional on $\bV_0$, an arbitrary realization of walls in $\rho(\hull W_0)^c$. Therefore, for any fixed collection of such walls, we may iteratively bound the probability of $W_0^{(i)}$ conditional on the occurrence of $\{W_0^{(j)}\}_{j<i}$, provided that the hulls of said walls are disjoint (guaranteed by our hypothesis), and obtain an overall bound of $\exp(-(\beta-C)(\ell r \wedge n\log n))$. A union bound over the location and structure of these walls thus concludes the proof.
\qed

\begin{remark}
\label{rem:h=hn-1-upper}
One can extend Theorem~\ref{thm:upper-bound}---and in turn, Theorem~\ref{thm:upper-gen-h}---to the case of $\sfh=h_n^*-1$, in the same vein as Remark~\ref{rem:h=hn-1-lower}, provided that the quantity $\gamma = n \exp(-\alpha_{h_n^*})$ (which satisfies $
e^{-2\beta - \epsilon_\beta} \leq \gamma <  e^{2\beta}$) is in a given range, e.g., 
\[ \gamma \leq \beta\,.\]
We claim that~\eqref{eq:fI-W0-bound} holds for $\fI_{\bF_0}$ defined w.r.t.\ a supporting ceiling $\widehat \cC_0\in\fC_{\geq 0}$. We accordingly modify $\bar\fI_{\bF_0}$ to feature a strict inequality:
$\bar\fI_{\bF_0}=\{\cI\in\fI_{\bF_0}: \bar M^\downarrow_{S_0} < h_n^*\}$, and claim that the arguments hold true thereafter. Indeed, since $\hgt(\widehat\cC_0)\geq 0$, having $\bar M^\downarrow_{S_0}\leq h_n^* - 1 = \sfh$ would guarantee that $\Phi(\cI)\in \Ifloor$ for every $\cI\in\fI_{\bF_0}$, whence~\eqref{eq:I_F0-upper-bound} remains true. 
We conclude that~\eqref{eq:bar-M-total} holds for this case when replacing $\alpha_{h_n^*+1}$ by $\alpha_{h_n^*}$, and so~\eqref{eq:I_F0-lower-bound} holds true if the term $e^{-2\beta+C}r$ is to be replaced by $(1+\epsilon_\beta) \gamma r/2$. Our assumption $\gamma \leq \beta$ thus supports the analogue of~\eqref{eq:barI-conditional-lower-bound} with a lower bound of $\exp(-(1+\epsilon_\beta) (\beta/2) r)$, which is outweighted by the term $\exp(-(\beta-C)r)$ from~\eqref{eq:I_F0-upper-bound} to give a final probability estimate of $\exp(-(\beta/2-C)r \wedge n\log n)$.
\end{remark}

\subsection*{Acknowledgements}
The authors thank Joseph Chen for helpful comments.
R.G.\ thanks the Miller Institute for Basic Research in Science for its support. The research of E.L.\ was supported in part by NSF grants DMS-1812095 and DMS-2054833. 

\bibliographystyle{abbrv}

\bibliography{references}

\end{document}